\documentclass[12pt,a4paper,reqno]{amsart}
\usepackage{amssymb,amsmath,amsthm}
\usepackage{xcolor}
\usepackage{bbm}
\usepackage{bigints}
\usepackage{mathtools} 
\usepackage{mathrsfs} 
\usepackage{hyperref} 
\usepackage{bigints}
\usepackage{csquotes} 

\newcommand\dela[1]{}

\numberwithin{equation}{section}
\input colordvi

\newcommand{\h}{\Upsilon}
\newcommand{\rkhs}{H_{\mu}}
\newcommand{\rkhsemb}{E}
\newcommand{\rT}{R}
\newcommand{\trm}{\textrm}

\allowdisplaybreaks

\advance\textwidth by 2.5cm 
\advance\oddsidemargin by -1.6cm
\advance\evensidemargin by -1.6cm

\newtheorem{Theorem}{Theorem}[section]
\newtheorem{Definition}[Theorem]{Definition}
\newtheorem{Proposition}[Theorem]{Proposition}
\newtheorem{Lemma}[Theorem]{Lemma}
\newtheorem{Corollary}[Theorem]{Corollary}
\newtheorem{Remark}[Theorem]{Remark}

\newcommand\restr[2]{{
		\left.\kern-\nulldelimiterspace 
		#1 
		\vphantom{\big|} 
		\right|_{#2} 
}}

\makeatletter
\def\@tocline#1#2#3#4#5#6#7{\relax
	\ifnum #1>\c@tocdepth 
	\else
	\par \addpenalty\@secpenalty\addvspace{#2}%
	\begingroup \hyphenpenalty\@M
	\@ifempty{#4}{%
		\@tempdima\csname r@tocindent\number#1\endcsname\relax
	}{%
		\@tempdima#4\relax
	}%
	\parindent\z@ \leftskip#3\relax \advance\leftskip\@tempdima\relax
	\rightskip\@pnumwidth plus4em \parfillskip-\@pnumwidth
	#5\leavevmode\hskip-\@tempdima
	\ifcase #1
	\or\or \hskip 1em \or \hskip 2em \else \hskip 3em \fi%
	#6\nobreak\relax
	\dotfill\hbox to\@pnumwidth{\@tocpagenum{#7}}\par
	\nobreak
	\endgroup
	\fi}
\makeatother

\begin{document}

\bibliographystyle{plain}
\pagenumbering{arabic}

\title[LDP for SGWE]{Large Deviations  for $(1+1)$-dimensional Stochastic Geometric Wave Equation}
\author[ZB, BG, MO and NR]{Zdzis{\l}aw Brze\'zniak, Ben Go{\l}dys, Martin Ondrej\'at and Nimit Rana}
\address{Department of Mathematics \\
The University of York \\
Heslington, York, YO105DD, UK} \email{zdzislaw.brzezniak@york.ac.uk}
\address{Department of Mathematics \\
	The University of Sydney \\
	School of Mathematics and Statistics, Carslaw Building, NSW 2006} \email{beniamin.goldys@sydney.edu.au}

\address{The Czech Academy of Sciences \\ Institute of Information Theory and Automation \\ Pod Vod\'arenskou v\v e\v z\'{\i} 4 \\ 182 08 Prague 8\\ Czech Republic} \email{ondrejat@utia.cas.cz}

\address{Fakult\"at f\"ur Mathematik \\
	Universit\"at Bielefeld\\
	Postfach 10 01 31, 33501 Bielefeld, Germany} \email{nrana@math.uni-bielefeld.de}

\today

\addtocounter{footnote}{-1} \vskip 1 true cm

\maketitle

\begin{abstract}
We consider stochastic wave map equation on real line with solutions taking values in a $d$-dimensional compact Riemannian manifold. We show first that this equation has unique, global, strong in PDE sense, solution in local Sobolev spaces. The main result of the paper is a proof of the Large Deviations Principle for solutions in the case of vanishing noise.
\end{abstract}

\noindent
\keywords{{\it Key words and phrases:} Large deviations, stochastic geometric wave equation, Riemannian manifold, infinite dimensional Brownian motion.}\\
\subjclass{{\it 2000 Mathematics Subject Classification:} 60H10, 58D20, 58DF15, 34G20, 46E35, 35R15, 46E50}

\tableofcontents

\section{Introduction}
Stochastic PDEs for manifold-valued processes has attracted a great deal of attention due to their  wide range of applications in physics, in particular in kinetic theory of phase transitions and quantum field theory, see e.g.  Bruned et. al. \cite{Bruned+Gabriel+Hairer+Zambotti_2019}, the first and the second named authors  \cite{Brz+Caroll}-\cite{Brz+Gold+Jeg_2017}, Carroll \cite{Carroll_1999}, Funaki \cite{Funaki_1992} and R\"ockner et. al. \cite{Rockner+Wu+Zhu+Zhu_2018} and references therein. In this paper we are dealing with a particular stochastic PDE, known as a stochastic geometric wave equation (SGWE), that was introduced and studied by the first and the third named authors  in  a series of papers \cite{Brz+Ondr_2007}, \cite{Brz+Ondr_2011, Brz+Ondr_2013}, see also \cite{bgo}.
\par\medskip\noindent
The aim of this paper is to prove a large deviations principle (LDP) for  the one-dimensional stochastic wave equation with solutions taking values in a $d$-dimensional compact Riemannian manifold $M$.
More precisely we will consider the equation
\begin{equation}\label{SGWEinConnection}
	\mathbf D_t\partial_tu^{\varepsilon} =\mathbf D_x\partial_xu^{\varepsilon} +  \sqrt{\varepsilon} Y_{u^{\varepsilon}} (\partial_t u^{\varepsilon}, \partial_x u^{\varepsilon} ) \, \dot{W},
\end{equation}
where $\varepsilon \in  (0,1]$ approaches zero.  Here $\mathbf D$ is the connection  on the pull-back bundle $u^{-1}TM$ of the tangent bundle over $M$ induced by the Riemannian connection on $M$, see e.g. \cite{Brz+Ondr_2010C,Shatah+Struwe_1998B},  $Y$ is a non-linearity and $W$ is a spatially homogeneous Wiener process on $\mathbb{R}$.
A precise formulation  is provided in Section \ref{sec:prelim}. Here we  only note that we will work with the extrinsic formulation of  \eqref{SGWEinConnection}, that is, we assume $M$ to be isometrically embedded  into a certain Euclidean space $\mathbb{R}^n$, which holds true due to the celebrated Nash isometric embedding theorem \cite{Nash_1956}. Then, in view of Remark 2.5 in \cite{Brz+Ondr_2007}, equation  \eqref{SGWEinConnection} can be written in the form
\begin{equation}\label{SGWE-fundamental}
	\partial_{tt}u^{\varepsilon} =  \partial_{xx}u^{\varepsilon} +  A_{u^{\varepsilon}}(\partial_t u^{\varepsilon}, \partial_t u^{\varepsilon}) - A_{u^{\varepsilon}}(\partial_x u^{\varepsilon}, \partial_x u^{\varepsilon})  + \sqrt{\varepsilon} Y_{u^{\varepsilon}}(\partial_t u^{\varepsilon}, \partial_x u^{\varepsilon})\, \dot{W},
\end{equation}
where $A$ is the second fundamental form of the submanifold $M \subseteq \mathbb{R}^n$. More details about the equivalence of extrinsic and intrinsic formulations of stochastic PDEs can be found in Sections 2 and 12 of  \cite{Brz+Ondr_2007}.
\par\medskip\noindent
Due to its importance for applications, LDP for stochastic PDEs has been widely studied by many authors. However, analysis of large deviations for stochastic PDEs for manifold-valued processes is very little understood. To the best of our knowledge, LDP has only been established for the stochastic Landau-Lifshitz-Gilbert equation with solutions taking values in the two dimensional sphere \cite{Brz+Gold+Jeg_2017}.  Our paper is the first to study LDP for SGWE. One should also mention a PhD thesis by Hussain \cite{Javed_2015T}, see also \cite{Brz+Javed_LDP_2019}, who has established the LDP  for stochastic heat equation with one codimensional  constraint.
\par\medskip\noindent
If $\varepsilon=0$ then equation \eqref{SGWE-fundamental} reduces to a deterministic equation for wave maps. It has been intensely studied in recent years due to its importance in field theory and general relativity, see for example \cite{geba} and references therein. It turns out that solutions to the deterministic geometric wave equation can exhibit a very complex behaviour including (multiple) blowups and shrinking and expanding bubbles, see \cite{Biernat+Bizon_2011,Bizon+Chmaj+Tabor_2001}. In some cases the Soliton Resolution Conjecture has been proved, see \cite{jia}. Various concepts of stability of  these phenomena, including the stability of soliton solutions has also been intensely studied \cite{Donninger_2011}. It seems natural to investigate stability for wave maps by investigating the impact of small random perturbations and this idea leads to equation \eqref{SGWE-fundamental}. Let us recall that the stability of solitons under the influence of noise has already been studied by means of LDP for the  Schr\"odinger equations, see \cite{debusz}. LDP, once established, will provide a tool for more precise analysis of the stability of wave maps.
\par\medskip\noindent
Finally, let us recall that in \cite{norris} large deviations techniques are applied to derive a rigorous connection between the Yang-Mills measure and the energy functional. While in our work the problem is much easier because of the assumed regularity of the noise, we believe we provide a starting point for an analogous result in the case of less regular noises. Equations of stochastic flows for harmonic maps with very irregular noise have been recently proposed in \cite{Bruned+Gabriel+Hairer+Zambotti_2019} and \cite{Rockner+Wu+Zhu+Zhu_2018}.
\par\medskip\noindent
Another motivation for studying equation \eqref{SGWE-fundamental} with $\epsilon>0$ comes from the Hamiltonian structure of deterministic wave equation. Deterministic Hamiltonian systems may have infinite number of invariant measures and are not ergodic, see the discussion of this problem in \cite{Eckmann+Ruelle_1985}. Characterisation of such systems is a long standing problem.
The main idea, which goes back to Kolmogorov-Eckmann-Ruelle, is to choose a suitable small random perturbation such that the solution to stochastic system is a Markov process with the unique invariant measure and then one can select a \enquote{physical} invariant measure of the deterministic system by taking the limit of vanishing noise, see for example \cite{Cruzeiro+Haba_1997}, where this idea is applied to wave maps. A finite dimensional toy example was studied in \cite{Brz+etal_2015}.
\par\medskip\noindent
Our proof of the large deviations principle relies on the weak convergence method introduced in \cite{Budhiraja+Dupuis_2000} and is based on a variational representation formula for certain functionals of the driving infinite dimensional Brownian motion. However, the approach of \cite{Budhiraja+Dupuis_2000} can not be directly applied to the SGWE and requires a number of modifications, see Section 5 below.
\par\medskip\noindent
Recently in \cite{Salins+Budh+Dup_2019} the authors have established a LDP for a certain  class of Banach space valued stochastic differential equations by a different method, but their argument does not apply to SGWE studied in this paper  because the
wave operator does not generate a compact $C_0$-semigroup.
\par\medskip\noindent

Finally, we note that the approach we developed in this paper can be applied  to a number of problems that are open at present, including the beam equation studied in \cite{Brz+Maslowski+Seidler_2005}, and the nonlinear wave equation with polynomial nonlinearity and spatially homogeneous noise. In particular, this method would generalize the results of \cite{Ondr_2010} and \cite{Zhang_2007}. Our approach would also lead to an extension of the work of Martirosyan \cite{Martirosyan_2017} who considers a nonlinear wave equations on a bounded domain. We believe that the methods of the present work will allow us to obtain the large deviations principle for the family of stationary measures generated by the flow of stochastic wave equation, with multiplicative white noise, in non-local Sobolev spaces over the full space $\mathbb{R}^d$.

\par\medskip\noindent
The organisation of the paper is as follows. In Section \ref{sec:notation}, we introduce our notation and state the definitions used in the paper.  Section \ref{sec:prelim} contains some properties of the nonlinear drift terms and  the diffusion coefficient that we need later.  In Section \ref{sec:skeleton}  we prove the existence of a unique global and strong in PDE sense solution to the skeleton equation associated to \eqref{SGWE-fundamental}.   The proof of Large Deviations Principle, based on weak convergence approach, is provided in Section \ref{sec:LDP}.  In Appendix \ref{sec:IntAndExtSoln}, we recall the intrinsic and extrinsic formulation of SGWE  from  \cite{Brz+Ondr_2007} and state, without proof, the equivalence result between them. We conclude the paper with Appendices \ref{sec:existUniqResult} and \ref{sec:EnergyIneqSWE}, where we state modified version of the existing results on global well-posedness of \eqref{SGWE-fundamental} and energy inequality from \cite{Brz+Ondr_2007} that we use frequently in the paper.

Finally, let us point out that the current paper is an expanded and corrected version of a paper \cite{BGR_2020}.

\subsection*{Acknowledgments}  Ben Go{\l}dys was supported by the Australian Research Council Project DP200101866,  Nimit Rana was supported by the Australian Research Council Projects
DP160101755 and DP190103451, Zdzis{\l}aw Brze{\'z}niak  was supported by the Australian Research Council Project ARC DP grant DP180100506
and Martin Ondrej{\'a}t was supported by the
Czech Science Foundation grant no. 19-07140S.
Nimit Rana and Zdzis{\l}aw Brze{\'z}niak would like to thank Department of Mathematics,
	the University of Sydney and School of Mathematics, UNSW, respectively, for hospitality during August/September 2019.

\section{Notation}\label{sec:notation}
For any two non-negative quantities $a$ and $b$, we write $a \lesssim b$ if there exists a universal constant $c >0$ such that $a \leq cb$, and we write $a \simeq b$ when $a \lesssim b$ and $b \lesssim a$.  In case we want to emphasize the dependence of $c$ on some parameters $a_1,\ldots,a_k$, then we write, respectively, $\lesssim_{a_1,\ldots,a_k}$ and $\simeq_{a_1, \ldots, a_k}$.
We will denote by $B_R(a)$,  for  $a\in \mathbb R$ and $R>0$,  the open ball in $\mathbb R$ with center at $a$ and we put $B_R=B_R(0)$.  Now we  list the notation used throughout the whole paper.
\begin{trivlist}
	\item[$\bullet$] $\mathbb{N}=\{0,1,\cdots\}$ denotes  the set of natural numbers, $\mathbb{R}_+=[0,\infty)$, $\mathrm{Leb}$  denotes the Lebesgue measure.
	\item[$\bullet$]  Let $I \subseteq \mathbb{R}$ be an open interval. By $L^p(I;\mathbb{R}^n), p \in [1,\infty)$, we denote the classical real Banach space of all (equivalence classes of) $\mathbb{R}^n$-valued $p$-integrable maps on $I$. The norm on $L^p(I;\mathbb{R}^n)$ is given by
	\begin{equation}\nonumber
		\| u\|_{L^p(I;\mathbb{R}^n)} := \left( \int_{I} |u(x) |^p \, dx \right)^{\frac{1}{p}}, \qquad u\in L^p(I;\mathbb{R}^n),
	\end{equation}
	where $|\cdot|$ is Euclidean norm on $\mathbb{R}^n$.  For $p=\infty$, we consider the usual modification to essential supremum.
	\item[$\bullet$]  For any $p\in [1,\infty]$, $L^p_{\textrm{loc}}(\mathbb{R};\mathbb{R}^n)$ stands for a metrizable topological vector space equipped with a natural countable family of seminorms $ \{p_j\}_{j\in\mathbb{N}}$ defined by
	\begin{equation}\nonumber
		p_j(u) := \| u\| _{L^p(B_j; \mathbb{R}^n)}, \qquad u\in L^2_{\textrm{loc}}(\mathbb{R};\mathbb{R}^n), \; j\in\mathbb{N}.
	\end{equation}
	\item[$\bullet$] By $H^{k,p}(I;\mathbb{R}^n)$, for $p\in [1,\infty]$ and $k\in \mathbb{N}$, we denote the Banach space of all $u \in L^p(I;\mathbb{R}^n)$ for which $D^j u\in L^p(I;\mathbb{R}^n), j=0,1,\ldots,k$, where $D^j$ is the weak derivative of order $j$. The norm here is given by
	\begin{equation}\nonumber
		 	\| u\|_{H^{k,p}(I;\mathbb{R}^n)} := \left(  \sum_{j=0}^{k}  \| D^j u\|_{L^p(I;\mathbb{R}^n)}^p \right)^{\frac{1}{p}}, \qquad u\in H^{k,p}(I;\mathbb{R}^n).
	\end{equation}
	\item[$\bullet$] We write $H^{k,p}_{\textrm{loc}}(\mathbb{R};\mathbb{R}^n)$, for $p\in [1,\infty]$ and $k\in \mathbb{N}$, to denote the space of all elements $u\in L^p_{\textrm{loc}}(\mathbb{R}; \mathbb{R}^n)$ whose weak derivatives up to order $k$ belong to $L^p_{\textrm{loc}}(\mathbb{R}; \mathbb{R}^n)$.  It is relevant to note that $H^{k,p}_{\textrm{loc}}(\mathbb{R}; \mathbb{R}^n)$ is a metrizable topological vector space equipped with the following natural countable family of seminorms $\{ q_j\}_{j\in\mathbb{N}}$,
	\begin{equation}\nonumber
		q_j(u) := \Vert u\Vert _{H^{k,p}(B_j;\mathbb{R}^n)}, \qquad u\in H^{k,p}_{\textrm{loc}}(\mathbb{R}; \mathbb{R}^n), \; j\in\mathbb{N}.
	\end{equation}
	The spaces $H^{k,2}(I; \mathbb{R}^n)$ and $H^{k,2}_{loc}(\mathbb{R}; \mathbb{R}^n)$ are usually denoted by $H^{k}(I; \mathbb{R}^n)$ and $H^k_{loc}(\mathbb{R}; \mathbb{R}^n)$ respectively.
	\item[$\bullet$] We set $\mathcal H := H^2(\mathbb{R};\mathbb{R}^n)\times H^1(\mathbb{R};\mathbb{R}^n)$, $\mathcal H_{\textrm{loc}} := H^2_{\textrm{loc}}(\mathbb{R}; \mathbb{R}^n)\times H^1_{\textrm{loc}}(\mathbb{R};\mathbb{R}^n)$.
	\item[$\bullet$] To shorten the notation in calculation we set the following rules:
	\begin{itemize}
		\item  if the space where function is taking value, for example $\mathbb{R}^n$,  is clear then to save the space we will omit $\mathbb{R}^n$, for example $H^k(I)$ instead $H^k(I;\mathbb{R}^n)$;
		\item  if $I= (0,T) \textrm{ or } (-R,R) \textrm{ or } B(x,R)$, for some $T, R>0$ and $x \in \mathbb{R}$,  then instead of $L^p(I; \mathbb{R}^n)$ we write, respectively, $L^p(0,T; \mathbb{R}^n)$, $L^p(B_R;\mathbb{R}^n)$, $L^p(B(x,R); \mathbb{R}^n)$. Similarly for $H^k$ and $H_{\textrm{loc}}^k$ spaces.
		\item write $\mathcal{H}(B_R)$  or $\mathcal{H}_R$ for $H^2((-R,R); \mathbb{R}^n) \times H^1((-R,R);\mathbb{R}^n)$.
	\end{itemize}
	\item[$\bullet$] For any nonnegative integer $j$, let $\mathcal{C}^j(\mathbb{R})$ be the space of real valued continuous functions whose derivatives up to order $j$ are continuous on $\mathbb{R}$. We  also need the family of spaces $\mathcal{C}_b^j(\mathbb{R})$ defined by
	\begin{equation}\nonumber
		\mathcal{C}_b^j(\mathbb{R}) := \left\{ u \in  \mathcal{C}^j(\mathbb{R}); \forall \alpha \in \mathbb{N}, \alpha \leq j, \exists K_\alpha , \| D^j u\|_{L^\infty(\mathbb{R})} < K_\alpha  \right\}.
	\end{equation}
	\item[$\bullet$] Given $T>0$ and Banach space $E$, we denote by $\mathcal{C}([0,T]; E)$  the real Banach space of all $E$-valued continuous functions $u: [0,T] \to E$ endowed with the norm
	\begin{equation}\nonumber
		\| u \|_{\mathcal{C}([0,T];E)} := \sup_{t \in [0,T]} \| u(t) \|_E, \qquad u \in \mathcal{C}([0,T];E).
	\end{equation}
	By $\prescript{}{0}{\mathcal{C}}([0,T]; E) $ we mean the set of elements of $\mathcal{C}([0,T];E)$ vanishes at origin, that is,
	\begin{equation}\nonumber
	\prescript{}{0}{\mathcal{C}}([0,T]; E)  := \left\{ u \in \mathcal{C}([0,T]; E)  : u(0) =0 \right\}.
	\end{equation}
	\item[$\bullet$] For given metric space $(X,\rho)$, by $\mathcal{C}(\mathbb{R}; X)$ we mean the space of continuous functions from $\mathbb{R}$  to $X$ which is equipped with the metric
	\begin{equation}\nonumber
		(f,g) \mapsto \sum_{j=1}^{\infty} \frac{1}{2^j} \min\{ 1, \sup_{t \in [-j,j] } \rho(f(t), g(t)) \}.
	\end{equation}
	\item[$\bullet$] We denote the tangent and the normal bundle of  a smooth manifold $M$ by $TM$ and $NM$,  respectively.   Let $\mathfrak{F}(M)$ be the set of all smooth $\mathbb{R}$-valued function on $M$.
	\item[$\bullet$] A map $u: \mathbb{R} \to M$ belongs to $H_{loc}^k(\mathbb{R};M)$ provided that $\theta \circ u \in H_{loc}^k(\mathbb{R};\mathbb{R})$ for every $\theta \in \mathfrak{F}(M)$.  We equip $H_{loc}^k(\mathbb{R};M)$ with the topology induced by the mappings
	\begin{equation}\nonumber
	H_{loc}^k(\mathbb{R};M) \ni u \mapsto \theta \circ u \in H_{loc}^k(\mathbb{R};\mathbb{R}), \quad \theta \in \mathfrak{F}(M).
	\end{equation}
	Since the tangent bundle $TM$ of a manifold $M$ is also a manifold, this definition covers Sobolev spaces of $TM$-valued maps too.
	\item[$\bullet$]  By $\mathcal{L}(X,Y)$ we denote the space of all linear continuous operators from a topological vector space $X$ to $Y$. If $H_1, H_2$ are two separable Hilbert spaces then $\mathscr{L}_2\left(H_1,H_2\right)\subset\mathcal L\left( H_1,H_2\right)$ will denote the space of Hilbert–Schmidt operators acting from $H_1$ to $H_2$.
	\item[$\bullet$] We denote by $\mathcal{S}(\mathbb{R})$ the space of Schwartz functions on $\mathbb{R}$ and write $\mathcal S^\prime(\mathbb{R})$ for its dual, which is the space of tempered distributions on $\mathbb R$. By $L_{\varpi}^2$ we denote the weighted space $L^2(\mathbb{R}, \varpi dx)$, where $\varpi(x) := e^{-x^2}, x \in \mathbb{R}$, is an element of $\mathcal{S}(\mathbb{R})$. Let $H_{\varpi}^s(\mathbb{R}),	s \geq 0$, be the completion of $\mathcal{S}(\mathbb{R})$ with respect to the norm
	\begin{equation}\nonumber
		\| u\|_{H_{\varpi}^s(\mathbb{R})} := \left( \int_{\mathbb{R}}   (1+|x|^2)^s |\mathcal{F}(\varpi^{1/2} u)(x)|^2 \, dx  \right)^{\frac{1}{2}},
	\end{equation}
	where $\mathcal{F}$ denoted the Fourier transform.
\end{trivlist}

\section{Preliminaries}\label{sec:prelim}
In this section we discuss all the required preliminaries about the nonlinearity and the diffusion coefficient  that we need in Section \ref{sec:skeleton}.   We are following Sections 3 to 5 of \cite{Brz+Ondr_2007} very closely here. Below we use the notation $\mathcal{F}(\cdot)$, along with $\widehat{\cdot}$, to denote the Fourier transform.

\subsection{The Wiener process}
 Let $\mu$ be a symmetric Borel measure on $\mathbb{R}$. The random forcing we consider is in the form of a spatially homogeneous Wiener process on $\mathbb{R}$ with a spectral measure $\mu$ satisfying
\begin{equation}\label{measure}
	\int_{\mathbb R}(1+|x|^2)^2 \,\mu(dx)<\infty\,.
\end{equation}
An $\mathcal S^\prime(\mathbb{R})$-valued process $W = \{W(t), t \geq 0 \}$, on a given stochastic basis $(\Omega,\mathfrak{F}, (\mathfrak{F}_t)_{t \geq 0}, \mathbb{P})$, is called a spatially homogeneous Wiener process with spectral measure $\mu$ provided that
\begin{enumerate}
	\item for every $\varphi\in\mathcal{S}(\mathbb{R}) $,  $\{ W(t) (\varphi), t \geq 0\}$ is a real-valued $\left(\mathfrak F_t\right)$-adapted Wiener process,
	\item $W(t)(a\varphi+\psi)=aW(t)(\varphi)+W(t)(\psi)$ holds almost surely for every $t\geq 0$, $a\in\mathbb R$ and $\varphi,\psi\in\mathcal {S} (\mathbb{R})$,
\end{enumerate}
It is shown in \cite{Pes+Zab_1997} that the Reproducing Kernel Hilbert Space (RKHS) $H_\mu$ of the Gaussian distribution of $W(1)$ is given by
\begin{equation}\nonumber
	H_{\mu} := \left\{  \widehat{\psi \mu} :\psi \in L^2(\mathbb{R}^d, \mu, \mathbb{C}), \psi(x) = \overline{\psi(-x)}, x \in \mathbb{R}  \right\},
\end{equation}
where $L^2(\mathbb{R}^d, \mu, \mathbb{C})$ is the Banach space of complex-valued  functions that are square integrable with respect to the measure $\mu $. Note that $H_{\mu}$ endowed with inner-product
\begin{equation}\nonumber
	\left\langle  \widehat{\psi_1 \mu}, \widehat{\psi_2 \mu}  \right\rangle_{H_{\mu}} := \int_{\mathbb{R}} \psi_1 (x) \overline{\psi_2(x)} \, \mu(dx),
\end{equation}
is a Hilbert space.

Recall from \cite{Pes+Zab_1997,Pes+Zab_2000}  that $W$ can be regarded as a cylindrical Wiener process on $\rkhs$ and it takes values in any Hilbert space $E$, such that the embedding $\rkhs \hookrightarrow E$ is Hilbert-Schmidt. Since we explicitly know the structure of $\rkhs$, the next result, whose proof is  based on \cite[Lemma 2.2]{Peszat_2002} and discussion with Szymon Peszat \cite{Peszat_Oral} shows that assumption \eqref{measure} is equivalent to saying that the paths of $W$ belong to $\mathcal{C}([0,T];H_{\varpi}^2(\mathbb{R}))$.
\begin{Lemma}\label{example-W}
	Let us assume that the measure $\mu$ satisfies \eqref{measure}. Then the identity map from $\rkhs$ into $H_{\varpi}^2(\mathbb{R})$ is a Hilbert-Schmidt operator.
\end{Lemma}
\begin{proof}[\textbf{Proof of Lemma \ref{example-W}}]
	To simplify the notation we set  $$L^2_{(s)}(\mathbb{R}, \mu):= \{ f \in L^2(\mathbb{R}, \mu; \mathbb{C}): f(x) = \overline{f(-x)}, \;\; \forall  x \in \mathbb{R}  \}. $$
	Let $\{e_k\}_{k \in \mathbb{N}} \subset \mathcal{S}(\mathbb{R})$ be an orthonormal basis of $L^2_{(s)}(\mathbb{R}, \mu)$. Then, by the definition of $\rkhs$, $\{\mathcal{F}(e_k \mu)\}_{k \in \mathbb{N}}$ is an orthonormal basis of $\rkhs$.
	Invoking the convolution property of the Fourier transform and the Bessel inequality,
	we obtain,
	\begin{align}
		\sum_{k=1}^{\infty} \| \widehat{e_k \mu} \|_{H_{\varpi}^{2}}^{2} & =  \sum_{k=1}^{\infty} \int_{\mathbb{R}} (1+ |x|^2 ) | \mathcal{F}\left(\varpi^{1/2} \mathcal{F}(e_k \mu)  \right)(x)|^2 \, dx \nonumber\\
		& =   \int_{\mathbb{R}} (1+ |x|^2 )^2 \left( \sum_{k=1}^{\infty}  |\mathcal{F}\left(\varpi^{1/2} \mathcal{F}(e_k \mu)  \right)(x)|^2 \right) \, dx \nonumber\\
		& =    \int_{\mathbb{R}} (1+ |x|^2 )^2 \left( \sum_{k=1}^{\infty}  \bigg|  \int_{\mathbb{R}} \mathcal{F}\left(\varpi^{1/2} \right)(x-z) e_k(z) \, \mu(dz)    \bigg|^2 \right) \, dx \nonumber\\
		& \leq   \int_{\mathbb{R}^2} (1+ |x|^2 )^2 |\mathcal{F}\left(\varpi^{1/2} \right)(x-z) |^2 \, \mu(dz)  \, dx \nonumber\\
		& =   \int_{\mathbb{R}^2} (1+ |x+z|^2 )^2  |\mathcal{F}\left(\varpi^{1/2} \right)(x) |^2 \, \mu(dz)  \, dx \nonumber\\
		 & \lesssim  \| \varpi^{1/2} \|_{H_{\varpi}^1(\mathbb{R})}^2  \int_{\mathbb{R}}  (1+ |z|^2 )^2  \, \mu(dz). \nonumber
	\end{align}
	Hence Lemma \ref{example-W}.
\end{proof}
It is relevant to note here that $H_{\varpi}^2(\mathbb{R})$ is a subset of $H_{\textrm{loc}}^2(\mathbb{R})$ and the embedding is continuous.

\begin{Remark}
	It is important to note that all the results of this paper are valid for any Wiener process which takes values in the space $H_{\varpi}^2(\mathbb{R})$ not just for the Wiener process which is space homogenous. However, in the case of  space homogeneity,  the solution process will be space homogeneous if the intial data is space homogeneous.
\end{Remark}

The next result, whose detailed proof can be found in \cite[Lemma 1]{Ondr_2004_mild}, plays very important role in deriving the required estimates for the terms involving diffusion coefficient.
\begin{Lemma}\label{hsop}
	If the measure $\mu$ satisfies \eqref{measure}, then $H_\mu$  is continuously embedded in $\mathcal{C}_b^2(\mathbb R)$.  Moreover, for given $g\in H^j(B(x,R); \mathbb{R}^n)$, where $x \in \mathbb{R}, R>0$ and $j\in\{0,1,2\}$, the multiplication operator
	$$H_\mu \ni \xi\mapsto g\cdot\xi\in H^j(B(x,R); \mathbb{R}^n), $$
	is Hilbert-Schmidt and  $\exists ~ c >0$, independent of $R$, $x$, $g$, $\xi$ and $j$, such that
	$$
	\Vert \xi\mapsto g\cdot\xi\Vert_{\mathscr L_2(H_\mu,H^j(B(x,R);\mathbb{R}^n)) } \leq c \| g \|_{H^j(B(x,R);\mathbb{R}^n)}.
	$$
\end{Lemma}
\begin{Remark}
Note that the constant of inequality $c$ in Lemma \ref{hsop} does not depend on the size and position of the ball. However, if we consider a cylindrical Wiener process, then $c$ will also depend on the centre $x$ but will be bounded on bounded sets with respect to $x$.
\end{Remark}

\subsection{Extensions of non-linear term}
By definition $A_p:T_pM \times T_pM\to N_pM$, $p\in M$,  where $T_pM\subseteq\mathbb R^n$ and $N_pM\subseteq\mathbb R^n$ are the tangent and the normal vector spaces at $p\in M$, respectively. It is well known, see e.g. \cite{Hermann_1968},  that $A_p$, $p\in M$,  is a symmetric bilinear form.

Since we are following the approach of \cite{Brz+Caroll}, \cite{Brz+Ondr_2007}, and \cite{Ham_1975}, one of the main steps in the proof of the existence theorem is to consider the problem \eqref{SGWE-fundamental} in the ambient space $\mathbb{R}^n$ with an appropriate extension of $A$ from their domain to $\mathbb{R}^n$.  In this section we discuss two extensions of $A$ which work fine in the context of stochastic wave map, as displayed in \cite{Brz+Ondr_2007}.

Let us denote by $\mathcal{E}$ the exponential function
$$T \mathbb{R}^n \ni (p,\xi) \mapsto p + \xi \in \mathbb{R}^n, $$
relative to the Riemannian manifold $\mathbb{R}^n$ equipped with the standard Euclidean metric.  The proof of the following proposition about the existence of an open set $O$ containing $M$, which is called a tubular neighbourhood of $M$,  can be found in \cite[Proposition 7.26, p. 200]{ONeill_1983}.
\begin{Proposition}\label{diffexp}
	There exists an $\mathbb{R}^n$-open neighbourhood $O$ around $M$ and an $NM$-open neighbourhood $V$ around the set $\{(p,0)\in NM :p\in  NM \}$  such that the restriction of the exponential map $\mathcal{E}|_V:V\to O$ is a diffeomorphism. Moreover, the neighbourhood $V$ can be chosen in such a way that $(p,t\xi)\in V$ whenever $t \in [-1,1]$ and $(p,\xi)\in V$.
\end{Proposition}
In case of no ambiguity,  we will denote the diffeomorphism $\mathcal E|_V:V\to O$ by $\mathcal E$. By using the Proposition \ref{diffexp},  diffeomorphism $i:NM \ni (p,\xi ) \mapsto  (p,-\xi) \in NM$ and the standard argument of partition of unity, one can obtain a function $\h: \mathbb{R}^n \to \mathbb{R}^n$ which identifies the manifold $M$ as its fixed point set. In precise we have the following result.
\begin{Lemma}\cite[Corollary 3.4 and Remark 3.5]{Brz+Ondr_2007}\label{cor_inv}
	There exists a smooth compactly supported function $\h: \mathbb{R}^n \to \mathbb{R}^n$ which has the following properties:
	\begin{enumerate}
		\item  restriction of $\h$ on $O$ is a diffeomorpshim,
		\item  $\restr{\h}{O}= \mathcal{E} \circ i \circ \mathcal{E}^{-1}: O \to O$ is an involution on the tubular neighborhood $O$ of $M$,
		\item  $\h(\h(q))=q$ for every $q\in O$, \label{inv3}
		\item if $q\in O$, then $\h(q)=q$ if and only if $q\in M$,
		\item   if $p\in M$, then
		$$\h^\prime(p)\xi=
		\begin{cases}\xi, &
		\text{ provided } \xi\in T_pM,\cr
		-\xi &
		\text{ provided } \xi\in N_pM.
		\end{cases}
		\label{inv5}
		$$
	\end{enumerate}
\end{Lemma}
The following result is the first extension of the second fundamental form that we use in this paper.
\begin{Proposition}\cite[Proposition 3.6]{Brz+Ondr_2007}\label{sft}
	If we define
	\begin{equation}\label{Boper}
	B_q(a,b)=\sum_{i,j=1}^n \frac{\partial^2 \h}{\partial q_i\partial q_j}(q)a_ib_j= \h^{\prime\prime}_q(a,b),\qquad q\in\mathbb R^n,\quad a,b\in\mathbb R^n,
	\end{equation}
	and
	\begin{equation}\label{exta}
		\mathcal A_q(a,b)=\frac{1}{2}B_{\h(q)}(\h^\prime(q)a, \h^\prime(q)b),\qquad q\in\mathbb R^n,\quad a,b\in\mathbb R^n,
	\end{equation}
	then, for every $p \in M$,
	\begin{equation}\nonumber
	\mathcal A_p(\xi,\eta)=A_p(\xi,\eta), \;\xi, \eta \in T_pM,
	\end{equation}
	and
	\begin{equation}\label{eqn-invariance-A}
	\mathcal A_{\h(q)}(\h^\prime(q)a, \h^\prime(q)b) = \h^\prime(q)\mathcal
	A_q(a,b)+B_q(a,b), \; q\in O, \,a,b\in\mathbb R^n.
	\end{equation}
\end{Proposition}

Along with the extension $\mathcal A$, defined by formula (\ref{exta}), we also need the extension $\mathscr A$, defined by formula \eqref{nexta}, of the second fundamental form tensor $A$ which will be perpendicular to the tangent space.

\begin{Proposition}\cite[Proposition 3.7]{Brz+Ondr_2007}\label{nextapro}
	Consider the function
	$$ \mathscr A: \mathbb{R}^n \times \mathbb{R}^n \times \mathbb{R}^n \ni (q,a,b) \mapsto \mathscr{A}_q(a,b)  \in \mathbb{R}^n,$$
	defined by formula
	\begin{equation}\label{nexta}
		\mathscr A_q(a,b)=\sum_{i,j=1}^n a_iv_{ij}(q)b_j=A_q(\pi_q(a),\pi_q(b)),\qquad q\in\mathbb R^n,\quad 	a\in\mathbb R^n,\quad b\in\mathbb R^n,
	\end{equation}
	where $\pi_p$, $p\in M$ is the orthogonal projection of $\mathbb R^n$ to $T_pM$, and $v_{ij}$, for $i,j\in\{1,\dots,n\}$, are smooth and symmetric (i.e. $v_{ij} = v_{ji}$) extensions of  $v_{ij}(p): = A_p(\pi_pe_i,\pi_pe_j)$ to ambient space $\mathbb{R}^n$.
	Then $\mathscr{A}$ satisfies the following:
	\begin{enumerate}
		\item  $\mathscr{A}$ is smooth in $(q,a,b)$ and symmetric in $(a,b)$ for every $q$,
		\item  $\mathscr A_p(\xi,\eta)=A_p(\xi,\eta)$ for every $p\in M$, $\xi,\eta\in T_pM$,
		\item  $\mathscr A_p(a,b)$ is perpendicular 	to $T_pM$ for every $p\in M$, $a,b\in\mathbb R^n$.
	\end{enumerate}
\end{Proposition}

\subsection{The $C_0$-group and the extension operators}\label{sec:semigroup}
Here we recall some facts on infinitesimal generators of the linear wave equation and on the extension operators in various  Sobolev spaces. Refer \cite[Section 5]{Brz+Ondr_2007} for details.
\begin{Proposition}\label{gr}
	Assume that $k,n \in \mathbb{N}$. The one parameter family of operators defined by
	$$
	S_t\left(\begin{array}{c}u\\v\end{array}\right)=\left(\begin{array}{ccc}
	\cos[t(-\Delta)^{1/2}]u^1&+&(-\Delta)^{-1/2}\sin[t(-\Delta)^{1/2}]v^1
	\\
	&\vdots&
	\\
	\cos[t(-\Delta)^{1/2}]u^n&+&(-\Delta)^{-1/2}\sin[t(-\Delta)^{1/2}]v^n
	\\
	-(-\Delta)^{1/2}\sin[t(-\Delta)^{1/2}]u^1&+&\cos[t(-\Delta)^{1/2}]v^1
	\\
	&\vdots&
	\\
	-(-\Delta)^{1/2}\sin[t(-\Delta)^{1/2}]u^n &+&\cos[t(-\Delta)^{1/2}]v^n
	\end{array}\right)
	$$
	is a $C_0$-group on
	$$ \mathcal H^k: =H^{k+1}(\mathbb R;\mathbb R^n)\times H^{k}(\mathbb R;\mathbb R^n), $$
	and its infinitesimal generator is an operator $\mathcal G^k=\mathcal G$ defined by
	\begin{eqnarray*}
		D(\mathcal G^k)&=&H^{k+2}(\mathbb R;\mathbb R^n) \times H^{k+1}(\mathbb R;\mathbb R^n) ,\\
		\mathcal G\left(\begin{array}{c}u\\v\end{array}\right)&=&
		\left(\begin{array}{c}
			v
			\\
			\Delta u
		\end{array}\right).
	\end{eqnarray*}
\end{Proposition}

The following theorem is well known, see e.g. \cite{LM_1972} and \cite[Section II.5.4]{Evans_1998}.
\begin{Proposition}\label{prop-extension}
	Let $k\in\mathbb{N}$. There exists a linear bounded operator
	$$E^k:H^k((-1,1); \mathbb{R}^n ) \to H^k(\mathbb R; \mathbb{R}^n),$$
	such that
	\begin{trivlist}
		\item[(i)] $E^kf=f$ almost everywhere on $(-1,1)$ whenever $f\in H^k((-1,1); \mathbb{R}^n)$,
		\item[(ii)] $E^kf$ vanishes outside of $(-2,2)$ whenever $f\in H^k((-1,1); \mathbb{R}^n)$,
		\item[(iii)] $E^kf\in \mathcal{C}^k(\mathbb R;  \mathbb{R}^n))$, if $f\in \mathcal{C}^k([-1,1];  \mathbb{R}^n))$,
		\item[(iv)] if $j\in\mathbb{N}$ and  $j<k$, then  there exists a unique   extension of $E^k$  to  a bounded linear operator from $H^j((-1,1); \mathbb{R}^n)$    to $H^j(\mathbb R; \mathbb{R}^n)$.
	\end{trivlist}
\end{Proposition}

\begin{Definition}
	For $k\in \mathbb{N}$, $r>0$ we define the operators
	$$E^k_r: H^j((-r,r); \mathbb{R}^n) \to H^j(\mathbb{R}; \mathbb{R}^n), \qquad j\in\mathbb{N}, j\leq k, $$
	called as $r$-scaled $E^k$ operators, by the following formula
	\begin{equation}\label{extension}
		(E^k_rf)(x)=\{E^k[y\mapsto f(yr)]\}\left( \frac{x}{r} \right),\qquad x\in\mathbb R,
	\end{equation}
	for $r>0$ and $f\in H^k((-r,r); \mathbb{R}^n)$.
\end{Definition}

The following remark will be useful in Lemma \ref{lem-lip}.
\begin{Remark}\label{rem-extension}
	We can rewrite \eqref{extension} as $(E_r^kf)(x) = (E^kf_r)(\frac{x}{r}), f \in H^k((-r,r);\mathbb{R}^n)$  where $f_r: (-1,1) \ni y \mapsto f(yr) \in \mathbb{R}^n$. Also, observe that for $ f \in H^1((-r,r);\mathbb{R}^n)$
	\begin{equation}\nonumber
		\| f_r \|_{H^1((-1,1); \mathbb{R}^n)}^2   \leq (r^{-1}+ r) \| f \|_{H^1((-r,r);\mathbb{R}^n)}^2.
	\end{equation}
\end{Remark}

\subsection{Diffusion coefficient}
In this subsection we discuss the assumptions on diffusion coefficient $Y$ which we only need in Section \ref{sec:skeleton}. It is relevant to note that due to a technical issue, which is explained in Section \ref{sec:LDP}, we need to consider stricter conditions on $Y$ in establishing the large deviation principle for \eqref{SGWE-fundamental}.  Here $Y_p:T_pM\times T_pM \to T_pM$,  for  $p \in M$,  is a mapping satisfying,
\begin{equation}\nonumber
	|Y_p(\xi,\eta)|_{T_pM}\le C_Y (1+|\xi|_{T_pM}+|\eta|_{T_pM}),\qquad p\in M,\quad\xi,\eta\in T_pM,
\end{equation}
for some constant $C_Y>0$ which is  independent of $p$.
By invoking Lemma \ref{cor_inv} and \cite[Proposition 3.10]{Brz+Ondr_2007}, we can extend the noise coefficient to map $Y: \mathbb{R}^n \times \mathbb{R}^n \times \mathbb{R}^n \ni (p,a,b) \mapsto  Y_p(a,b) \in \mathbb{R}^n$ which satisfies the following:
\begin{enumerate}
	\item[\textbf{Y.1}] for $q\in O$ and $a,b \in \mathbb{R}^n$,
	\begin{equation}\label{y1}
	Y_{\h(q)} \left(  \h^\prime(q)a, \h^\prime(q) b   \right) = \h^\prime (q) Y_q(a,b),
	\end{equation}
	\item[\textbf{Y.2}] there exists an compact set $K_Y \subset \mathbb{R}^n$  containing $M$ such that $Y_p(a,b) = 0$, for all $a,b \in \mathbb{R}^n$, whenever $p \notin K_Y$,
	\item[\textbf{Y.3}]  $Y$ is of $\mathcal{C}^2$-class and there exist positive constants $C_{Y_i}, i \in \{1,2,3\}$ such that, with notation $Y(p,a,b) := Y_p(a,b)$,  for every $p,a,b \in \mathbb{R}^n$,
	\begin{align}
	& \vert Y_p(a,b) \vert \leq C_{Y_0}  (1+ |a|+ |b|), \label{y00}\\
	& \bigg\vert \frac{\partial Y}{\partial p_i}(p,a,b) \bigg\vert \leq C_{Y_1} (1+ |a|+|b|), \quad i=1,\ldots,n, \label{y2} \\
	& \bigg\vert \frac{\partial Y}{\partial a_i}(p,a,b) \bigg\vert + \bigg\vert \frac{\partial Y}{\partial b_i}(p,a,b) \bigg\vert \leq C_{Y_2}, \quad i=1,\ldots,n, \label{y3} \\
	& \bigg\vert \frac{\partial^2 Y}{\partial x_j \partial y_i}(p,a,b) \bigg\vert \leq C_{Y_3}, \quad x,y \in \{p,a,b\} \textrm{ and } i,j\in \{1,\ldots,n \}  \label{y4}.
	\end{align}
\end{enumerate}

\section{Skeleton equation}\label{sec:skeleton}
The purpose of this section is to introduce and study the deterministic equation associated to   \eqref{SGWE-fundamental}. Define
\begin{equation}\nonumber
\prescript{}{0}{H}^{1,2}(0,T; \rkhs) := \left\{ h \in \prescript{}{0}{\mathcal{C}}([0,T];\rkhs): \dot{h} \in L^2(0,T;\rkhs)  \right\}.
\end{equation}
Note that $\prescript{}{0}{H}^{1,2}(0,T;\rkhs )$ is a Hilbert space with norm $\int_{0}^{T} \|\dot{h}(t)\|_{\rkhs}^2 \, dt$ and the map
\begin{equation}\nonumber
	L^2(0,T;\rkhs) \ni \dot{h} \mapsto h = \left\{ t \mapsto \int_{0}^{t} \dot{h}(s)\, ds  \right\} \in \prescript{}{0}{H}^{1,2}(0,T;\rkhs),
\end{equation}
is an isometric isomorphism.
For $h \in \prescript{}{0}{H}^{1,2}(0,T;\rkhs )$, let us consider the so called \enquote{skeleton equation} associated to problem
\begin{equation}\label{CP-skeleton-int}
\left\{\begin{aligned}
& \mathbf D_t\partial_tu =\mathbf D_x\partial_xu+Y_u(\partial_tu,\partial_xu)\,\dot h,
\\
& u(0,\cdot)=u_0,  \partial_tu(t,\cdot)_{|t=0} =v_0
\end{aligned}
\right.\end{equation}
\begin{equation}\label{CP-skeleton}
\left\{\begin{aligned}
& \partial_{tt}u =  \partial_{xx}u +  A_{u}(\partial_t u, \partial_t u) - A_{u}(\partial_x u, \partial_x u)  +Y_u (\partial_t u, \partial_x u ) \, \dot{h}\,,
\\
& u(0,\cdot)  = u_0,\,\, \partial_tu (0,\cdot)=v_0.
\end{aligned}
\right.\end{equation}

Recall that $M$ is a compact Riemannian manifold which is isometrically embedded  into some Euclidean space $\mathbb  R^n$, and hence, we can assume that  $M$ is  a submanifold of  $\mathbb  R^n$. The following main result of this section is closely related to  \cite[Theorem 11.1]{Brz+Ondr_2007}.
\begin{Theorem}\label{thm-skeleton}
	Let $T>0$, $h \in \prescript{}{0}{H}^{1,2}(0,T;\rkhs)$ and $(u_0, v_0) \in H_{\textrm{loc}}^2 \times H_{\textrm{loc}}^1 (\mathbb{R};TM)$ are given.  Then for every $R >T$, there exists a  $u: [0,T) \times \mathbb{R} \to M$  such that the following hold:
	\begin{trivlist}
		\item[(i)] $u$ belongs to $\mathcal{C}^1(\mathbb  R_+\times\mathbb  R; M)$,
		\item[(ii)]  $[0,T) \ni t \mapsto u(t,\cdot) \in H^2((-R,R);M)$ is continuous,
		\item[(iii)]  $[0,T)\ni t \mapsto u(t,\cdot) \in H^1((-R,R);M)$ is continuously differentiable,
		\item[(iv)]$u(0,x)=u_0(x)$ and $\partial_tu(0,x)=v_0(x,\omega)$ holds for every $x\in\mathbb R$,
		\item[(v)] for every vector field $X$ on $M$, and every $t\ge 0$ and $R>0$
		\begin{align}
		\langle\partial_tu(t),X(u(t))\rangle_{T_{u(t)}M}&= \langle v_0,X(u_0)\rangle_{T_{u(t)}M}+\int_0^t\langle\mathbf D_x\partial_xu(s),X(u(s))\rangle_{T_{u(s)}M}\,ds
		\nonumber\\
		&\quad + \int_0^t\langle\partial_tu(s),\nabla_{\partial_tu(s)}X\rangle_{T_{u(s)}M}\,ds \nonumber
		\\
		&\quad + \int_0^t\langle X(u(s)),Y_{u(s)}(\partial_tu(s),\partial_xu(s)) \dot h (s)\rangle_{T_{u(s)}M} \, ds,
		\nonumber
		\end{align}
		holds in $L^2(-R,R)$.
	\end{trivlist}
	Moreover, if there exists another map $U: [0,T) \times \mathbb{R} \to M$ which also satisfies the above properties then
	$$ U(t,x)=u(t,x) \quad \textrm{for every} \quad |x|\leq R-t\quad \textrm{and} \quad t \in[0,T). $$
\end{Theorem}
\begin{proof}[\textbf{Proof of Theorem \ref{thm-skeleton}}]
	First note that, due to Theorem \ref{thm-equiv}, to prove the existence part it is sufficient to prove that for every $R >T$, there exists a  $u: [0,T) \times \mathbb{R} \to M$  such that the following  hold:
	\begin{enumerate}
		\item   $[0,T) \ni t \mapsto u(t,\cdot) \in H^2((-R,R);\mathbb{R}^n)$ is continuous,
		\item   $[0,T)\ni t \mapsto u(t,\cdot) \in H^1((-R,R);\mathbb{R}^n)$ is continuously differentiable,
		\item  $u(t,x) \in M$ for every $t \in [0,T)$ and $x \in \mathbb{R}$,
		\item  $u(0,x)  = u_0(x)$ and $\partial_t u(0,x) = v_0(x)$ for every $x \in \mathbb{R}$,
		\item  for every $t \in [0,T)$ the following holds in $L^2((-R,R);\mathbb{R}^n)$,
		\begin{align}\label{e1-skeleton}
		\partial_t u(t) & = v_0 + \int_{0}^{t} \left[ \partial_{xx}u(s) - A_{u(s)}(\partial_x u(s), \partial_x u(s)) + A_{u(s)}(\partial_t u(s), \partial_t u(s))   \right] ds \nonumber\\
		& \quad + \int_{0}^{t}  Y_{u(s)}(\partial_t u(s), \partial_x u(s)) \dot{h}(s )\, ds.
		\end{align}
	\end{enumerate}
	For the uniqueness we will show that if there exists another map $U: [0,T) \times \mathbb{R} \to M$ which also satisfies the above properties [1]-[5], then
	$$ U(t,x)=u(t,x) \quad \textrm{for every} \quad |x|\leq R-t\quad \textrm{and} \quad t \in[0,T). $$
	Since we seek solutions that take values in the Fr\'echet space $H^2_{\textrm{loc}}(\mathbb R; \mathbb{R}^n)\times H^1_{\textrm{loc}}(\mathbb R; \mathbb{R}^n)$,  we  localize the problem using a sequence of non-linear wave equations.

	Let us  fix $r> R + T$, and $k\in\mathbb N$. Let  $\varphi:\mathbb R\to\mathbb R$ be  a smooth compactly supported function such that $\varphi (x)=1$ for $x \in (-r,r)$ and $\varphi(x) = 0$ for $x \notin (-2r,2r)$.
	Next, with the convention $z=(u,v)\in\mathcal H$, we define the following maps
	\begin{eqnarray*}
		\mathbf F_r &:&[0,T]\times\mathcal H\ni \left(t,z\right)\mapsto \left(\begin{array}{c}0\\E^1_{r-t}[\mathcal A_u(v,v)-\mathcal A_u(u_x, u_x)]\end{array}\right)\in \mathcal H,
		\\
		\mathbf F_{r,k}&:&[0,T]\times\mathcal H\ni
		(t,z)\mapsto\left\{\begin{array}{lrl}
			\mathbf F_r(t,z),&\textrm{if}&\vert z\vert_{\mathcal H_{r-t}}\le k
			\\
			\left(2-\frac{1}{k}\vert z\vert_{\mathcal H_{r-t}}\right)\mathbf F_r(t,z),&\textrm{if}&k\le\vert z\vert_{\mathcal H_{r-t}}\le 2k
			\\
			0,&\textrm{if}&2k\le\vert z\vert_{\mathcal H_{r-t}}
		\end{array}\right.\in \mathcal H,
		\\
		\mathbf G_r&:&[0,T]\times\mathcal H\ni \left(t,z\right)\mapsto\left(\begin{array}{c}0\\(E^1_{r-t} Y_u(v,  u_x)) \cdot  \end{array}\right)\in \mathscr L_2(\rkhs,\mathcal H),
		\\
		\mathbf G_{r,k}&:&[0,T]\times\mathcal H\ni
		(t,z)\mapsto\left\{\begin{array}{lrl}
			\mathbf G_r(t,z),&\textrm{if}&\vert z\vert_{\mathcal H_{r-t}}\le k
			\\
			\left(2-\frac{1}{k}\vert z\vert_{\mathcal H_{r-t}}\right)\mathbf G_r(t,z),&\textrm{if}&k\le\vert z\vert_{\mathcal H_{r-t}}\le 2k
			\\
			0,&\textrm{if}&2k\le\vert z\vert _{\mathcal H_{r-t}}
		\end{array}\right. \in \mathscr L_2(\rkhs,\mathcal H),
		\\
		\mathbf Q_r&:&\mathcal H\ni z\mapsto\left(\begin{array}{c}\varphi\cdot \h(u)\\\varphi\cdot \h^\prime(u)v\end{array}\right)\in \mathcal H,
	\end{eqnarray*}
	where $(E^1_{r-t} Y_u(v,u_x)) \cdot $ means that, for every $(u,v) \in \mathcal{H}$, $E^1_{r-t} Y_u(v,u_x) \in H_{\textrm{loc}}^1(\mathbb{R}; \mathbb{R}^n)$ and the multiplication operator defined as
	$$ (E^1_{r-t} Y_u(v,u_x)) \cdot : \rkhs \ni \xi \mapsto (E^1_{r-t} Y_u(v,u_x)) \cdot \xi \in H_{\textrm{loc}}^1(\mathbb{R};\mathbb{R}^n), $$
	satisfy Lemma \ref{hsop}.

	The following two properties, which we state without proof, of $\mathbf{Q}_r$ are taken from \cite[Section 7]{Brz+Ondr_2007}.
	\begin{Lemma}\label{lem_consistency}
		If $z=(u,v)\in\mathcal H$ is such that $u(x)\in M$ and $v(x)\in T_{u(x)}M$ for $\vert x\vert <r$, then $\mathbf{Q}_r(z)=z$ on $(-r,r)$.	
	\end{Lemma}
	\begin{Lemma}\label{c2q}
		The mapping $\mathbf Q_r$ is of $\mathcal{C}^1$-class and  its derivative, with $z=(u,v)\in\mathcal H$,  satisfies
		$$
		\mathbf Q_r^\prime(z)w=\left(\begin{array}{c}\varphi\cdot \h^\prime(u)w^1\\\varphi\cdot [\h^{\prime\prime}(u)(v,w^1)+\h^\prime(u)w^2]\end{array}\right),\quad  w = (w^1,w^2) \in\mathcal H.
		$$
	\end{Lemma}
	The next lemma is about the locally Lipschitz properties of the localized maps defined above.
	\begin{Lemma}\label{lem-lip}
		For each $k\in\mathbb N$ the functions $\mathbf F_r$, $\mathbf F_{r,k}$, $\mathbf G_r$, $\mathbf G_{r,k}$ are continuous
		and there exists a constant $C_{r,k} >0$ such that
		\begin{align}\label{lem-lip-result}
		& \| \mathbf F_{r,k}(t,z)-\mathbf F_{r,k}(t,w) \|_{\mathcal H}  + \| \mathbf G_{r,k}(t,z)-\mathbf G_{r,k}(t,w) \|_{\mathscr{L}_2(\rkhs,\mathcal{H})} \le C_{r,k} \| z-w\|_{\mathcal H_{r-t}},
		\end{align}
		holds for every $t\in [0,T]$ and every $z,w\in\mathcal H$.
	\end{Lemma}
	\begin{proof}[\textbf{Proof of Lemma \ref{lem-lip}}]
		Let us fix $t \in [0,T]$ and $z=(u,v), w=(\tilde{u}, \tilde{v}) \in \mathcal{H}$. Note that due to the definitions of $\mathbf F_{r,k}$ and $\mathbf G_{r,k}$, it is sufficient to prove \eqref{lem-lip-result} in the case $\| z\|_{\mathcal H_{r-t}} , \| w \|_{\mathcal H_{r-t}} \leq k$.

		 Let us set $I_{rt} := (t-r,r-t)$.  Since in the chosen case $\mathbf F_{r,k}(t,z) = \mathbf F_{r}(t,z)$ and $\mathbf F_{r,k}(t,w) = \mathbf F_{r}(t,w)$, by Proposition \ref{prop-extension} and Remark \ref{rem-extension}, there exists $C_E(r,t) >0$ such that
		\begin{align}\label{lip-t1}
		\| \mathbf F_{r,k}(t,z)-\mathbf F_{r,k}(t,w) \|_{\mathcal H}  & \leq C_E(r,t) \left[ \| \mathcal{A}_{u}(v,v) - \mathcal{A}_{ \tilde{u} }( \tilde{v},\tilde{v} ) \|_{H^1(I_{rt})} \right. \nonumber\\
		& \quad \left. +  \|  \mathcal{A}_{u}(u_x,u_x) - \mathcal{A}_{\tilde{u}}( \tilde{u}_x,\tilde{u}_x) \|_{H^1(I_{rt})}  \right].
		\end{align}
		Since $\h$ is smooth and has compact support, see Lemma  \ref{cor_inv}, from \eqref{exta} observe that
		$$ \mathcal{A} : \mathbb{R}^n \ni q \mapsto \mathcal{A}_q \in \mathcal{L}(\mathbb{R}^n \times \mathbb{R}^n; \mathbb{R}^n), $$
		is smooth, compactly supported (in particular bounded) and globally Lipschitz. Recall the following well-known interpolation inequality, refer \cite[(2.12)]{Brz+Gold+Jeg_2017},
		\begin{equation}\label{interpolation1}
		\| u\|_{L^\infty(I)}^2  \leq k_e^2 \| u\|_{L^2(I)} \| u\|_{H^1(I)}, \quad u \in H^1(I),
		\end{equation}
		where $I$ is any open interval in $\mathbb{R}$ and $k_e = 2 \max\left\{ 1, \frac{1}{\sqrt{|I|}}  \right\}$. Note that since $r>R + T$ and $t \in [0,T]$, $|I_{rt}| = 2(r-t) >2R$. Thus, we can choose $k_e = 2 \max\left\{ 1, \frac{1}{\sqrt{|R|}} \right\}$.Consequently, using the above mentioned properties of $\mathcal{A}$ and the interpolation inequality \eqref{interpolation1} we get
		\begin{align}\label{lip-t2}
		\| \mathcal{A}_{u}(v,v) - \mathcal{A}_{\tilde{u}}(\tilde{v},\tilde{v}) \|_{L^2(I_{rt})} & \leq  \| \mathcal{A}_{u}(v,v) - \mathcal{A}_{\tilde{u}}(v,v) \|_{L^2(I_{rt})} \nonumber\\
		& \quad  + \| \mathcal{A}_{\tilde{u}}(v,v) - \mathcal{A}_{\tilde{u}}(\tilde{v},v) \|_{L^2(I_{rt})} \nonumber\\
		& \quad + \| \mathcal{A}_{\tilde{u}}(\tilde{v},v) - \mathcal{A}_{\tilde{u}}(\tilde{v},\tilde{v}) \|_{L^2(I_{rt})} \nonumber\\
		& \leq L_{\mathcal{A}} \| v\|_{L^\infty(I_{rt})}^2 \| u-\tilde{u}\|_{L^2(I_{rt})} \nonumber\\
		& \quad + B_{\mathcal{A}} \left[ \| v\|_{L^\infty(I_{rt})}  + \| \tilde{v} \|_{L^\infty(I_{rt})} \right] \| v-\tilde{v}\|_{L^2(I_{rt})} \nonumber\\
		& \leq C(L_{ \mathcal{A}},B_{\mathcal{A}},R, k, k_e) \| z-w\|_{\mathcal{H}_{r-t}},
		\end{align}
		where $L_{\mathcal{A}}$ and $B_{\mathcal{A}}$  are the Lipschitz constants and bound of $\mathcal{A}$, respectively. Next, since $\mathcal{A}$ is smooth and have compact support, if we set $L_{\mathcal{A}^\prime}$ and $B_{\mathcal{A}^\prime}$  are the Lipschitz constants and bound of
		$$ \mathcal{A}^\prime: \mathbb{R}^n \ni q \mapsto d_q\mathcal{A} \in \mathcal{L}(\mathbb{R}^n \times \mathbb{R}^n \times \mathbb{R}^n ; \mathbb{R}^n),$$
		then by adding and subtracting the terms as we did to get \eqref{lip-t2} followed by the properties of $\mathcal{A}^\prime$  and the interpolation inequality \eqref{interpolation1} we have
		\begin{align}\label{lip-t3}
		& \| d_x \left[ \mathcal{A}_{u}(v,v) - \mathcal{A}_{\tilde{u}}(\tilde{v},\tilde{v}) \right] \|_{L^2(I_{rt})} \nonumber\\
		&  \leq \| d_{u}\mathcal{A}(v,v) (u_x) - d_{\tilde{u}}\mathcal{A}(\tilde{v},\tilde{v})(\tilde{u}_x) \|_{L^2(I_{rt})}  + 2 \| \mathcal{A}_{u}( v_x,v) - \mathcal{A}_{\tilde{u}}(\tilde{v}_x,\tilde{v})  \|_{L^2(I_{rt})} \nonumber\\
		& \leq L_{\mathcal{A}^\prime} \| u_x\|_{L^\infty(I_{rt})}  \|v\|_{L^\infty(I_{rt})}^2  \|u - \tilde{u} \|_{L^2(I_{rt})} +  B_{\mathcal{A}^\prime}  \|v\|_{L^\infty(I_{rt})}^2  \| u_x - \tilde{u}_x \|_{L^2(I_{rt})}  \nonumber\\
		& \quad +  B_{\mathcal{A}^\prime} \left[\|v\|_{L^\infty(I_{rt})} + \|\tilde{v} \|_{L^\infty(I_{rt})} \right]  \|v - \tilde{v}\|_{L^2(I_{rt})}  \| \tilde{u}_x \|_{L^\infty(I_{rt})} \nonumber\\
		& \quad + 2 \left[   L_{\mathcal{A}} \|u - \tilde{u} \|_{L^\infty(I_{rt})} \|v\|_{L^\infty(I_{rt})} \| v_x\|_{L^2(I_{rt})} + B_{\mathcal{A}} \| v_x -  \tilde{v}_x  \|_{L^2(I_{rt})} \|v\|_{L^\infty(I_{rt})} \right. \nonumber\\
		& \quad \left. + B_{\mathcal{A}} \|v - \tilde{v}  \|_{L^\infty(I_{rt})} \| \tilde{v}_x\|_{L^2(I_{rt})}  \right] \nonumber\\
		& \lesssim_{L_{\mathcal{A}},B_{\mathcal{A}}, L_{\mathcal{A}^\prime},B_{\mathcal{A}^\prime}, k_e}  \left[\|u - \tilde{u} \|_{H^2(I_{rt})} \|u\|_{H^2(I_{rt})}  \|v\|_{H^1(I_{rt})}^2 + \|u - \tilde{u} \|_{H^2(I_{rt})} \|v\|_{H^1(I_{rt})}^2 \right. \nonumber\\
		& \quad \left. + \|v - \tilde{v}\|_{H^1(I_{rt})} \left[ \|v\|_{H^1(I_{rt})} +  \|\tilde{v} \|_{H^1(I_{rt})} \right] \| \tilde{u} \|_{H^2(I_{rt})}  + \|u - \tilde{u} \|_{H^2(I_{rt})} \|v\|_{H^1(I_{rt})}^2 \right. \nonumber\\
		& \quad \left. +  \|v - \tilde{v}  \|_{H^1(I_{rt})}  \left( \|v\|_{H^1(I_{rt})} + \|\tilde{v}\|_{H^1(I_{rt})} \right)  \right] \nonumber\\
		& \lesssim_{k} \|z - w \|_{\mathcal{H}_{r-t}},
		\end{align}
		where the last step is due to the case $\| z\|_{\mathcal H_{r-t}} , \| w \|_{\mathcal H_{r-t}} \leq k$.
		By following similar procedure of \eqref{lip-t2} and \eqref{lip-t3} we also get
		\begin{equation}\nonumber
		\|  \mathcal{A}_{u}( u_x,  u_x) - \mathcal{A}_{\tilde{u}}(\tilde{u}_x,  \tilde{u}_x) \|_{H^1(I_{rt})} \lesssim_{L_{\mathcal{A}},B_{\mathcal{A}}, L_{\mathcal{A}^\prime},B_{\mathcal{A}^\prime}, k_e,k} \|z - w \|_{\mathcal{H}_{r-t}}.
		\end{equation}
		Hence by substituting the estimates back in \eqref{lip-t1} we are done with \eqref{lem-lip-result} for $F_{r,k}$-term.

		Next, we move to the terms of $G_{r,k}$. As for $F_{r,k}$, it is sufficient to perform the calculations for the case $\| z\|_{\mathcal H_{r-t}} , \| w \|_{\mathcal H_{r-t}} \leq k$. By invoking Lemma \ref{hsop} followed by Remark \ref{rem-extension} we have
		\begin{align}
		\| \mathbf G_{r,k}(t,z)-\mathbf G_{r,k}(t,w) \|_{\mathscr{L}_2(\rkhs,\mathcal{H})}^2  &  \leq \| (E^1_{r-t} Y_{u}(v,  u_x)) \cdot  - (E^1_{r-t} Y_{\tilde{u}}(\tilde{v}, \tilde{u}_x)) \cdot  \|_{\mathscr{L}_2(\rkhs,H^1(\mathbb{R}))}^2  \nonumber\\
		& \leq c_{r,t} ~C_E(r,t) ~ \| Y_{u}(v, u_x)  - Y_{\tilde{u}}(\tilde{v}, \tilde{u}_x) \|_{H^1(I_{rt})}^2. \nonumber
		\end{align}
		Recall that the 1-D Sobolev embedding gives $H^1(\mathbb{R}) \hookrightarrow L^\infty(\mathbb{R})$. Consequently, by the Taylor formula \cite[Theorem 5.6.1]{Cartan_1971B} and  inequalities \eqref{y2}-\eqref{y3} we have
		\begin{align}\label{lip-t5}
		\| Y_{u}(v, \partial_ xu)  - Y_{\tilde{u}}(\tilde{v},  \tilde{u}_x) \|_{L^2(I_{rt})}^2 & \leq   \int_{I_{rt}} |Y_{u(x)}(v(x),  u_x(x))  - Y_{\tilde{u}(x)}(v(x), u_x(x)) |^2 \, dx  \nonumber\\
		& \quad + \int_{I_{rt}} |Y_{\tilde{u}(x)}(v(x), u_x(x))  - Y_{\tilde{u}(x)}(v(x),\tilde{u}_x(x)) |^2 \, dx  \nonumber\\
		& \quad + \int_{I_{rt}} |Y_{\tilde{u}(x)}(v(x), \tilde{u}_x(x))  - Y_{\tilde{u}(x)}(\tilde{v}(x),\tilde{u}_x(x)) |^2 \, dx  \nonumber\\
		& \leq C_{Y}^2 \left[ 1 + \| v\|_{H^1(I_{rt})}^2 + \| u\|_{H^1(I_{rt})}^2   \right] \| u - \tilde{u}\|_{H^2(I_{rt})} ^2 \nonumber\\
		& \quad + C_{Y_2}^2  \left[  \|  u_x - \tilde{u}_x  \|_{H^1(I_{rt})} ^2 +  \| v -\tilde{v}  \|_{H^1(I_{rt})} ^2 \right] \nonumber\\
		& \lesssim_{k, C_{Y}, C_{Y_2} } \| z-w\|_{\mathcal{H}_{r-t}}^2 .
		\end{align}
		For homogeneous part of the norm, that is $L^2$-norm of the derivative,   we have
		\begin{align}\label{lip-t6}
		& \| d_x \left[ Y_{u}(v,u_x)  - Y_{\tilde{u}}(\tilde{v},\tilde{u}_x) \right] \|_{L^2(I_{rt})}^2 \nonumber\\
		 & \lesssim \int_{I_{rt}}\sum_{i=1}^{n}\left\{ \bigg\vert \frac{\partial Y}{\partial p_i}(u(x),v(x),  u_x (x))  \frac{d u^i}{dx}(x)  - \frac{\partial Y}{\partial p_i}(\tilde{u}(x),\tilde{v}(x), \tilde{u}_x (x)) \frac{d \tilde{u}^i}{dx}(x) \bigg\vert^2 \right. \nonumber\\
		 & \quad \left. + \bigg\vert \frac{\partial Y}{\partial a_i}(u(x),v(x),  u_x (x))  \frac{d v^i}{dx}(x) - \frac{\partial Y}{\partial a_i}(\tilde{u}(x),\tilde{v}(x),  \tilde{u}_x (x)) \frac{d \tilde{v}^i}{dx}(x)  \bigg\vert^2 \right. \nonumber\\
		 & \quad \left. + \bigg\vert \frac{\partial Y}{\partial b_i}(u(x),v(x), u_x (x)) \frac{d u_x^i}{dx}(x) - \frac{\partial Y}{\partial b_i}(\tilde{u}(x),\tilde{v}(x),  \tilde{u}_x (x))  \frac{d \partial_ x \tilde{u}^i}{dx}(x)  \bigg\vert^2 \right\}  \, dx  \nonumber\\
		& =: Y_1 + Y_2 + Y_3.
		\end{align}
		We will estimate each term separately by using the 1-D Sobolev embedding, the Taylor formula and inequalities \eqref{y2}-\eqref{y4} as follows:
		\begin{align}\label{lip-t7}
			Y_1 & \lesssim \int_{I_{rt}}\sum_{i=1}^{n} \left\{  \bigg\vert \frac{\partial Y}{\partial p_i}(u(x),v(x), \ u_x (x))  \frac{d u^i}{dx}(x)  - \frac{\partial Y}{\partial p_i}(\tilde{u}(x),v(x), u_x (x)) \frac{d u^i}{dx}(x) \bigg\vert^2 \right. \nonumber\\
			& \quad \left. + \bigg\vert \frac{\partial Y}{\partial p_i}(\tilde{u}(x),v(x), u_x (x)) \frac{d u^i}{dx}(x)   - \frac{\partial Y}{\partial p_i}(\tilde{u}(x),v(x),  u_x (x)) \frac{d \tilde{u}^i}{dx}(x) \bigg\vert^2  \right. \nonumber\\
			& \quad\left. + \bigg\vert \frac{\partial Y}{\partial p_i}(\tilde{u}(x),v(x),  u_x (x)) \frac{d \tilde{u}^i}{dx}(x)    - \frac{\partial Y}{\partial p_i}(\tilde{u}(x),\tilde{v}(x),  u_x (x)) \frac{d \tilde{u}^i}{dx}(x) \bigg\vert^2  \right. \nonumber\\
			& \quad\left. + \bigg\vert \frac{\partial Y}{\partial p_i}(\tilde{u}(x),\tilde{v}(x), u_x (x)) \frac{d \tilde{u}^i}{dx}(x)   - \frac{\partial Y}{\partial p_i}(\tilde{u}(x),\tilde{v}(x),  \tilde{u}_x (x)) \frac{d \tilde{u}^i}{dx}(x) \bigg\vert^2  \right\} \, dx  \nonumber\\
			& \lesssim C_{Y_3}^2 \| u - \tilde{u} \|_{L^2(I_{rt})}^2 \|  u_x\|_{H^1(I_{rt})}^2  + C_{Y_1}^2 \left[ 1 + \| v\|_{H^1(I_{rt})}^2 + \| u_x\|_{H^1(I_{rt})}^2 \right] \|  u_x - \tilde{u}_x \|_{L^2(I_{rt})}^2 \nonumber\\
			& \quad + C_{Y_3}^2 \| v - \tilde{v} \|_{L^2(I_{rt})}^2 \| \tilde{u}_x \|_{H^1(I_{rt})}^2  + C_{Y_3}^2 \| u_x - \tilde{u}_x \|_{L^2(I_{rt})}^2 \| \tilde{u}_x \|_{H^1(I_{rt})}^2 \nonumber\\
     		& \lesssim_{k, C_{Y_2}, C_{Y_3} , C_{Y_1} } \| z-w\|_{\mathcal{H}_{r-t}}^2.
		\end{align}
		Terms $Y_2$ and $Y_3$ are quite similar so it is enough to estimate only one. For $Y_2$ we have the following calculation
		\begin{align}\label{lip-t8}
			Y_2  &  \lesssim  \int_{I_{rt}}\sum_{i=1}^{n} \left\{ \bigg\vert \frac{\partial Y}{\partial a_i}(u(x),v(x),  u_x (x))  \frac{d v^i}{dx}(x) - \frac{\partial Y}{\partial a_i}(\tilde{u}(x),v(x),  u_x (x)) \frac{d v^i}{dx}(x)  \bigg\vert^2 \, dx  \right. \nonumber\\
			& \quad \left. + \bigg\vert \frac{\partial Y}{\partial a_i}(\tilde{u}(x),v(x),  u_x (x)) \frac{d v^i}{dx}(x)  - \frac{\partial Y}{\partial a_i}(\tilde{u}(x),\tilde{v}(x),  u_x (x)) \frac{d v^i}{dx}(x)  \bigg\vert^2 \, dx  \right. \nonumber\\
			& \quad \left. + \bigg\vert \frac{\partial Y}{\partial a_i}(\tilde{u}(x),\tilde{v}(x),  u_x (x)) \frac{d v^i}{dx}(x)  - \frac{\partial Y}{\partial a_i}(\tilde{u}(x),\tilde{v}(x),  \tilde{u}_x (x)) \frac{d v^i}{dx}(x)  \bigg\vert^2 \, dx  \right. \nonumber\\
			& \quad \left. + \bigg\vert \frac{\partial Y}{\partial a_i}(\tilde{u}(x),\tilde{v}(x),  \tilde{u}_x (x)) \frac{d v^i}{dx}(x)      - \frac{\partial Y}{\partial a_i}(\tilde{u}(x),\tilde{v}(x),  \tilde{u}_x (x)) \frac{d \tilde{v}^i}{dx}(x)  \bigg\vert^2 \, dx  \right\} \nonumber\\
			& \lesssim C_{Y_3}^2 \| u - \tilde{u} \|_{H^1(I_{rt})}^2 \| v_x\|_{L^2(I_{rt})}^2 + C_{Y_3}^2 \| v - \tilde{v} \|_{H^1(I_{rt})}^2 \| v_x\|_{L^2(I_{rt})}^2 \nonumber\\
			& \quad + C_{Y_3}^2 \| u_x - \tilde{u}_x \|_{H^1(I_{rt})}^2 \| v_x\|_{L^2(I_{rt})}^2 + C_{Y_3}^2 C_{r,t}  \| v_x - \tilde{v}_x \|_{L^2(I_{rt})}^2 \nonumber\\
			& \lesssim_{k, C_{r,t}C_{Y_3}} \| z-w\|_{\mathcal{H}_{r-t}}^2.
		\end{align}
		Hence by substituting \eqref{lip-t7}-\eqref{lip-t8} into \eqref{lip-t6} we get
		\begin{equation}\nonumber
			\| d_x \left[ Y_{u}(v,u_x)  - Y_{\tilde{u}}(\tilde{v},\tilde{u}_x) \right] \|_{L^2(I_{rt})}^2 \lesssim_{k, C_{r,t},C_{Y_2}, C_{Y_3} , C_{Y_1} } \| z-w\|_{\mathcal{H}_{r-t}}^2.
		\end{equation}
		which together with \eqref{lip-t5} gives $G_{r,k}$ part of \eqref{lem-lip-result}. Hence the Lipschitz property Lemma \ref{lem-lip}.
	\end{proof}

	The following result follows directly from Lemma \ref{lem-lip} and the  standard theory of PDE via semigroup approach, refer \cite{Ball_1977} and \cite{Tayfun_2017T} for detailed proof.
	\begin{Corollary}\label{cor-existence approximate solution}
		Given any $\xi \in \mathcal H$ and $h \in \prescript{}{0}{H}^{1,2}(0,T;\rkhs)$, there exists a unique $z$ in $\mathcal{C}([0,T]; \mathcal{H})$ such that for all $t\in [0,T]$
		$$
		z(t)=S_t\xi+\int_0^tS_{t-s}\mathbf F_{r,k}(s,z(s))\,ds+\int_0^tS_{t-s}( \mathbf G_{r,k}(s,z(s))  \dot{h}(s))  \,ds.
		$$
	\end{Corollary}
	\begin{Remark}\label{rem-habuse}
		Here by $\mathbf G_{r,k}(s,z(s))  \dot{h}(s)$ we understand that both components of $\mathbf G_{r,k}(s,z(s))$ are acting on $\dot{h}(s)$.
	\end{Remark}
	From now on, for each $r>R+T$  and $k \in \mathbb{N}$, the solution from Corollary \ref{cor-existence approximate solution} will be denoted by $z_{r,k}$ and called  the \textit{approximate solution}. To proceed further we define the following two auxiliary functions
	\begin{eqnarray*}
		\widetilde F_{r,k}&:&[0,T]\times\mathcal H \ni (t,z)\mapsto
		\left(\begin{array}{c}0\\\varphi\cdot \h^\prime(u)\mathbf F^2_{r,k}(t,z)+\varphi B_u(v,v)-\varphi B_u(u_x,u_x)\end{array}\right) \\
		&& \qquad\qquad \qquad\qquad \quad - \left(\begin{array}{c}0\\ \Delta\varphi\cdot h(u)+2\varphi_x\cdot h^\prime(u)u_x \end{array}\right)  \in \mathcal H,
		\\
		\text{and} &&
		\\
		\widetilde G_{r,k}&:&[0,T]\times\mathcal H \ni (t,z) \mapsto \left(\begin{array}{c}0\\\varphi\cdot \h^\prime(u)\mathbf G^2_{r,k}(t,z)\end{array}\right) \in \mathcal H.
	\end{eqnarray*}
	Here $\mathbf{F}_{r,k}^2(s,z_{r,k}(s))$ and $\mathbf{G}_{r,k}^2(s,z_{r,k}(s))$ denote the second components of the  vectors $\mathbf{F}_{r,k}(s,z_{r,k}(s))$ and $\mathbf{G}_{r,k}(s,z_{r,k}(s))$, respectively.
	The following corollary relates the solution $z_{r,k}$ with its transformation under the map $\mathbf{Q}_r$ and allow to understand the need of the functions $\widetilde F_{r,k}$ and $\widetilde G_{r,k}$.
	\begin{Corollary}\label{cor-ZTildaSoln}	
		Let us assume that $\xi : =(E^2_ru_0,E^1_rv_0)$ and that  $z_{r,k} \in \mathcal{C}([0,T]; \mathcal{H})$ satisfies
		\begin{equation}\label{eqn_local}
			z_{r,k}(t)=S_t\xi+\int_0^tS_{t-s}\mathbf F_{r,k}(s,z_{r,k}(s))\,ds+\int_0^tS_{t-s} (\mathbf G_{r,k}(s,z_{r,k}(s)) \dot{h}(s) ) \,ds,\qquad t \in [0,T].
		\end{equation}
		Then $\widetilde z_{r,k}=\mathbf Q_r(z_{r,k})$ satisfies, for each $t \in [0,T]$,
		$$
		\widetilde z_{r,k}(t)=S_t\mathbf Q_r(\xi)+\int_0^tS_{t-s}\widetilde F_{r,k}(s,z_{r,k}(s))\,ds+\int_0^tS_{t-s} ( \widetilde G_{r,k}(s,z_{r,k}(s)) \dot{h}(s) ) \,ds.
		$$
	\end{Corollary}
	\begin{proof}[\textbf{Proof of Corollary \ref{cor-ZTildaSoln}}]
		First observe that by the action of $\mathbf Q_r^\prime$ and $\mathcal{G}$ on the elements of $\mathcal{H}$ from Lemma \ref{c2q} and \eqref{gr}, respectively,  we get
		\begin{align}\label{cor-ZTildaSoln-t1}
		\mathbf{Q}_r^\prime(z_{r,k}(s)) & \left( \mathbf F_{r,k}(s,z_{r,k}(s)) +  \mathbf G_{r,k}(s,z_{r,k}(s)) \dot{h}(s) \right)  \nonumber\\
		& \quad = \left(\begin{array}{c} 0 \\ \varphi \cdot \left\{  [\h^\prime(u_{r,k}(s))] (\mathbf F_{r,k}^2(s,z_{r,k}(s))) + [\h^\prime(u_{r,k}(s))] (\mathbf G_{r,k}^2(s,z_{r,k}(s)) \dot{h}(s)) \right\} \end{array}\right).
		\end{align}
		Moreover, since by applying Lemma \ref{c2q} and \eqref{gr} to $z =(u,v) \in \mathcal{H}$ we have
		\begin{align}\label{cor-ZTildaSoln-t2}
		&  F(z) := 	\mathbf Q_r^\prime \mathcal{G} z - \mathcal{G} \mathbf Q_r z = \left(\begin{array}{c} \varphi \cdot [\h^\prime(u)](v) \\ \varphi \cdot \left\{ [\h^{\prime\prime}(u)](v,v) + [\h^{\prime}(u)](u^{\prime\prime}) \right\} \end{array}\right) \nonumber\\
		& \quad - \left(\begin{array}{c} \varphi \cdot [\h^\prime(u)](v) \\  \varphi^{\prime\prime} \cdot \h(u) + 2  \varphi^\prime \cdot [\h^\prime(u)](u^\prime) + \varphi \cdot  [\h^\prime(u)](u^{\prime\prime})  + \varphi \cdot [\h^{\prime\prime}(u)](u^\prime ,u^\prime) \end{array}\right),
		\end{align}
		substitution $z = z_{r,k}(s) = (u_{r,k}(s), v_{r,k}(s)) \in \mathcal{H}$  in \eqref{cor-ZTildaSoln-t2} with \eqref{cor-ZTildaSoln-t1} followed by definition  \eqref{Boper}  gives, for $s \in [0,T] $,
		\begin{align}
		& \mathbf{Q}_r^\prime(z_{r,k}(s))  \left( \mathbf F_{r,k}(s,z_{r,k}(s)) +  \mathbf G_{r,k}(s,z_{r,k}(s)) \right) + F(z_{r,k}(s))  \nonumber\\
		&  = \left(\begin{array}{c} 0 \\  \varphi \cdot  [\h^\prime(u_{r,k}(s))] (\mathbf F_{r,k}^2(s,z_{r,k}(s)))  + \varphi \cdot [\h^{\prime\prime}(u_{r,k}(s))](v_{r,k}(s),v_{r,k}(s))  \\  -  \varphi \cdot [\h^{\prime\prime}(u_{r,k}(s))](\partial_x u_{r,k}(s),\partial_x u_{r,k}(s))   \end{array}\right) \nonumber\\
		&    -  \left(\begin{array}{c} 0 \\  - \varphi^{\prime\prime} \cdot \h(u_{r,k}(s)) + 2 \varphi^\prime \cdot [\h^\prime(u_{r,k}(s))](\partial_x u_{r,k}(s)) + \varphi \cdot [\h^\prime(u_{r,k}(s))] (\mathbf G_{r,k}^2(s,z_{r,k}(s))) \end{array}\right) \nonumber\\
		&  = \widetilde F_{r,k}(s,z_{r,k}(s)) +  \widetilde G_{r,k}(s,z_{r,k}(s)). \nonumber
		\end{align}
		Hence, if we have
		\begin{equation}\label{cor-ZTildaSoln-t3}
			\int_{0}^{T} \left[  \| \mathbf F_{r,k}(s,z_{r,k}(s)) \|_{\mathcal{H}} + \| \mathbf G_{r,k}(s,z_{r,k}(s)) \dot{h}(s)\|_{\mathcal{H}} \right] \, ds < \infty,
		\end{equation}
		then by invoking \cite[Lemma 6.4]{Brz+Ondr_2007} with
		$$ L = \mathbf Q_r,  K = U = \mathcal{H}, A =B = \mathcal{G}, g(s) = 0, f(s) = \mathbf F_{r,k}(s,z_{r,k}(s)) + \mathbf G_{r,k}(s,z_{r,k}(s)) \dot{h}(s),$$
		we are done with the proof here. But \eqref{cor-ZTildaSoln-t3} follows by Lemma \ref{lem-lip}, because $h \in \prescript{}{0}{H}^{1,2}(0,T;\rkhs)$ and the following  holds, due to the H\"older inequality with the abuse of notation as mentioned in Remark \ref{rem-habuse},
		\begin{align}
			& \int_{0}^{T}  \| \mathbf G_{r,k}(s,z_{r,k}(s)) \dot{h}(s)\|_{\mathcal{H}}\, ds = \int_{0}^{T}  \| \mathbf G_{r,k}^2(s,z_{r,k}(s)) \dot{h}(s)\|_{H^1(\mathbb{R})}\, ds \nonumber \\
			 & \hspace{3cm} \leq \left( \int_{0}^{T}  \| (\mathbf G_{r,k}^2(s,z_{r,k}(s)) )  \cdot\|_{\mathscr{L}_2(\rkhs, H^1(\mathbb{R}))}^2 \, ds \right)^{\frac{1}{2}} \left( \int_{0}^{T}  \| \dot{h}(s)\|_{\rkhs}^2 \, ds \right)^{\frac{1}{2}}  . \nonumber
		\end{align}
	\end{proof}
	Next we prove that the approximate solution $z_{r,k}$ stays on the manifold. Define the following three positive reals: for each $r >R+T$ and $k \in \mathbb{N}$,
	\begin{equation}
		\label{stoppingTimes}
		\left\{
		\begin{aligned}
			&   \tau_k^1 := \inf\,\{t\in[0,T]: \| z_{r,k}(t)\|_{\mathcal H_{r-t}}\ge k\},   \\
			& \tau_k^2 := \inf\,\{t\in[0,T]: \| \widetilde z_{r,k}(t)\|_{\mathcal H_{r-t}}\ge k\}, \\
			& \tau_k^3 : = \inf\,\{t\in[0,T]: \exists x,\,|x|\le r-t,\,u_{r,k}(t,x)\notin O\},  \\
			& \tau_k  :=\tau_k^1\land\tau_k^2\land\tau_k^3.
		\end{aligned}\right.
	\end{equation}
	Also, define the following  $\mathcal{H}$-valued functions of time $t \in [0,T]$
	\begin{align}\label{defn-aProcesses}
	& a_k(t)=S_t\xi+\int_0^tS_{t-s} \mathbbm{1}_{[0,\tau_k)}(s)\mathbf F_{r,k}(s,z_{r,k}(s))\,ds \nonumber\\
	& \qquad \qquad + \int_0^tS_{t-s}( \mathbbm{1}_{[0,\tau_k)}(s)\mathbf G_{r,k}(s,z_{r,k}(s)) \dot{h}(s)) \,ds, \nonumber\\
	& \widetilde a_k(t)= S_t\mathbf Q_r(\xi)+\int_0^tS_{t-s}\mathbbm{1}_{[0,\tau_k)}(s)\widetilde F_{r,k}(s,z_{r,k}(s))\,ds \nonumber\\
	& \qquad \qquad + \int_0^tS_{t-s}(\mathbbm{1}_{[0,\tau_k)}(s)\widetilde G_{r,k}(s,z_{r,k}(s))  \dot{h}(s) ) \,ds .
	\end{align}
	\begin{Proposition}\label{prop-coincidence}
		For each $k \in \mathbb{N}$ and $\xi : =(E^2_ru_0,E^1_rv_0)$, the functions $a_k$, $\widetilde a_k$, $z_{r,k}$ and $\widetilde z_{r,k}$ coincide on $[0,\tau_k)$. In particular, $u_{r,k}(t,x)\in M$ for $|x|\le r-t$ and $t\le\tau_k$.  Consequently, $\tau_k=\tau_k^1=\tau_k^2\le\tau_k^3$.
	\end{Proposition}
	\begin{proof}[\textbf{Proof of Proposition \ref{prop-coincidence}}]
		Let us fix $k$. First note that, due to indicator function,
		\begin{equation}\label{prop-coincidence-t0}
			a_k=z_{r,k} \qquad \textrm{ and } \qquad \widetilde a_k=\widetilde z_{r,k} \textrm{ on } [0,\tau_k).
		\end{equation}
		Next, since $E_{r-s}^1f = f$ on $|x|\leq r-s$, see Proposition \ref{prop-extension},  and $\varphi=1$ on $(-r,r)$, by Lemma \ref{lem_consistency} followed by  \eqref{eqn-invariance-A} we infer that
		\begin{equation}\label{prop-coincidence-t1}
		\left\{
		\begin{aligned}
		&   \mathbbm{1}_{[0,\tau_k)}(s)[\widetilde F_{r,k}(s,z_{r,k}(s))](x)=\mathbbm{1}_{[0,\tau_k)}(s)[\mathbf F_{r,k}(s,\widetilde z_{r,k}(s))](x),  \\
		& \mathbbm{1}_{[0,\tau_k)}(s)[\widetilde G_{r,k}(s,z_{r,k}(s))e](x)=\mathbbm{1}_{[0,\tau_k)}(s)[\mathbf G_{r,k}(s,\widetilde z_{r,k}(s))e](x), \quad e\in K,
		\end{aligned}\right.
		\end{equation}
		holds for every $|x|\le r-s$, $0\le s\le T$.
		Now we claim that if we denote
		$$ p(t) : = \frac{1}{2} \| a_k(t)-\widetilde a_k(t)\|_{\mathcal H_{r-t}}^2, $$
		then the map $s\mapsto p(s\land\tau_k)$ is continuous and uniformly bounded. Indeed, since, by Proposition \ref{prop-extension},  $\xi(x) = (u_0(x), v_0(x)) \in TM$ for $|x| \leq r$, the  uniform boundedness is an easy consequence of bound property of $C_0$-group, Lemmata \ref{lem_consistency}  and \ref{lem-lip}. Continuity of $s\mapsto p(s\land\tau_k)$ follows from the following:
		\begin{enumerate}
			\item  for every $z \in \mathcal{H}$, the map $t \mapsto \| z\|_{\mathcal{H}_{r-t}}^2$ is continuous;
			\item  for each  $t$, the map
			\[L^2(\mathbb{R}) \ni u \mapsto \int_{0}^{t} |u(s)|^2 \, ds \in \mathbb{R}, \]
			 is locally Lipschitz.
		\end{enumerate}
		Now observe that by applying Proposition \ref{prop_magic} for
		$$ k =1,L = I, T = r, x=0 \quad \textrm{and} \quad z(t) = (u(t), v(t)) : = a_k(t) - \widetilde{a}_k(t), $$
		we get $\mathbf{e}(t,r;0,z(t)) = p(t)$,  and the following
		\begin{align}\label{prop-coincidence-t1a-2}
		\mathbf e(t,r;0,(t)) & \leq  \mathbf e(0,r;0,z_0)+\int_0^t V(s,z(s))\,ds.
		\end{align}
		Here
		\begin{align}
		V(t,z(t)) & : = \langle 	u(t),v(t) \rangle_{L^2(B_{r-t})}+ \langle  v(t), f(t)\rangle_{L^2(B_{r-t})}  + \langle  \partial_x v(t), \partial_x f(t)\rangle_{L^2(B_{r-t})}   \nonumber\\
		&    \qquad + \langle 	 v(t), g(t) \rangle_{L^2(B_{r-t})}  + \langle 	\partial_x  v(t), \partial_x g(t) \rangle_{L^2(B_{r-t})}, \nonumber
		\end{align}		
		and
		\begin{align}
		& \left(\begin{array}{c} 0\\ f(t) \end{array}\right) : = \mathbbm{1}_{[0,\tau_k)}(t) [\mathbf F_{r,k}(t,z_{r,k}(t)) - \widetilde{F}_{r,k}(t,z_{r,k}(t))], \nonumber\\
		& \left(\begin{array}{c} 0\\ g(t) \end{array}\right) : = \mathbbm{1}_{[0,\tau_k)}(t) [\mathbf G_{r,k}(t,z_{r,k}(t)) \dot{h}(t) - \widetilde{G}_{r,k}(t,z_{r,k}(t)) \dot{h}(t) ]. \nonumber
		\end{align}
		Due to the extension operators $E_r^2$  and $E_r^1$ the initial data $\xi$ in the definition \eqref{defn-aProcesses} satisfies the assumption of Lemma \ref{lem_consistency}, $S_t\mathbf Q_r(\xi) = S_t \xi$, and so
		$\mathbf{e}(0,0;0,z(0)) = p(0)=0$. Next observe that by the Cauchy-Schwarz inequality  we have
		\begin{align}
		V(t,z(t)) & \leq  \frac{1}{2}  \|u(t)\|_{L^2(B_{r-t})}^2 +  \frac{3}{2}\|v(t)\|_{L^2(B_{r-t})}^2 + \frac{1}{2} \|f(t)\|_{L^2(B_{r-t})}^2 +  \|\partial_x v(t)\|_{L^2(B_{r-t})}^2 \nonumber\\
		& \quad  + \frac{1}{2}  \|\partial_x f(t)\|_{L^2(B_{r-t})}^2 + \frac{1}{2}\|g(t)\|_{L^2(B_{r-t})}^2+ \frac{1}{2}\|\partial_x g(t)\|_{L^2(B_{r-t})}^2  \nonumber\\
		& \leq  3 p(t) + \frac{1}{2}  \|f(t)\|_{H^1(B_{r-t})}^2   + \frac{1}{2} \|g(t) \|_{H^1(B_{r-t})}^2. \nonumber
		\end{align}
		By using above into \eqref{prop-coincidence-t1a-2} and, then, by invoking equalities \eqref{prop-coincidence-t1} and \eqref{prop-coincidence-t0}, definition \eqref{stoppingTimes},  Lemma \ref{hsop} and Lemma \ref{lem-lip}    we have the following calculation, for every $t \in [0,T]$,
		\begin{align}\label{prop-coincidence-t1a-6}
		p(t) & \leq \int_{0}^{t} 3p(s)\, ds + \frac{1}{2} \int_{0}^{t} \mathbbm{1}_{[0,\tau_k)}(s) \|\mathbf F_{r,k}^2(s,z_{r,k}(s)) - \mathbf{F}_{r,k}^2(s,\tilde{z}_{r,k}(s)) \|_{H^1(B_{r-s})}^2 \, ds  \nonumber\\
		& \quad + \frac{1}{2} \int_{0}^{t} \mathbbm{1}_{[0,\tau_k)}(s) \|\mathbf G_{r,k}^2(s,z_{r,k}(s)) - \mathbf{G}_{r,k}^2(s,\tilde{z}_{r,k}(s)) \|_{\mathscr{L}_2(\rkhs, H^1(B_{r-s}))}^2  \|\dot{h}(s)\|_{\rkhs}^2 \, ds \nonumber\\
		& \leq 3 \int_{0}^{t}  p(s)\, ds + \frac{1}{2} C_{r,k}^2  \int_{0}^{t} \mathbbm{1}_{[0,\tau_k)}(s) \|z_{r,k}(s) - \widetilde{z}_{r,k}(s) \|_{\mathcal{H}_{r-s}}^2 \, ds  \nonumber\\
		& \quad + \frac{1}{2}C_{r,k}^2 \int_{0}^{t} \mathbbm{1}_{[0,\tau_k)}(s) \|z_{r,k}(s) - \widetilde{z}_{r,k}(s) \|_{\mathcal{H}_{r-s}}^2  \|\dot{h}(s)\|_{\rkhs}^2 \, ds  \nonumber\\
		& \leq (3+C_{r,k}^2 ) \int_{0}^{t}   p(s) (1+ \|\dot{h}(s)\|_{\rkhs}^2) \, ds.
		\end{align}
		Consequently by the Gronwall Lemma, for $t \in [0,\tau_k]$,
		\begin{equation}\label{prop-coincidence-t1a-8}
			p(t) \lesssim_{C_{r,k}} p(0)  \exp\left[ \int_{0}^{t} (1+ \|\dot{h}(s)\|_{\rkhs}^2) \, ds  \right].
		\end{equation}
		Note that the right hand side in \eqref{prop-coincidence-t1a-8} is finite because $h \in \prescript{}{0}{H}^{1,2}(0,T;\rkhs)$.  Since we know that $p(0)=0$ we arrive to  $p(t)=0$ on $ t \in [0,\tau_k]$ . This further implies that $a_k(t,x)=\widetilde{a}_k(t,x)$ hold for $|x|\le r-t$ and $t\le\tau_k$. Consequently, $z_{r,k}(t,x) = \widetilde{z}_{r,k}(t,x)$ hold for $|x|\le r-t$ and $t\le\tau_k$. So, because $\widetilde{z}_{r,k}(t,x) = \mathbf{Q}_r(z_{r,k}(t))$ and $\varphi=1$ on $(-r,r)$,
		\begin{equation}\label{prop-coincidence-t2}
			u_{r,k}(t,x) = \h(u_{r,k}(t,x)), \qquad \textrm{ for } |x|\leq r-t , \quad   t\leq \tau_k.
		\end{equation}
		Since, by definition \eqref{stoppingTimes} of $\tau_k$,  $u_{r,k}(t,x)\in O$, equality \eqref{prop-coincidence-t2} and Lemma \ref{cor_inv}, gives $u_{r,k}(t,x)\in M$ for $|x|\le r-t$ and $t\le\tau_k$. This suggests that $\tau_k\le\tau_k^3$ and  hence $\tau_k=\tau_k^1\land\tau_k^2$. It remains to show that $\tau_k^1=\tau_k^2$. But suppose it does not hold and without loss of generality we assume that $\tau_k^1 > \tau_k^2 $.   Then by definition \eqref{stoppingTimes} and the continuity of $z_{r,k}$ and $\tilde{z}_{r,k}$ in time we have
		\begin{equation}\nonumber
		\| z_{r,k}(\tau_k^2,\cdot)\|_{\mathcal{H}_{r-\tau_k^2}} < k \quad \textrm{ but } \quad \| \widetilde{z}_{r,k}(\tau_k^2,\cdot) \|_{\mathcal{H}_{r-\tau_k^2}} \geq k,
		\end{equation}
		which contradicts the above mentioned consequence of  $p=0$ on $[0,\tau_k]$. Hence we conclude that $\tau_k^1=\tau_k^2$ and this finishes the proof of Proposition  \ref{prop-coincidence}.
	\end{proof}
	Next in the ongoing proof of Theorem \ref{thm-skeleton} we show that the approximate solutions extend each other.  Recall that $r> R + T$ is fixed for given $T >0$.
	\begin{Lemma}\label{ndcr}
		Let $k\in\mathbb N$ and $\xi= (E_r^2u_0, E_r^1 v_0)$. Then $z_{r,k+1}(t,x)=z_{r,k}(t,x)$ on $|x|\le r-t$, $t\le\tau_k$, and $\tau_k \le\tau_{k+1}$.
	\end{Lemma}
	\begin{proof}[\textbf{Proof of Lemma \ref{ndcr}}] Define $$p(t) : = \frac{1}{2} \Vert a_{k+1}(t)-a_k(t)\Vert^2_{H^1(B_{r-t}) \times L^2(B_{r-t})}. $$
		As an application of Proposition \ref{prop_magic}, by performing the  computation based on \eqref{prop-coincidence-t1a-2} -  \eqref{prop-coincidence-t1a-6}, with $k =0$ and rest the same, we obtain
		\begin{align}\label{ndcr-t1}
		p(t) & \le 2 \int_0^tp(s)\,ds+ \frac{1}{2} \int_0^t\Vert \mathbbm{1}_{[0,\tau_{k+1})}(s)\mathbf F^2_r(s,z_{r,k+1}(s))-\mathbbm{1}_{[0,\tau_k)}(s)\mathbf F^2_r(s,z_{r,k}(s))\Vert^2_{L^2(B_{r-s})}\,ds \nonumber\\
		& \quad +\frac{1}{2} \int_0^t\Vert \mathbbm{1}_{[0,\tau_{k+1})}(s)\mathbf G^2_r(s,z_{r,k+1}(s)) \dot{h}(s) -\mathbbm{1}_{[0,\tau_k)}(s)\mathbf G^2_r(s,z_{r,k}(s)) \dot{h}(s)\Vert^2_{L^2(B_{r-s})}\,ds.
		\end{align}
		Then,  since $F_r$ and $G_r$ depends on  $u_{r,k}(s)$, $u_{r,k+1}(s)$ and their first partial derivatives, with respect to time $t$ and space $x$, which are actually bounded on the interval $(-(r-s),r-s)$ by some constant $C_r$ for every $s<\tau_{k+1}\land\tau_k$, by evaluating \eqref{ndcr-t1} on $t\land\tau_{k+1}\land\tau_k$ following the use of Lemmata \ref{lem-lip} and \ref{hsop} we get
		\begin{align}
		p(t\land\tau_{k+1}\land\tau_k) & \leq 2 \int_0^tp(s\land\tau_{k+1}\land\tau_k)\,ds \nonumber\\
		& \quad + \frac{1}{2} \int_0^{t\land\tau_{k+1}\land\tau_k }\Vert \mathbf F^2_r(s,z_{r,k+1}(s))-\mathbf F^2_r(s,z_{r,k}(s))\Vert^2_{L^2(B_{r-s})}\,ds \nonumber\\
		& \quad + \frac{1}{2} \int_0^{t\land\tau_{k+1}\land\tau_k}\Vert \mathbf G^2_r(s,z_{r,k+1}(s))\zeta(s)-\mathbf G^2_r(s,z_{r,k}(s))\dot{h}(s)\Vert^2_{L^2(B_{r-s})}\,ds \nonumber \\
		& \lesssim_k \int_0^tp(s\land\tau_{k+1}\land\tau_k) (1+ \| \dot{h}(s)\|_{\rkhs}^2) \,ds.  \nonumber
		\end{align}
		Hence by the Gronwall Lemma we infer that $p=0$ on $[0,\tau_{k+1}\land\tau_k]$.

		Consequently, we claim that $\tau_k\le\tau_{k+1}$.  We divide the proof of our claim in the following three exhaustive subcases. Due to \eqref{stoppingTimes}, the subcases when $\|\xi\|_{\mathcal{H}_r} > k+1$ and $k  < \|\xi\|_{\mathcal{H}_r} \leq k+1$ are trivial. In the last subcase when $\|\xi\|_{\mathcal{H}_r} \leq k$ we  prove the claim $\tau_k\le\tau_{k+1}$ by the method of contradiction, and so assume that $\tau_k > \tau_{k+1}$ is true.  Then, because of continuity in time of $z_{r,k}$ and $z_{r,k+1}$,  by \eqref{stoppingTimes} we have
		\begin{equation}\label{ndcr-t3}
		\| z_{r,k}(\tau_{k+1})\|_{\mathcal{H}_{r-\tau_{k+1}}} < k \quad \textrm{ and } \quad \| z_{r,k+1}(\tau_{k+1})\|_{\mathcal{H}_{r-\tau_{k+1}}} \geq k.
		\end{equation}
		However, since $p(t) = 0$ for  $t \in [0,\tau_{k+1}\land\tau_k]$  and $(u_0(x),v_0(x)) \in TM$ for $|x|< r$, by argument based on the one made  after \eqref{prop-coincidence-t1a-8},  in the Proposition \ref{prop-coincidence}, we get $z_{r,k}(t,x) =z_{r,k+1}(t,x)$ for every $t \in [0,\tau_{k+1}]$ and $|x|\leq r-t$. But this contradicts \eqref{ndcr-t3} and we finish the proof of our claim and, in result, the proof of Lemma \ref{ndcr}.
	\end{proof}
	Since by definition \eqref{stoppingTimes} and Lemma \ref{ndcr} the sequence of stopping times $\{\tau_k\}_{k  \geq 1}$ is bounded and non-decreasing, it makes sense to denote by $\tau$ the limit of $\{\tau_k\}_{k  \geq 1}$. Now by using \cite[Lemma 10.1]{Brz+Ondr_2007}, we prove that the approximate solutions do not explode which is same as the following in terms of $\tau$.
	\begin{Proposition}\label{nonexp}
		For $\tau_k$ defined in \eqref{stoppingTimes}, $\tau := \lim\limits_{k \to \infty} \tau_k = T$.
	\end{Proposition}
	\begin{proof} [\textbf{Proof of Proposition \ref{nonexp}}]
		We  first notice  that  by a particular case of the Chojnowska-Michalik Theorem \cite{Chojn-M_1979}, when the diffusion coefficient is absent, we have that for each $k$ the  approximate solution $z_{r,k}$, as a function of time $t$, is $H^1(\mathbb{R};\mathbb{R}^n) \times L^2(\mathbb{R};\mathbb{R}^n)$-valued and  satisfies
		\begin{equation}\label{vcfa}
			z_{r,k}(t)=\xi+\int_0^t\mathcal Gz_{r,k}(s)\,ds+\int_0^t\mathbf F_{r,k}(s,z_{r,k}(s))\,ds+\int_0^t\mathbf G_{r,k}(s,z_{r,k}(s))  \dot{h}(s)\,ds,
		\end{equation}
		for $t \leq T$. In particular,
		$$
			u_{r,k}(t)=\xi_1+\int_0^tv_{r,k}(s)\,ds,
		$$
		for $t\le T$, where $\xi_1 = E_r^2 u_0$ and  the integral converges in $H^1(\mathbb{R};\mathbb{R}^n)$. Hence
		$$\partial_tu_{r,k}(s,x)=v_{r,k}(s,x), \qquad \textrm{for all} \quad s\in [0,T], x \in\mathbb{R}. $$

		Next, by keeping in mind the Proposition \ref{prop-coincidence}, we set
		\begin{equation}\nonumber
			l(t) : =\Vert a_k(t)\Vert_{H^1(B_{r-t})\times L^2(B_{r-t})}^2 \quad \textrm{ and } \quad q(t) : =\log(1+\Vert a_k(t)\Vert_{\mathcal H_{r-t}}^2).
		\end{equation}
		By applying  Proposition \ref{prop_magic}, respectively, with $ k =0,1$ and $L(x) = x, \log(1+x)$,  followed by the use of Lemma \ref{lem-lip} we get
		\begin{align}\label{sob1}
			l(t) &\leq  l(0)+\int_0^tl(s)\,ds + \int_0^t\mathbbm{1}_{[0,\tau_k]}(s)\langle v_{r,k}(s),\varphi(s)\rangle_{L^2(B_{r-s})}\,ds	\nonumber	
			\\
			& \quad + \int_0^t\mathbbm{1}_{[0,\tau_k]}(s)\langle v_{r,k}(s),\psi(s) \rangle_{L^2(B_{r-s})}\,ds,
		\end{align}
		and
		\begin{align}\label{sob2}
			q(t) &\leq  q(0)+\int_0^t\frac{\Vert a_k(s)\Vert_{\mathcal H_{r-s}}^2}{1+\Vert a_k(s)\Vert_{\mathcal H_{r-s}}^2}\,ds \nonumber		\\
			& \quad +\int_0^t\mathbbm{1}_{[0,\tau_k]}(s)\frac{\langle v_{r,k}(s), \varphi(s) \rangle_{L^2(B_{r-s})}}{1+\Vert a_k(s)\Vert_{\mathcal H_{r-s}}^2}\,ds
			+ \int_0^t\mathbbm{1}_{[0,\tau_k]}(s)\frac{\langle\partial_xv_{r,k}(s),\partial_x[\varphi(s) ]\rangle_{L^2(B_{r-s})}}{1+\Vert a_k(s)\Vert_{\mathcal H_{r-s}}^2}\,ds
			\nonumber \\
			& \quad +\int_0^t\mathbbm{1}_{[0,\tau_k]}(s)\frac{\langle v_{r,k}(s), \psi(s) \rangle_{L^2(B_{r-s})}}{1+\Vert a_k(s)\Vert_{\mathcal H_{r-s}}^2}\,ds+  \int_0^t\mathbbm{1}_{[0,\tau_k]}(s)\frac{\langle\partial_xv_{r,k}(s),\partial_x[\psi(s) ]\rangle_{L^2(B_{r-s})}}{1+\Vert a_k(s)\Vert_{\mathcal H_{r-s}}^2}\,ds.			
		\end{align}
		Here
		\begin{align}
			& \varphi(s) := \mathcal A_{u_{r,k}(s)}(v_{r,k}(s),v_{r,k}(s))-\mathcal A_{u_{r,k}(s)}(\partial_x u_{r,k}(s),\partial_x u_{r,k}(s)) , \nonumber\\
			& \psi(s) := Y_{u_{r,k}(s)} (\partial_t u_{r,k}(s), \partial_x u_{r,k}(s))  \dot{h} (s). \nonumber
		\end{align}
		Since by Proposition \ref{prop-coincidence} $u_{r,k}(s,x)\in M$ for $|x|\le r-s$ and $s\le\tau_k$, we have
		$$ u_{r,k}(s,x)\in M \quad \textrm{and} \quad \partial_tu_{r,k}(s,x) =v_{r,k}(s,x)\in T_{u_{r,k}(s,x)}M, $$
		on the mentioned domain of $s$ and $x$.  Consequently, by Proposition \ref{sft}, we get
		\begin{align}\label{nonexp-t1}
		& \mathcal A_{u_{r,k}(s,x)}(v_{r,k}(s,x),v_{r,k}(s,x))=A_{u_{r,k}(s,x)}(v_{r,k}(s,x),v_{r,k}(s,x)), \\
		& \mathcal A_{u_{r,k}(s,x)} (\partial_xu_{r,k}(s,x),\partial_xu_{r,k}(s,x)) = 		A_{u_{r,k}(s,x)}(\partial_xu_{r,k}(s,x),\partial_xu_{r,k}(s,x)), \nonumber
		\end{align}
		on $|x|\le r-s$ and $s\le\tau_k$. Hence, since $v_{r,k}(s,x)\in T_{u_{r,k}(s,x)}M$, and by definition, $A_{u_{r,k}(s,x)} \in N_{u_{r,k}(s,x)}M$, the $L^2$-inner product on domain $B_{r-s}$ vanishes and, in result, the second integrals in \eqref{sob1} and \eqref{sob2} are equal to zero.

		Next, to deal with the integral containing terms $\psi$, we follow Lemma \ref{lem-lip} and we invoke Lemma \ref{hsop}, estimate \eqref{y00}, and Proposition \ref{prop-coincidence} to get
		\begin{align}\label{nonexp-t3}
		&  \langle v_{r,k}(s), Y_{u_{r,k}(s)} (\partial_t u_{r,k}(s), \partial_x u_{r,k}(s))  \dot{h} (s) \rangle_{L^2(B_{r-s})}  \nonumber\\
		& \lesssim \| v_{r,k}(s)\|_{L^2(B_{r-s})}^2 +    \| Y_{u_{r,k}(s)} (\partial_t u_{r,k}(s), \partial_x u_{r,k}(s))  \dot{h} (s) \|_{L^2(B_{r-s})}^2  \nonumber\\
		&  \leq  \| v_{r,k}(s)\|_{L^2(B_{r-s})}^2 + C_{Y_0}^2 C_r^2 \left( 1+ \| v_{r,k}(s)\|_{L^2(B_{r-s})}^2 + \| \partial_x u_{r,k}(s) \|_{L^2(B_{r-s})}^2  \right)\| \dot{h}(s) \|_{\rkhs}^2\nonumber\\
		& \lesssim (1+l(s)) (1+ \| \dot{h}(s)\|_{\rkhs}^2 ),
		\end{align}
		for some $C_r >0$, and estimates \eqref{y2}-\eqref{y3} yields
		\begin{align}\label{nonexp-t3i}
		&  \langle v_{r,k}(s), Y_{u_{r,k}(s)} (\partial_t u_{r,k}(s), \partial_x u_{r,k}(s))  \dot{h} (s) \rangle_{L^2(B_{r-s})}  \nonumber\\
		&  \qquad +  \langle \partial_x v_{r,k}(s), \partial_x [Y_{u_{r,k}(s)} (\partial_t u_{r,k}(s), \partial_x u_{r,k}(s))  \dot{h} (s)] \rangle_{L^2(B_{r-s})}  \nonumber\\
		& \lesssim \| v_{r,k}(s)\|_{H^1(B_{r-s})}^2 +    \| Y_{u_{r,k}(s)} (\partial_t u_{r,k}(s), \partial_x u_{r,k}(s))  \dot{h} (s) \|_{H^1(B_{r-s})}^2  \nonumber\\
		&  \leq \| v_{r,k}(s)\|_{H^1(B_{r-s})}^2 +  \| \dot{h}(s) \|_{\rkhs}^2  \left[ C_{Y_0}^2 C_r^2 \left( 1+ \| v_{r,k}(s)\|_{L^2(B_{r-s})}^2 + \| \partial_x u_{r,k}(s) \|_{L^2(B_{r-s})}^2  \right) \right. \nonumber\\
		& \quad \left. + C_{Y_1}^2  \left( 1+ \| v_{r,k}(s)\|_{H^1(B_{r-s})}^2 + \| \partial_x u_{r,k}(s) \|_{H^1(B_{r-s})}^2  \right) \| u_{r,k}(s)\|_{H^1(B_{r-s})}^2 \right. \nonumber\\
		& \quad \left.  + C_{Y_2}^2  \left( \| v_{r,k}(s)\|_{L^2(B_{r-s})}^2 + \| \partial_x u_{r,k}(s) \|_{L^2(B_{r-s})}^2  \right) \right]  \nonumber\\
		& \lesssim_{C_r, C_{Y_i}} ~ (1+l(s))~ (1+ \|a_k(s)\|_{\mathcal{H}_{r-s}}^2)  (1+ \| \dot{h}(s)\|_{\rkhs}^2 ), \quad i=0,1,2.
		\end{align}
		By substituting the estimates \eqref{nonexp-t1} and  \eqref{nonexp-t3} in the inequality \eqref{sob1} we get
		\begin{equation}\label{nonexp-t3a}
			l(t)  \lesssim  l(0)+\int_0^t \mathbbm{1}_{[0,\tau_k]}(s) (1+l(s)) ~ (1+\| \dot{h}(s)\|_{\rkhs}^2) \,ds.
		\end{equation}
		Now we define $S_j$ as the set of initial data whose norm under extension is bounded by $j$, in precise,  $$S_j := \{ (u_0,v_0) \in \mathcal{H}_{\textrm{loc}} :  ~ \|\xi \|_{\mathcal{H}_r}\leq j  \textrm{ where }  \xi :=  (E_r^2 u_0, E_r^1v_0) \}. $$
		Then, for the initial data belonging to $S_j$, the Gronwall Lemma  on \eqref{nonexp-t3a} yields
		\begin{equation}\label{gw1}
			1+ l_j(t\land\tau_k) \leq    K_{r,j}, \qquad t\le T,\quad 	j\in\mathbb N,
		\end{equation}
		where the constant $K_{r,j}$ also depends on $\| \dot{h}\|_{L^2(0,T; \rkhs)}$ and $l_j$ stands to show that $\eqref{gw1}$ holds under $S_j$ only.

		Next to deal with the third integral in \eqref{sob2}, denote by $O$ its integrand, we recall the  following celebrated Gagliardo-Nirenberg inequalities, see e.g. \cite{Friedman_1969},
		\begin{equation}\label{GNineq}
			\vert \psi\vert_{L^\infty(r-s)}^2\le\vert \psi\vert_{L^2(B_{r-s})}^2+2\vert \psi\vert_{L^2(B_{r-s})}\vert \dot\psi\vert_{L^2(B_{r-s})},\qquad\psi\in H^1(B_{r-s}).
		\end{equation}
		Then by applying \cite[Lemma 10.1]{Brz+Ondr_2007} followed by the generalized H\"older inequality and \eqref{GNineq} we infer
		\begin{align}\label{nonexp-t4}
		|O(s)| & \lesssim \mathbbm{1}_{[0,\tau_k)}(s)\frac{\int_{B_{r-s}}\{|\partial_xv_{r,k}||\partial_xu_{r,k}||v_{r,k}|^2+|\partial_{xx}u_{r,k}||\partial_xu_{r,k}|^2|v_{r,k}|+|\partial_xv_{r,k}||\partial_xu_{r,k}|^3\}\,dx}{1+\Vert a_k(s)\Vert^2_{\mathcal H_{r-s}}} \nonumber\\
		& \lesssim \mathbbm{1}_{[0,\tau_k)}(s) \frac{l(s)\Vert a_k(s)\Vert^2_{\mathcal H_{r-s}}}{1+\Vert a_k(s)\Vert^2_{\mathcal H_{r-s}}} \le \mathbbm{1}_{[0,\tau_k)}(s) (1+l(s)).
		\end{align}
		So, by substituting \eqref{nonexp-t1},  \eqref{nonexp-t3} and \eqref{nonexp-t4} in \eqref{sob2} we have
		\begin{equation}\nonumber
			q(t) \lesssim 1+ q(0) + \int_0^t \mathbbm{1}_{[0,\tau_k)}(s) (1+l(s)) ~(1+\| \dot{h}(s) \|_{\rkhs}^2)  \,ds.
		\end{equation}
		Consequently, by applying \eqref{gw1},   we obtain  on $S_j$,
		\begin{align} \label{gw2}
			q_j(t \land \tau_k) & \lesssim 1 + q_j(0) + \int_0^t  [1+l_j(s \land \tau_k )] ~(1+\| \dot{h}(s)\|_{\rkhs}^2)  \,ds \nonumber\\
			& \leq  C_{r,j} ~ \| \dot{h}\|_{L^2(0,T;\rkhs)}, \qquad j \in \mathbb{N}, t \in [0,T],
		\end{align}
		for some $C_{r,j} >0$, where in the last step we have used that $r> T$ and on set $S_j$ the quantity $q_j(0)$ is bounded by $\log(1+j)$.

		To complete the proof let us fix $t < T$. Then, by Proposition \ref{prop-coincidence},
		$$ \vert a_k(\tau_k)\vert_{\mathcal H_{r-\tau_k}}=\vert z_{r,k}(\tau_k)\vert_{\mathcal H_{r-\tau_k}}\ge k \quad \textrm{whenever} \quad \tau_k\le t. $$
		So for every $k$ such that $\tau_k \leq t$ we have
		\begin{equation}\nonumber
			 \log(1+k^2) \le  q(\tau_k) = q(t\land\tau_k).
		\end{equation}
		Thus by restricting us to $S_j$ and using inequality  \eqref{gw2}, we obtain
		\begin{equation}\label{prol}
			\log(1+k^2) \le q_j(t\land\tau_k) \lesssim C_{r,j} \| \dot{h}\|_{L^2(0,T;\rkhs)}.
		\end{equation}
		In this way, if  $\lim\limits_{k \to \infty} \tau_k = t_0$ for any $t_0 < T$, then by taking $k \to \infty$ in \eqref{prol} we get $C_{r,j} \|\dot{h} \|_{L^2(0,T;\rkhs)} \geq \infty$ which is absurd. Since this holds for every $j \in \mathbb{N}$ and $t_0 <T$, we infer that  $\tau=T$. Hence, the proof of Proposition \ref{nonexp} is complete.
	\end{proof}
	Now we have all the machinery required to finish the proof of Theorem \ref{thm-skeleton}.   Define
	$$
	w_{r,k}(t) : =\left(\begin{array}{c}E^2_{r-t}u_{r,k}(t)\\E^1_{r-t}v_{r,k}(t)\end{array}\right),
	$$
	and observe that $w_{r,k} : [0,T) \to \mathcal{H}$ is continuous.  If we set
	\begin{equation}\label{sol}
	z_r(t) : =\lim_{k\to\infty}w_{r,k}(t), \qquad t<T,
	\end{equation}
	then by Lemma \ref{ndcr} and  Proposition \ref{nonexp} it is straightforward to verify that, for every $t < T$, the sequence  $\{ w_{r,k}(t) \}_{k \in \mathbb{N}}$ is Cauchy in $\mathcal{H}$. But since $\mathcal{H}$ is complete,  the limit in \eqref{sol} converges in $\mathcal H$.
	Moreover, since  by  Proposition \ref{nonexp} $z_{r,k}(t) = z_{r, k_1}(t)$ for every $k_1 \geq k$ and $t \leq \tau_k$,  we have that $z_r(t)=w_{r,k}(t)$ for $t\le\tau_k$. In particular, $ [0,T) \ni t \mapsto z_r(t) \in\mathcal{H}$  is continuous and $z_r(t,x)=z_{r,k}(t,x)$ for $|x|\le r-t$ if $t\le\tau_k$.

	Hence, if we write $z_r(t) = (u_r(t), v_r(t))$, then we have shown that $u_r$ satisfy the first conclusion of the Theorem \ref{thm-exist}. In the remaining proof of the existence part we will show that the $z_r$, defined in \eqref{sol}, will satisfy all the remaining conclusions.
	Evaluation of (\ref{vcfa}) at $t \land \tau_k$ together applying the result from previous paragraph gives
	\begin{equation}\label{vcfa2}
	z_{r,k}(t\land\tau_k)=\xi+\int_0^{t\land\tau_k}\mathcal Gz_{r,k}(s)\,ds+\int_0^{t\land\tau_k}\mathbf F_r(s,z_{r,k}(s))\,ds+\int_0^{t\land\tau_k}\mathbf G_r(s,z_{r,k}(s)) \dot{h}(s)\,ds,
	\end{equation}
	and this equality holds in $H^1(\mathbb R; \mathbb{R}^n)\times L^2(\mathbb R; \mathbb{R}^n)$. Restricting to the interval $(-R,R)$, (\ref{vcfa2}) becomes
	$$
	z_r(t\land\tau_k)=\xi+\int_0^{t\land\tau_k}\mathcal Gz_r(s)\,ds+\int_0^{t\land\tau_k}\mathbf F_r(s,z_r(s))\,ds+\int_0^{t\land\tau_k}\mathbf G_r(s,z_r(s))  \dot{h}(s)\,ds,
	$$
	under the action of natural projection from $H^1(\mathbb R; \mathbb{R}^n)\times L^2(\mathbb R; \mathbb{R}^n)$ to $H^1((-R,R); \mathbb{R}^n)\times L^2((-R,R); \mathbb{R}^n)$. Here the integrals converge in $H^1((-R,R); \mathbb{R}^n)\times L^2((-R,R); \mathbb{R}^n)$. Taking the limit  $k \to \infty $ on both the sides, the dominated convergence theorem yields
	\begin{equation}\nonumber
	z_r(t)=\xi+\int_0^t\mathcal Gz_r(s)\,ds+\int_0^t\mathbf F_r(s,z_r(s))\,ds+\int_0^t\mathbf G_r(s,z_r(s)) \dot{h}(s)\,ds,\qquad t <T,
	\end{equation}
	in $H^1((-R,R); \mathbb{R}^n) \times L^2((-R,R); \mathbb{R}^n)$. In particular, by looking to each component separately we have, for every $t <T$,
	\begin{equation}\label{uSolnIntegral}
		u_r(t)=u_0+\int_0^tv_r(s)\,ds,
	\end{equation}
	in $H^1((-R,R); \mathbb{R}^n)$, and
	\begin{align}\label{vSolnIntegral}
		v_r(t) & =v_0+\int_0^t\left[\partial_{xx}u_r(s)+A_{u_r(s)}(v_r(s),v_r(s))-A_{u_r(s)}(\partial_xu_r(s),\partial_xu_r(s))\right]\,ds \nonumber\\
		& \qquad +\int_0^tY_{u_r(s)} ( v_r(s),\partial_xu_r(s))  \dot{h}(s) \,ds,
	\end{align}
	holds in $L^2((-R,R); \mathbb{R}^n)$. It is relevant to note that in the formula above, we have replaced $\mathcal{A}$ by $A$ which make sense because due to Proposition \ref{prop-coincidence} and  Proposition \ref{nonexp}, $u_r(t,x)=u_{r,k}(t,x)\in M$ for $|x|\le r-t$ and $t < T$. Hence we are done with the proof of existence part.

	Concerning the uniqueness, define
	$$
	Z(t) := \left(\begin{array}{c}E^2_RU(t)\\E^1_R\partial_tU(t)\end{array}\right),\qquad t<T,
	$$
	and observe that it is a $\mathcal{H}$-valued continuous function of $t \in [0,T)$. Define also
	$$
		\sigma_k : =\tau_k\land\inf\,\{t<T: \Vert Z(t)\Vert_{\mathcal H_{r-t}} \ge k\},
	$$
	and the $\mathcal H$-valued function, for $t < T$,
	$$
	\beta(t) : =S_t\xi+\int_0^tS_{t-s}\mathbbm{1}_{[0,\sigma_k)}(s)\mathbf F_{r,k}(s,Z(s))\,ds+\int_0^tS_{t-s}\mathbbm{1}_{[0, \sigma_k)}(s)\mathbf G_{r,k}(s,Z(s)) \dot{h}(s) \,ds.
	$$
	In the same vein as in the existence part of the proof,  as an application of the Chojnowska-Michalik Theorem and projection operator,  the restriction of $\beta$ on $\mathcal{H}_R$, which we denote by $b$, satisfies
	$$
	b(t)=\xi+\int_0^t\mathcal Gb(s)\,ds+\int_0^t\left(\begin{array}{c}0\\\mathcal A_{U(s)}(\partial_tU(s),\partial_tU(s))-\mathcal A_{U(s)}(\partial_xU(s),\partial_xU(s))\end{array}\right)\,ds
	$$
	$$
	+\int_0^t\left(\begin{array}{c}0\\Y_{U(s)}(\partial_tU(s),\partial_xU(s))  \dot{h}(s) \end{array}\right)\,ds,\qquad t\le \sigma_k,
	$$
	where the integrals converge in $H^1((-R,R); \mathbb{R}^n) \times L^2((-R,R); \mathbb{R}^n)$. Then since $U(t)$ and $\partial_t U(t)$ have similar form, respectively to \eqref{uSolnIntegral} and \eqref{vSolnIntegral}, by direct computation we deduce that function $p$ defined as
	$$
	p(t) : = b(t) - \left(\begin{array}{c}U(t) \\ \partial_tU(t) \end{array}\right),
	$$
	satisfies
	$$
	p(t)=\int_0^t\mathcal Gp(s)\,ds,\qquad t\le \sigma_k.
	$$
	Since the above implies that $p$ satisfies the linear homogeneous wave equation  with null initial data, by \cite[Remark 6.2]{Brz+Ondr_2007},
	\begin{equation}\label{pzero}
		p(t,x)=0 \quad \textrm{ for } \quad |x|\le R-t, t\le \sigma_k.
	\end{equation}
	Next we set
	$$ q(t): =\Vert \beta(t)-a_k(t)\Vert^2_{\mathcal H_{R-t}}, $$
	and apply Proposition \ref{prop_magic}, with $k =1, T= r, L= I$, to obtain
	\begin{align}\label{nonexp-t5}
		q(t\land \sigma_k) & \le 2 \int_0^{t\land \sigma_k} q(s)\, ds + \int_0^t \Vert \mathbf F_{r,k}(s,Z(s))-\mathbf F_{r,k}(s,a_k(s))\Vert^2_{\mathcal H} \,ds \nonumber\\
		& \quad +\int_0^{t\land \sigma_k}\Vert \mathbf G_{r,k}(s,Z(s))  \dot{h}(s) -\mathbf G_{r,k}(s,a_k(s))  \dot{h}(s) \Vert ^2_{\mathcal{H}}\,ds.
	\end{align}
	But we know that $r-t > R-t$, and by definition $\sigma_k \leq \tau_k$ which implies
	$$ F_{r,k}(t,z)=F_{R,k}(t,z), \qquad G_{r,k}(t,z)=G_{R,k}(t,z)\quad  \textrm{on} \quad (t-R,R-t), $$
	whenever $\Vert z\Vert_{\mathcal H_{r-t}}\le k$. Consequently, the estimate \eqref{nonexp-t5} becomes
	\begin{align}
		q(t\land \sigma_k)& \le 2 \int_0^{t\land \sigma_k}q(s) \, ds + \int_0^{t\land \sigma_k} \Vert \mathbf F_{R,k}(s,Z(s))-\mathbf F_{R,k}(s,a_k(s))\Vert^2_{\mathcal H}]\,ds \nonumber\\
		& \quad +\int_0^{t\land \sigma_k}\Vert \mathbf G_{R,k}(s,Z(s))  \dot{h}(s) -\mathbf G_{R,k}(s,a_k(s)) \dot{h}(s) \Vert ^2_{\mathcal{H}}\,ds. \nonumber
	\end{align}
	Invoking Lemmata \ref{lem-lip} and \ref{hsop} followed by \eqref{pzero} yields
	$$
	q(t\land \sigma_k) \le C_R  \int_0^{t\land \sigma_k} q(s)  (1+\| \dot{h}(s)\|_{\rkhs}^2 )  \,ds.
	$$
	Therefore, we get $q=0$ on $[0,\sigma_k)$ by the Gronwall Lemma. Since in the limit $k \to \infty$,  $\sigma_k$ goes to $T$ as $\tau_k$, by taking  $k$ to infinity, by Proposition \ref{prop-coincidence} we obtain that $u_r(t,x)=U(t,x)$ for each $t<T$ and $|x|\le R-t$. The proof of Theorem \ref{thm-skeleton} completes here.
\end{proof}

\section{Large deviation principle}\label{sec:LDP}
In this section we establish a large deviation principle (LDP) for system \eqref{SGWE-fundamental} via a weak convergence approach developed in \cite{Budhiraja+Dupuis_2000} and \cite{Budhiraja+Dupuis+Maroulas_2008} which is based on variational representations of infinite-dimensional Wiener processes.

First, let us recall the general criteria for LDP obtained in \cite{Budhiraja+Dupuis_2000}. Let $(\Omega,\mathfrak{F},  \mathbb{P})$ be a probability space with an increasing family $\mathbb{F} := \{\mathfrak{F}_t,  t \geq 0   \}$ of the sub-$\sigma$-fields of $\mathfrak{F}$ satisfying the usual conditions.  Let $\mathscr{B}(E)$ denote the Borel $\sigma$-field of the Polish space $E$ (i.e. complete separable metric space).
Since we are interested in the large deviations of continuous stochastic processes, we follow \cite{Chueshov+Millet_2010} and consider the following definition of large deviations principle given in terms of random variables.
\begin{Definition}
	The  $(E, \mathscr{B}(E))$-valued random family $ \left\{ X^{\varepsilon} \right\}_{\varepsilon >0}$, defined on $(\Omega,\mathfrak{F},\mathbb{P})$, is said to satisfy a large deviation principle on $E$ with the good rate function $\mathcal{I}$ if the following conditions hold:
	\begin{enumerate}
		\item  \textbf{$\mathcal{I}$ is a good rate function}: The function $\mathcal{I} : E \to [0,\infty]$ is such that for each $\mathcal{M} \in [0,\infty)$ the level set $\{\phi \in E: \mathcal{I}(\phi) \leq \mathcal{M} \}$ is a compact subset of $E$.
		\item  \textbf{Large deviation upper bound}: For each closed subset $F$ of $E$
		\begin{equation}\nonumber
		\limsup_{\varepsilon \to 0} ~ \varepsilon \log \mathbb{P}\left[  X^{\varepsilon} \in F\right] \leq - \inf_{u \in F} \mathcal{I}(u).
		\end{equation}
		\item  \textbf{Large deviation lower bound}: For each open subset $G$ of $E$
		\begin{equation}\nonumber
		\liminf_{\varepsilon \to 0} ~ \varepsilon \log \mathbb{P}\left[  X^{\varepsilon}  \in G \right] \geq - \inf_{u \in G} \mathcal{I}(u),
		\end{equation}
		where by convention the infimum over an empty set is $+\infty$.
	\end{enumerate}
\end{Definition}
Assume that $K, H$ are separable Hilbert spaces such that the embedding $K \hookrightarrow H$ is Hilbert-Schmidt. Let $W:= \{ W(t), t \geq 0 \}$ be a cylindrical Wiener process on $K$ defined on  $(\Omega, \mathfrak{F}, \mathbb{F}, \mathbb{P})$. Hence the paths of $W$ take values in $\mathcal{C}([0,\infty); H)$.

Let us,  for the whole section,   fix a number  $T>0$.
Note that the RKHS linked to $W$ restricted to the time interval $[0,T]$ is equal to  $\prescript{}{0}{H}^{1,2}(0,T;K)$.
Let $\mathscr{S}$ be the class of $K$-valued $\mathbb{F}$-predictable processes $\phi$ belonging to $\prescript{}{0}{H}^{1,2}(0,T;K)$, $\mathbb{P}$-almost surely.
For $\mathcal{M} >0$, we set
\begin{equation}\label{eqn-S_M}
S_{\mathcal{M}} := \left\{ h \in \prescript{}{0}{H}^{1,2}(0,T;K): \int_{0}^{T} \| \dot{h}(s)\|_K^2 \,ds \leq \mathcal{M}  \right\}.
\end{equation}
The set $S_{\mathcal{M}}$ endowed with the weak topology from $\prescript{}{0}{H}^{1,2}(0,T;K)$, is metrizable by the following metric
\begin{equation}\nonumber
d_1(h,k) := \sum_{i=1}^{\infty} \frac{1}{2^i} \biggl\vert  \int_{0}^{T} \langle \dot{h}(s) - \dot{k}(s), e_i \rangle_K  \, ds \biggl\vert,
\end{equation}
where $\{e_i\}_{i \in\mathbb{N}}$ is a complete orthonormal basis for $L^2(0,T;K)$, is a Polish space, see \cite{Budhiraja+Dupuis+Maroulas_2008}.
Define  $\mathscr{S}_{\mathcal{M}}$ as the set of bounded stochastic controls by
\begin{equation}\nonumber
\mathscr{S}_{\mathcal{M}} := \{ \phi \in \mathscr{S} : \phi(\omega) \in S_{\mathcal{M}} , \mathbb{P}\textrm{-a.s.} \}.
\end{equation}
Note that $\cup_{M >0}\mathscr{S}_{\mathcal{M}}$ is a proper subset of $\mathscr{S}$. Next, consider a family indexed by $\varepsilon \in (0,1]$ of Borel measurable maps
\begin{equation}\nonumber
J^{\varepsilon} : \prescript{}{0}{\mathcal{C}}([0,T]; H) \to E.
\end{equation}
We denote by $\mu^{\varepsilon}$ the `` image" measure on $E$ of $\mathbb{P}$ by $J^{\varepsilon}$, that is,
\begin{equation}\nonumber
\mu^{\varepsilon} = J^{\varepsilon}(\mathbb{P}), \quad i.e.\quad  \mu^{\varepsilon}(A) = \mathbb{P} \left( (J^{\varepsilon})^{-1}(A)\right), \quad  A \in\mathscr{B}(E).
\end{equation}
We have the following result.
\begin{Theorem}\label{thm-BudDup} \cite[Theorem 4.4]{Budhiraja+Dupuis_2000} Suppose that there exists a measurable map $J^{0} : \prescript{}{0}{\mathcal{C}}([0,T];H) \to E$ such that
	
	\begin{enumerate}
		\item[\textbf{BD1} :] if $\mathcal{M}>0$ and a family $\{ h_{\varepsilon} \} \subset \mathscr{S}_{\mathcal{M}}$  converges in law as $S_{\mathcal{M}}$-valued  random elements to $h \in\mathscr{S}_{\mathcal{M}}$ as $\varepsilon \to 0$, then the random variables
		\begin{equation}\nonumber
		\prescript{}{0}{\mathcal{C}}([0,T]; H) \ni \omega  \mapsto J^{\varepsilon} \left( \omega + \frac{1}{\sqrt{\varepsilon}}  \int_{0}^{\cdot} \dot{h}_{\varepsilon}(s) \,ds \right) \in E,
		\end{equation}
		converges in law, as $\varepsilon \searrow 0$, to the random variable $J^0 \left(  \int_{0}^{\cdot} \dot{h}(s) \,ds\right)$,
		\item[\textbf{BD2} :] for every $\mathcal{M} > 0$, the set
		\begin{equation}\nonumber
		\left\{  J^0 \left(  \int_{0}^{\cdot} \dot{h}(s) \,ds\right): h \in S_{\mathcal{M}}  \right\},
		\end{equation}
		is a compact subset of $E$.
	\end{enumerate}
	Then the family of measures $\mu^{\varepsilon}$ satisfies the large deviation principle (LDP) with the rate function defined by
	\begin{equation}\label{rateFn}
	\mathcal{I}(u) := \inf \left\{ \frac{1}{2} \int_{0}^{T} \|\dot{h}(s)\|_K^2\, ds : \prescript{}{0}{H}^{1,2}(0,T;K)  \textrm{ and } u =  J^0 \left(  \int_{0}^{\cdot} \dot{h}(s) \,ds\right) \right\},
	\end{equation}
	with the convention $\inf \{ \emptyset\} = +\infty$.
\end{Theorem}

\subsection{Main result}\label{sec:mainResult}
In is important to note that in transferring the general theory argument from Theorem \ref{thm-BudDup} in our setting we require some information about the difference of solutions at two different times, hence we need to strengthen the assumptions on diffusion coefficient. In the remaining part of this paper,  we assume that
$$Y: M \ni p \mapsto  Y(p) \in T_pM,$$
is a smooth vector field on compact Riemannian manifold $M$, which can be considered as a submanifold of $\mathbb{R}^n$, such that its extension, denote again by $Y$, on the ambient space $\mathbb{R}^n$ is smooth and satisfies

\begin{trivlist}
	\item[\textbf{Y.4}] there exists a compact set $K_Y \subset \mathbb{R}^n$ such that $Y(p) = 0$ if $p \notin K_Y$,
	\item[\textbf{Y.5}] for $q\in O$, $Y(\h(q))  = \h^\prime (q) Y(q)$,
	\item[\textbf{Y.6}] for some $C_Y>0$
	\begin{align}\nonumber
		|Y(p)| & \leq C_Y (1+|p|),  \quad  \bigg\vert  \frac{\partial Y}{\partial p_i} (p) \bigg\vert \leq C_Y, \textrm{ and } \bigg\vert  \frac{\partial^2 Y}{\partial p_i \partial p_j} (p) \bigg\vert \leq C_Y, \nonumber
	\end{align}
	for $p \in K_Y, i,j =1,\ldots,n$.
\end{trivlist}

\begin{Remark}\label{rem-Y-LDP}
	\begin{enumerate}
		\item    Since $K_Y$ is compact, there exists a  $C_K$ such that $|Y(p)| \leq C_K$ for  $p \in \mathbb{R}^n$.
		\item  For $M = \mathbb{S}^2 $ case,  $Y(p) = p \times e, p \in M$, for some fixed vector $e\in \mathbb{R}^3$ satisfies above assumptions.
	\end{enumerate}
\end{Remark}
Since, due to the above assumptions, $Y$ and its first order partial derivatives are Lipschitz, by 1-D Sobolev embedding we easily get the next result.
\begin{Lemma}\label{lem-Y-LDP}
	For any $R>0$, there exists a constant $C_{Y,R} >0$ such that the extension $Y$ defined above satisfy
	\begin{align}
		& (1)\quad \| Y(u) \|_{H^j(B_R)} \leq C_{Y,R} ( 1 + \| u\|_{H^j(B_R)} ), \quad j=0,1,2,\nonumber\\
		& (2)\quad \| Y(u) - Y(v) \|_{L^2(B_R)} \leq C_{Y,R} \| u-v \|_{L^2(B_R)}, \nonumber\\
		& (3)\quad \| Y(u) - Y(v) \|_{H^1(B_R)} \leq C_{Y,R} \| u-v \|_{H^1(B_R)}  \left( 1 + \|u\|_{H^1(B_R)}  + \|v\|_{H^1(B_R)} \right). \nonumber
	\end{align}
\end{Lemma}
\noindent Let $(\mathfrak F^{W,0}_t)$ be the $\Bbb P$-augmented filtration generated by the Wiener process $W$. Now we state the main result of this section for the following small noise Cauchy problem
\begin{equation}
\label{LDP-CP}
\left\{
\begin{aligned}
&  \partial_{tt}u^{\varepsilon} =  \partial_{xx}u^{\varepsilon} +  A_{u^{\varepsilon}}(\partial_t u^{\varepsilon}, \partial_t u^{\varepsilon}) - A_{u^{\varepsilon}}(\partial_x u^{\varepsilon}, \partial_x u^{\varepsilon})   + \sqrt{\varepsilon} Y(u^{\varepsilon}) \dot{W},   \\
& \left( u^{\varepsilon}(0), \partial_t u^{\varepsilon}(0) \right) = \left( u_0 , v_0 \right),
\end{aligned}\right.
\end{equation}
with the hypothesis that $(u_0,v_0)$ is $\mathfrak F_0$-measurable $H_{\trm{loc}}^2 \times H_{\trm{loc}}^1(\mathbb R,TM)$-valued random variable, such that $u_0(x,\omega)\in M$ and $v_0(x,\omega)\in T_{u_0(x,\omega)}M$ hold for every $\omega\in\Omega$ and $x\in\mathbb R$.
Since the small noise problem \eqref{LDP-CP}, with initial data $(u_0,v_0) \in \mathscr{H}_{\textrm{loc}}(\mathbb{R}; M)$,  is a particular case of Theorem \ref{thm-exist}, for given $\varepsilon >0$ and $T>0$, there exists a unique global strong $(\mathfrak F^{W,0}_t)$-adapted solution to \eqref{LDP-CP}, which we denote by $z^{\varepsilon} := (u^{\varepsilon}, \partial_t u^{\varepsilon})$, with values in the Polish space
\begin{equation}\nonumber
\mathcal{X}_T := \mathcal{C}\left([0,T]; H_{\textrm{loc}}^2 (\mathbb{R}; \mathbb{R}^n) \right) \times \mathcal{C}\left([0,T]; H_{\textrm{loc}}^1(\mathbb{R}; \mathbb{R}^n) \right),
\end{equation}
and satisfy the properties mentioned in Appendix \ref{sec:existUniqResult}.

	Below, let $H_\mu$ be embedded in a separable Hilbert space $E$ via a Hilbert-Schmidt inclusion $\mathbf i:H_\mu\hookrightarrow E$ as in Example \ref{example-W}, define a filtration
	$$
	\mathcal G_t=\sigma(\pi_s:s\le t),\qquad t\in[0,T]
	$$
	on $\prescript{}{0}{\mathcal{C}}([0,T]; \rkhsemb)$ where $\pi_s(f)=f(s)$, denote by $\mathfrak w$ the Wiener measure with the covariance operator $\mathbf i{\mathbf i}^*$ on $\prescript{}{0}{\mathcal{C}}([0,T]; \rkhsemb)$ and denote by $\mathbf B$ the identity mapping on $\prescript{}{0}{\mathcal{C}}([0,T]; \rkhsemb)$.
	
	\begin{Lemma}\label{lem-Jmap} Let $(u_0,v_0) \in \mathscr{H}_{\textrm{loc}}(\mathbb{R}; M)$. Then there exists a Borel measurable mapping $J^{\varepsilon}=(U^{\varepsilon},V^{\varepsilon})$
		\begin{equation}\label{Jepsilon}
		J^{\varepsilon} : \prescript{}{0}{\mathcal{C}}([0,T]; \rkhsemb)  \to \mathcal{X}_T,
		\end{equation}
		such that
		\begin{itemize}
			\item[(a)] $U^\varepsilon(t,x), V^\varepsilon(t,x)$ are $\mathcal G^{\mathfrak w}_t$-adapted for every $(t,x)\in[0,T]\times\Bbb R$,
			\item[(b)] $U^\varepsilon(t,x):\prescript{}{0}{\mathcal{C}}([0,T]; \rkhsemb)\to M$ for every $(t,x)\in[0,T]\times\Bbb R$,
			\item[(c)] $t\mapsto U^\varepsilon(t)\in H^1_{\textrm{loc}}(\Bbb R;\mathbb{R}^n)$ is continuously differentiable and
			$$
			\frac{dU^\varepsilon}{dt}=V^\varepsilon,
			$$
			\item[(d)] $(U^\varepsilon(0),V^\varepsilon(0))=(u_0,v_0)$,
			\item[(e)] $(U^\varepsilon,\mathbf B)$ is a solution of \eqref{LDP-CP} in the sense of Theorem \ref{thm-exist} for the probability measure $\mathfrak w$,
			\item[(f)] if $\tilde W$ is an $E$-valued Wiener process with covariance operator $\mathbf i{\mathbf i}^*$ on some stochastic basis then $(U^\varepsilon(\tilde W),\tilde W)$ is a solution of \eqref{LDP-CP} in the sense of Theorem \ref{thm-exist}.
		\end{itemize}
	\end{Lemma}
	
	\begin{proof}[\textbf{Proof of Lemma \ref{lem-Jmap}}]
		Define a stopping operator
		$$
		L_t:\prescript{}{0}{\mathcal{C}}([0,T]; \rkhsemb)\to\prescript{}{0}{\mathcal{C}}([0,T]; \rkhsemb):f\mapsto f(\cdot\land t)
		$$
		and observe that $\mathcal G_t=\sigma(L_t)$ and $\mathfrak F^W_t=\sigma(L_t(W))$. Doob-Dynkin lemma yields existence of a Borel measurable mapping $J^{\varepsilon}$ such that $z^\varepsilon=J^{\varepsilon}(W)$ a.s., and since $z^\varepsilon$ is $(\mathfrak F^{W,0}_t)$-adapted, the same lemma yields existence of a Borel measurable mapping
		$$
		l_t:\prescript{}{0}{\mathcal{C}}([0,T]; \rkhsemb)\to H_{\textrm{loc}}^2 (\mathbb{R}; \mathbb{R}^n)\times H_{\textrm{loc}}^1 (\mathbb{R}; \mathbb{R}^n)
		$$
		such that $z^\varepsilon(t)=l_t(L_t(W))$ a.s.. Hence $\mathfrak w(J^\varepsilon_t=l_t\circ L_t)=1$ and we conclude that $J^\varepsilon_t$ is $\mathcal G^{\mathfrak w}_t$-measurable for every $t\in[0,T]$. In particular, we have proved (a). Since $U^\varepsilon(t,x)(W)=u^\varepsilon(t,x)\in M$ a.s. for every $(t,x)\in[0,T]\times\Bbb R$ by definition, we get that, $\mathfrak w$-a.s.,  $U^\varepsilon(t,x)\in M$ for every $(t,x)\in[0,T]\times\Bbb R$ since paths of $U^\varepsilon$ are jointly continuous. Thus (b) holds $\mathfrak w$-a.s. Next,
		$$
		u^\varepsilon(t,x)=u_0(x)+\int_0^t\partial_tu^\varepsilon(s,x)\,ds
		$$
		holds a.s. for every $(t,x)\in[0,T]\times\Bbb R$ so, as in the previous step, $\mathfrak w$-a.s.,
		$$
		U^\varepsilon(t,x)=u_0(x)+\int_0^tV^\varepsilon(s,x)\,ds
		$$
		holds for every $(t,x)\in[0,T]\times\Bbb R$ since paths of $U^\varepsilon$ and $V^\varepsilon$ are jointly continuous. In particular, (c) holds $\mathfrak w$-a.s. Moreover, it is obvious that (d) holds $\mathfrak w$-a.s. To deal with the $\mathfrak w$-exceptional set, denote by $\gamma$ the smooth geodesic flow on $\Bbb R\times TM$ and redefine, on this exceptional set,
		$$
		J^\varepsilon(t,x)=(\gamma(t,u_0(x),v_0(x)),\dot\gamma(t,u_0(x),v_0(x)))
		$$
		which satisfies (b), (c) and (d) as well. Finally, if we define $(\tilde u^\varepsilon,\tilde v^\varepsilon)=(\tilde U^\varepsilon(\tilde W),\tilde V^\varepsilon(\tilde W))$ then the finite-dimensional distributions of the processes
		$$
		(V^\varepsilon,\partial_{xx}U^\varepsilon,A_{U^\varepsilon}(\partial_xU^\varepsilon,\partial_xU^\varepsilon),A_{U^\varepsilon}(V^\varepsilon,V^\varepsilon),Y(U^\varepsilon),\mathbf B)
		$$
		$$
		(\partial_tu^\varepsilon,\partial_{xx}u^\varepsilon,A_{u^\varepsilon}(\partial_xu^\varepsilon,\partial_xu^\varepsilon),A_{u^\varepsilon}(\partial_tu^\varepsilon,\partial_tu^\varepsilon),Y(u^\varepsilon),W)
		$$
		$$
		(\tilde v^\varepsilon,\partial_{xx}\tilde u^\varepsilon,A_{\tilde u^\varepsilon}(\partial_x\tilde u^\varepsilon,\partial_x\tilde u^\varepsilon),A_{\tilde u^\varepsilon}(\tilde v^\varepsilon,\tilde v^\varepsilon),Y(\tilde u^\varepsilon),\tilde W)
		$$
		coincide in every in $L^2((-R,R;\Bbb R^n))$ hence we obtain (e) and (f) e.g. by \cite[Theorem 8.3 and Theorem 8.6]{Ondr_2004}. Let us just point out that the measurability and qualitative properties of $\tilde u^\varepsilon$ and $\tilde v^\varepsilon=\frac{d\tilde u^\varepsilon}{dt}$ are guaranteed by (a)-(d).
	\end{proof}

Recall from Section \ref{sec:prelim} that the random perturbation $W$ we consider is a cylindrical Wiener process on $\rkhs$ and there exists a separable Hilbert space $E$ such that the embedding of $\rkhs$ in  $E$ is Hilbert-Schmidt. Hence we can apply the general theory from previous section with the notations defined by taking $\rkhs$ instead of $K$.

Let us define a Borel map
\begin{equation}\label{eqn-J^0}
J^0: \prescript{}{0}{\mathcal{C}}([0,T]; \rkhsemb) \to \mathcal{X}_T.
\end{equation}
Note that it is well-defined due to Lemma \ref{lem-Jmap}.  If $h \in \prescript{}{0}{\mathcal{C}}([0,T]; \rkhsemb) \setminus \prescript{}{0}{H}^{1,2}(0,T;\rkhs)$, then we set $J^0(h)=0$. If $h \in \prescript{}{0}{H}^{1,2}(0,T;\rkhs)$ then by Theorem \ref{thm-skeleton} there exists a function in $\mathcal{X}_T$, say $z_h$, that solves
\begin{equation}\label{LDP-skeleton}
\left\{\begin{aligned}
& \partial_{tt}u =  \partial_{xx}u +  A_{u}(\partial_t u, \partial_t u) - A_{u}(\partial_x u, \partial_x u)  +Y(u) \,  \dot{h},
\\
& u(0,\cdot)  = u_0, \partial_tu (0,\cdot)=v_0,
\end{aligned}
\right.\end{equation}
uniquely and we set $J^0(h) = z_h$.
\begin{Remark}
	At some places in the paper we denote $J^0(h)$ by $J^0\left( \int_{0}^{\cdot} \dot{h}(s)\, ds \right)$ to make it clear that the considered differential equation is controlled by $\dot{h}$ not by $h$.
\end{Remark}

The main result of this section is as follows:
\begin{Theorem}\label{thm-LDP}
The family of laws $\{ \mathscr{L}(z^{\varepsilon}): \varepsilon \in (0,1] \}$  on $\mathcal{X}_T$, where $z^\varepsilon := (u^{\varepsilon}, \partial_t u^{\varepsilon})$ is the unique solution to \eqref{LDP-CP}  satisfies the large deviation principle with rate function $\mathcal{I}$ defined in \eqref{rateFn}.
\end{Theorem}

Note that, in light of Theorem \ref{thm-BudDup}, in order to prove the Theorem \ref{thm-LDP} it is sufficient to show the following two statements:
\begin{enumerate}
	\item[]\textbf{Statement 1}: For each $\mathcal{M}>0$, the set $K_{\mathcal{M}} := \{ J^0(h): h \in S_{\mathcal{M}}\}$  is a compact subset of $\mathcal{X}_T$, where $S_{\mathcal{M}} \subset \prescript{}{0}{H}^{1,2}(0,T;\rkhs)$ is the centred closed ball of radius $\mathcal{M}$ endowed with the weak topology.    \\
	\item[]\textbf{Statement 2}: Assume that $\mathcal{M}>0$, that $\{\varepsilon_n \}_{n \in \mathbb{N}}$ is an $(0,1]$-valued sequence convergent to $0$, that $\{h_n \}_{n \in \mathbb{N}} \subset \mathscr{S}_{\mathcal{M}}$ converges in law to $h \in\mathscr{S}_{\mathcal{M}}$ as $\varepsilon \to 0$. Then,  the processes
	\begin{equation}\label{Noise+SkeletonPrc}
	\prescript{}{0}{\mathcal{C}}([0,T]; E) \ni W(\cdot)  \mapsto J^{\varepsilon_n} \left( W(\cdot) + \frac{1}{\sqrt{\varepsilon_n}}  \int_{0}^{\cdot} \dot{h}_n(s) \,ds \right) \in \mathcal{X}_T,
	\end{equation}
	converges in law on $\mathcal{X}_T$ to $J^0\left( \int_{0}^{\cdot} \dot{h} (s) \, ds \right)$.
\end{enumerate}

\begin{Remark}
	By combining the proofs of Theorem \ref{thm-exist} and Theorem \ref{thm-skeleton} we infer that the map \eqref{Noise+SkeletonPrc} is well-defined and  $J^{\varepsilon_n} \left( W(\cdot) + \frac{1}{\sqrt{\varepsilon_n}}  \int_{0}^{\cdot} \dot{h}_n(s) \,ds \right) $ solves the following stochastic control Cauchy problem
	\begin{equation}
	\label{CP_noise+skeleton}
	\left\{
	\begin{aligned}
	&  \partial_{tt}u^{\varepsilon_n} =  \partial_{xx}u^{\varepsilon_n} +  A_{u^{\varepsilon_n}}(\partial_t u^{\varepsilon_n}, \partial_t u^{\varepsilon_n}) - A_{u^{\varepsilon_n}}(\partial_x u^{\varepsilon_n}, \partial_x u^{\varepsilon_n})  + Y(u^{\varepsilon_n})\dot{h}_n    \\
	& \qquad \qquad  + \sqrt{\varepsilon_n} Y(u^{\varepsilon_n}) \dot{W}, \\
	& \left( u^{\varepsilon_n}(0), \partial_t u^{\varepsilon_n}(0) \right) = \left( u_0 , v_0 \right),
	\end{aligned}\right.
	\end{equation}
	for the initial data $\left( u_0 , v_0 \right)\in H_{\trm{loc}}^2 \times H_{\trm{loc}}^1(\mathbb{R};TM)$.
\end{Remark}

\begin{Remark}
	It is clear by now that verification of an LDP comes down to proving two convergence results, see \cite{Brz+Manna+Panda_2019,Brz+Manna+Zhai_2018,Brz+Peng+Zhai_sub,Chueshov+Millet_2010,Sri+Sunder_2006}. As it was shown first in \cite{Brz+Gold+Jeg_2017}, the second convergence result follows  from  the first one via the Jakubowski version of the Skorokhod representation theorem. Therefore, establishing LDP, de facto, reduces to proving one convergence result for deterministic controlled problem called also the skeleton equation.
	This convergence result is specific to the stochastic PDE in question and require techniques related to the considered equation. Thus, for instance, the proof in \cite[Lemma 6.3]{Brz+Gold+Jeg_2017}  for the stochastic Landau-Lifshitz-Gilbert equation, is different from the proof, for stochastic Navier-Stokes equation,  of \cite[Proposition 3.5]{Chueshov+Millet_2010}.  On technical level, the proof of corresponding result, i.e. \textbf{Statement 1}, is the main contribution of our work.
\end{Remark}

\subsection{Proof of Statement 1} \label{subsec: Statement 1}

Let us fix $\mathcal{M}>0$ and consider a sequence of controls $\{ h_n\}_{n \in \mathbb{N}} \subset S_{\mathcal{M}}$.
Let $z_n = (u_n, v_n) := J^0(h_n)$, for $n \in \mathbb{N}$, be a solution to problem \eqref{LDP-skeleton}, corresponding to  control $h_n$.
  Since $S_{\mathcal{M}}$ is the closed unit ball in the  Hilbert space $\prescript{}{0}{H}^{1,2}(0,T;\rkhs) $, by the Banach-Alaoglu Theorem \cite[Theorem 3.15]{Rudin_FA_1991} or \cite[Theorem 3.16]{Brezis_2011},  $S_{\mathcal{M}}$ is weakly compact. Consequently there exists a subsequence of $\{ h_n\}_{n \in \mathbb{N}}$,  we still denote this by $\{ h_n\}_{n \in \mathbb{N}}$,  which converges weakly to a limit $h \in  S_{\mathcal{M}}$. Hence in order to complete the proof  of \textbf{Statement 1} we only need to show that the subsequence  $\{ z_n\}_{n \in \mathbb{N}}$   converges to $z_h = (u_h,v_h)$ which, by definition, is the unique solution to  the  Cauchy  problem of the skeleton equation \eqref{LDP-skeleton} with  the control $h$.

  Before delving into the proof of this claim we establish the following a priori estimate which is a preliminary step required to prove, Proposition \ref{prop-weakConv}, the main result of this section. Let us recall that $T>0$ is fixed for the whole section and $\mathcal{M}>0$ is chosen and fixed in this subsection.
\begin{Lemma}\label{lem-aprioriEnergyEst}
	If  $x \in \mathbb{R}$, then  there exists a constant $\mathcal{B}>0$, which depends on $\|(u_0,v_0)\|_{\mathcal{H}(B(x,T))}, \mathcal{M}$ and $T$, such that
	\begin{equation}\label{lem-aprioriEnergyEst-res1}
		\sup_{h \in S_{\mathcal{M}}} \sup_{t \in [0,\frac{T}{2}]} \textbf{e}(t,T;x,z_h(t)) \leq \mathcal{B},
	\end{equation}
	where $z_h$ is the unique global strong solution to problem \eqref{LDP-skeleton} and
	\begin{align}
	\textbf{e}(t,T;x,z): & = \frac{1}{2}\|z\|_{\mathcal{H}_{B(x,T-t)}}^2 = \frac{1}{2}\left\{ \Vert u\Vert^2_{L^2(B(x,T-t))} + \Vert \partial_x u\Vert^2_{L^2(B(x,T-t))} + \Vert v \Vert^2_{L^2(B(x,T-t))} \right. \nonumber\\
	& \quad \left. + \Vert \partial_{xx} u\Vert^2_{L^2(B(x,T-t))} +  \Vert \partial_x v \Vert^2_{L^2(B(x,T-t))} \right\}, \quad z= (u,v) \in \mathcal{H}_{\textrm{loc}}. \nonumber
	\end{align}
	Moreover, if we restrict $x$ on an interval $[-a,a]\subset \mathbb{R}$,  then the constant $\mathcal{B}:= \mathcal{B}(\mathcal{M},T,a)$, which also depends on `$a$' now,  can be chosen such that
	\begin{equation}\nonumber
	\sup_{x \in [-a,a]}	\sup_{h \in S_{\mathcal{M}}} \sup_{t \in [0,\frac{T}{2}]} \textbf{e}(t,T;x,z_h(t)) \leq \mathcal{B}.
	\end{equation}
\end{Lemma}
\begin{proof}[\textbf{Proof of Lemma \ref{lem-aprioriEnergyEst}}] Let us choose and fix    $x \in \mathbb{R}$.
	First note that the last part follows from the first one because by assumptions, $(u_0,v_0) \in \mathcal{H}_{\trm{loc}}$, in particular, $\|(u_0,v_0)\|_{\mathcal{H}(-a-T,a+T)} < \infty$ and therefore,
	\[\sup_{x \in [-a,a]} \|(u_0,v_0)\|_{\mathcal{H}(B(x,T))} \leq \|(u_0,v_0)\|_{\mathcal{H}(-a-T,a+T)}<\infty.\]
	The procedure to prove \eqref{lem-aprioriEnergyEst-res1} is based on the proof of Proposition \ref{nonexp}. Let us fix $h$ in  $S_{\mathcal{M}}$ and denote the corresponding solution $z_h := (u_h,v_h)$ which exists due to Theorem \ref{thm-skeleton}.
	
	Since $x$ is fixed, we will avoid writing it explicitly in the norm. Define
	$$l(t,T;x) := \frac{1}{2} \| (u_h(t), v_h(t)\|_{H^1(B_{T-t})\times L^2(B_{T-t})}^2, \quad t \in [0,T]. $$
	To shorten the notation we will write $l(t)$ in place of $l(t,T;x)$. Thus, invoking Proposition \ref{prop_magic}, with $ k =0$ and $L = I$,  implies, for $t \in [0,T]$,
	\begin{align}\label{lem-aprioriEnergyEst-t0}
	l(t) & \leq  l(0)+ \int_0^t \langle u_h(r) , v_h(s) \rangle_{L^2(B_{T-s})} \, ds  + \int_0^t \langle v_h(s) , f_h(s) \rangle_{L^2(B_{T-s})} \, ds 	\nonumber \\
	& \quad + \int_0^t \langle v_h(s) , Y(u_h(s)) \dot{h}(s) \rangle_{L^2(B_{T-s})} \, ds,
	\end{align}
	where
	\begin{equation}\nonumber
		f_h(r) := A_{u_h(r)}(v_h(r), v_h(r) - A_{u_h(r)}(\partial_x u_h(r), \partial_x u_h(r).
	\end{equation}
	Since $v_h(r) \in T_{u_h(r)}M$ and by definition $A_{u_h(r)}(\cdot, \cdot) \in N_{u_h(r)}M $, the second integral in \eqref{lem-aprioriEnergyEst-t0} vanishes. Because $u_h(r) \in M$, invoking the Cauchy-Schwartz inequality,  Lemmata \ref{hsop}  and \ref{lem-Y-LDP} implies
	\begin{equation}\nonumber
		l(t) \leq l(0) + \left( \frac{C_{Y}^2 C_T^2}{2} + 2 \right) \int_{0}^{t}  (1+l(s)) (1+\| \dot{h}(s)\|_{\rkhs}^2) \, ds .
	\end{equation}
	Consequently, by appying the Gronwall Lemma and using $h \in S_{\mathcal{M}}$ we get
	\begin{align}\label{lem-aprioriEnergyEst-t2}
		l(t) \lesssim_{C_Y,C_T} (1+ l(0) )  \left[ T + \|\dot{h}\|_{L^2(0,T;\rkhs)}^2 \right] \leq (T+\mathcal{M}) (1+l(0)).
	\end{align}
	Next we define
	$$ q(t) : = \log(1+\Vert z_h(t)\Vert_{\mathcal H_{T-t}}^2). $$
	Then  Proposition \ref{prop_magic}, with $ k =1$ and $L(x) = \log(1+x)$,  gives, for $t \in [0,\frac{T}{2}]$,
	\begin{align}
	q(t) &\leq  q(0)+\int_0^t\frac{\Vert z_h(s)\Vert_{\mathcal H_{T-s}}^2}{1+\Vert z_h(s)\Vert_{\mathcal H_{T-s}}^2}\,ds \nonumber \\
	&+ \int_0^t \frac{\langle v_h (s), f_h(s) \rangle_{L^2(B_{T-s})}}{1+\Vert z_h(s)\Vert_{\mathcal H_{T-s}}^2}\,ds
	+ \int_0^t  \frac{\langle\partial_xv_h(s),\partial_x[f_h(s) ]\rangle_{L^2(B_{T-s})}}{1+\Vert z_h(s)\Vert_{\mathcal H_{T-s}}^2}\,ds
	\nonumber \\
	& + \int_0^t \frac{\langle v_h(s), Y(u_h(s)) \dot{h}(s) \rangle_{L^2(B_{T-s})}}{1+\Vert z_k(s)\Vert_{\mathcal H_{T-s}}^2}\,ds+  \int_0^t \frac{\langle\partial_x v_h(s),\partial_x[Y(u_h(s)) \dot{h}(s) ]\rangle_{L^2(B_{T-s})}}{1+\Vert z_h(s)\Vert_{\mathcal H_{T-s}}^2} \,ds. \nonumber
	\end{align}
	Since by perpendicularity the second integral in above vanishes, by doing the calculation based on \eqref{nonexp-t3i} and  \eqref{nonexp-t4}  we deduce
	\begin{align}
		q(t) & \lesssim_{T}  1 +  q(0)  + \int_0^t  \frac{l(s)\Vert z_h(s)\Vert^2_{\mathcal H_{T-s}}}{1+\Vert z_h(s)\Vert^2_{\mathcal H_{T-s}}}  \, ds \nonumber\\
		& \quad +  \int_0^t \frac{(1+l(s))~ (1+ \|z_h(s)\|_{\mathcal{H}_{T-s}}^2)  (1+ \| \dot{h}(s)\|_{\rkhs}^2 )}{1+\Vert z_k(s)\Vert_{\mathcal H_{T-s}}^2} \, ds \nonumber\\
		& \leq 1 + q(0) + \int_{0}^{t} (1+l(s)) (1+ \| \dot{h}(s)\|_{\rkhs}^2)  \, ds, \nonumber
	\end{align}
	which further implies, due to  \eqref{lem-aprioriEnergyEst-t2} and $h \in S_{\mathcal{M}}$,
	\begin{equation}\nonumber
		q(t) \lesssim 1 + q(0) +  (T+\mathcal{M})^2 (1+l(0)).
	\end{equation}
	In terms of $z_h$, that is, for each $x \in \mathbb{R}$ and $t \in [0,\frac{T}{2}]$,	
	\begin{equation}\nonumber
		\| z_h(t) \|_{\mathcal{H}_{B(x,T-t)}}^2 \lesssim \exp\left[ \| (u_0,v_0)\|_{\mathcal{H}_{B(x,T)}}^2 (T+\mathcal{M})^2  \right].
	\end{equation}
	Since above holds for every $t \in [0,\frac{T}{2}], h\in S_{\mathcal{M}}$, by taking supremum on $t$ and $h$ we get \eqref{lem-aprioriEnergyEst-res1}, and hence the proof of Lemma \ref{lem-aprioriEnergyEst}.
\end{proof}

\begin{Remark}\label{rem-aprioriEnergyEst}
	Since $B(x,\frac{T}{2}) \subseteq B(x,T-t)$ for every $t \in [0,\frac{T}{2}]$, Lemma \ref{lem-aprioriEnergyEst} also implies
	\begin{equation}\nonumber
		\sup_{x \in [-a,a]}  \sup_{h \in S_{\mathcal{M}}}\sup_{t \in [0,\frac{T}{2}] } \frac{1}{2} \left\{  \Vert u_h(t) \Vert^2_{H^2(B(x,R))}  + \Vert v_h(t) \Vert^2_{H^1(B(x,R))} \right\}  \leq \mathcal{B}(\mathcal{M},T,a),
	\end{equation}
	for $R=\frac{T}{2}$.
\end{Remark}

Recall that, in the current subsection \ref{subsec: Statement 1}, we have the sequence $\{ h_n\}_{n \in \mathbb{N}}$ which converges weakly to a limit $h \in  S_{\mathcal{M}}$. Now we prove the main result of this subsection which will allow to complete the proof of \textbf{Statement 1}.
\begin{Proposition}\label{prop-strngConv}
Let $z_n = (u_n, v_n) := J^0(h_n)$, for $n \in \mathbb{N}$, be a solution to problem \eqref{LDP-skeleton}, corresponding to  control $h_n$ and similarly let $z_h = (u_h, v_h) := J^0(h)$.
Then  	the sequence $\{ z_n \}_{n \in \mathbb{N}}$  converges to $z_h$ in the space $\mathcal{X}_{T}$.
\\
 In particular,  the map
	\begin{equation}\nonumber
		S_{\mathcal{M}} \in h \mapsto J^0(h)  \in \mathcal{X}_{T},
 	\end{equation}
 	is Borel measurable.
\end{Proposition}
\begin{proof}[\textbf{Proof of Proposition \ref{prop-strngConv}}]
	Let us first note that the second part of the Proposition  follows from first one  because continuous maps are Borel measurable.

Towards proving the first conclusion let us  consider the objects as in the assumptions of Proposition \ref{prop-strngConv}. In particular,
$z_h = (u_h,v_h)$ and $z_n= (u_n,v_n)$,  are the unique global strong solutions,   respectively, to
	\begin{equation}\label{prop-strngConv-uh}
		\left\{
		\begin{aligned}
			&  \partial_{tt}u_h =  \partial_{xx}u_h +  A_{u_h}(\partial_t u_h, \partial_t u_h) - A_{u_h}(\partial_x u_h, \partial_x u_h) + Y(u_h)\dot{h},   \\
			& \left( u_h(0), v_h(0) \right) = \left( u_0 , v_0 \right), \quad \textrm{ where } v_nh := \partial_t u_h,
		\end{aligned}\right.
	\end{equation}
	and
	\begin{equation}\label{prop-strngConv-un}
		\left\{
		\begin{aligned}
			&  \partial_{tt}u_n =  \partial_{xx}u_n +  A_{u_n}(\partial_t u_n, \partial_t u_n) - A_{u_n}(\partial_x u_n, \partial_x u_n)   +  Y(u_n) \dot{h}_n,   \\
			& \left( u_n(0), v_n(0) \right) = \left( u_0 , v_0 \right),  \quad \textrm{ where } v_n := \partial_t u_n.
		\end{aligned}\right.
	\end{equation}
	Hence $\mathfrak{z}_n:= (\mathfrak{u}_n, \mathfrak{v}_n) = z_h - z_n $ is the unique global strong solution to, with null initial data,
	\begin{align}\label{prop-strngConv-barun}
		\partial_{tt}\mathfrak{u}_n & =  \partial_{xx}\mathfrak{u}_n - A_{u_h}(\partial_x u_h, \partial_x u_h) +  A_{u_n}(\partial_x u_n, \partial_x u_n) + A_{u_h}(\partial_t u_h, \partial_t u_h)  \nonumber\\
		& \quad - A_{u_n}(\partial_t u_n, \partial_t u_n)  +  Y(u_h) \dot{h}  - Y(u_n) \dot{h}_n,
	\end{align}
	where $\mathfrak{v}_n := \partial_t \mathfrak{u}_n$.  This implies that
	\begin{align}
		\mathfrak{z}_n(t) = \int_{0}^{t} S_{t-s}\left(\begin{array}{c}0\\f_n(s)\end{array}\right)\, ds + \int_{0}^{t} S_{t-s} \left(\begin{array}{c}0\\g_n(s)\end{array}\right)\, ds, \quad t \in [0,T]. \nonumber
	\end{align}
	Here
	\begin{align}
		f_n(s) & := - A_{u_h(s)}(\partial_x u_h(s), \partial_x u_h(s)) +  A_{u_n(s)}(\partial_x u_n(s), \partial_x u_n(s)) + A_{u_h(s)}(\partial_t u_h(s), \partial_t u_h(s))  \nonumber\\
		& \qquad - A_{u_n(s)}(\partial_t u_n(s), \partial_t u_n(s)) , \nonumber
	\end{align}
	and
	\begin{equation}
		g_n(s) := Y(u_h(s))\dot{h}(s)  - Y(u_n(s))\dot{h}_n(s). \nonumber
	\end{equation}
	We aim to show that
	\begin{equation}\nonumber
		\mathfrak{z}_n  \xrightarrow[n \to \infty]{} 0  \quad \textrm{ in } \quad \mathcal{C}\left( [0,T], H_{\textrm{loc}}^2(\mathbb{R}; \mathbb{R}^n) \right) \times \mathcal{C}\left( [0,T], H_{\textrm{loc}}^1(\mathbb{R}; \mathbb{R}^n) \right),
	\end{equation}
	that is, for every $R >0$ and $x \in \mathbb{R}$,
	\begin{equation} \label{prop-strngConv-t1}
		\sup_{t \in [0,T]} \left[  \|\mathfrak{u}_n(t) \|_{H^2(B(x,R))}^2 + \|\mathfrak{v}_n(t) \|_{H^1 (B(x,R))}^2 \right]   \to 0 \textrm{ as } n \to \infty.
	\end{equation}

	Without loss of generality we assume $x=0$.  Since a compact set in $\mathbb{R}$ can be covered by a finite number of any given closed interval of non-zero length, it is sufficient to prove above for a fixed $R>0$ whose value we set to $T$.
	
	Let $\varphi$ be a bump function which takes value $1$ on $B_R$ and vanishes outside $\overline{B_{2R}}$.    Define
	$$\bar{u}_n(t,x) : = u_n(t,x) \varphi(x) \quad \textrm{ and } \quad \bar{u}_h(t,x) : = u_h(t,x) \varphi(x),$$
	so
	$$ \bar{v}_n(t,x) =  \varphi(x) v_n(t,x), \qquad  \quad  \bar{v}_h(t,x) =  \varphi(x) v_h(t,x),$$
	and  with notation $\bar{\mathfrak{u}}_n : = \bar{u}_n - \bar{u} _h$,
	\begin{align}
		& \partial_{tt} \bar{\mathfrak{u}}_n - \partial_{xx} \bar{\mathfrak{u}}_n   = \left[ A_{u_n}(\partial_t u_n, \partial_t u_n) - A_{u_n}(\partial_x u_n, \partial_x u_n)  - A_{u_h}(\partial_t u_h, \partial_t u_h) \right.\nonumber\\
		& \quad \left. + A_{u_h}(\partial_x u_h, \partial_x u_h)  \right] \varphi   - (u_n - u_h)  \partial_{xx} \varphi   -   2 ( \partial_x u_n - \partial_x u_h) \partial_x \varphi  + \left[ Y(u_n)\dot{h}_n - Y(u_h)\dot{h} \right] \varphi \nonumber\\\
		& \quad =:  \bar{f}_n + \bar{g}_n. \nonumber
	\end{align}
	Here
	\begin{align}
		\bar{f}_n(s) & : = \left[ A_{u_n(s)}(\partial_t u_n(s), \partial_t u_n(s)) - A_{u_n(s)}(\partial_x u_n(s), \partial_x u_n(s))  - A_{u_h(s)}(\partial_t u_h(s), \partial_t u_h(s)) \right. \nonumber\\
		& \quad  \left. + A_{u_h(s)}(\partial_x u_h(s), \partial_x u_h(s))  \right] \varphi  - (u_n(s) - u_h(s))  \partial_{xx} \varphi   -   2 ( \partial_x u_n(s) - \partial_x u_h(s)) \partial_x \varphi , \nonumber
	\end{align}
	and
	\begin{equation}
		\bar{g}_n(s) := \left[ Y(u_n(s))\dot{h}_n(s) - Y(u_h(s))\dot{h}(s) \right] \varphi, \quad s \in [0,T]. \nonumber
	\end{equation}
	Next, by direct computation we can find constants $C_{\varphi}, \bar{C}_{\varphi} >0 $, depend on $\varphi, \varphi^\prime, \varphi^{\prime\prime},$ such that, for all $t \in [0,T]$ and $n \in \mathbb{N}$,
	\begin{align}\label{prop-strngConv-normChngVarphi}
		\|\bar{\mathfrak{u}}_n(t) \|_{H^2(-R,R)}^2  + \| \bar{\mathfrak{v}}_n(t) \|_{H^1(-R,R)}^2  & \leq C_{\varphi} \bigl[  \| \mathfrak{u}_n(t) \|_{H^2(-R,R)}^2  + \| \mathfrak{v}_n(t) \|_{H^1(-R,R)}^2 \bigr] \nonumber\\
		& \leq \bar{C}_{\varphi} \bigl[ \|\bar{\mathfrak{u}}_n(t) \|_{H^2(-R,R)}^2 + \| \bar{\mathfrak{v}}_n(t) \|_{H^1(-R,R)}^2 \bigr] .
	\end{align}
	Hence, in order to prove assertion   \eqref{prop-strngConv-t1} it is enough to prove the  following
	\begin{equation}\label{prop-strngConv-LimitBump}
		\sup_{t \in [0,T]} \left[  \|\bar{\mathfrak{u}}_n(t) \|_{H^2(-R,R)}^2 + \| \bar{\mathfrak{v}}_n(t) \|_{H^1 (-R,R)}^2 \right]   \to 0 \textrm{ as } n \to \infty.
	\end{equation}
Using  the time dependent balls in the space $\mathbb{R}$,  what  is more natural in the context of  the wave equations, we observe that  claim \eqref{prop-strngConv-LimitBump} is a consequence of the
following one.
	\begin{equation}\label{prop-strngConv-LimitBumpi}
		\sup_{t \in [0,R]} \left[  \|\bar{\mathfrak{u}}_n(t) \|_{H^2(B_{\mathcal{T}-t})}^2 + \| \bar{\mathfrak{v}}_n(t) \|_{H^1 (B_{\mathcal{T}-t})}^2 \right]   \to 0 \textrm{ as } n \to \infty,
	\end{equation}
	where $\mathcal{T} := 4T $.
Indeed,  because for every $t \in [0,R]$, $T-t > 2R$ and consequently, we have
	\begin{align}
		  \|\bar{\mathfrak{u}}_n(t) \|_{H^2(B_R)}^2 + \| \bar{\mathfrak{v}}_n(t) \|_{H^1 (B_R)}^2 & \leq   \|\bar{\mathfrak{u}}_n(t) \|_{H^2(B_{2R})}^2 + \| \bar{\mathfrak{v}}_n(t) \|_{H^1 (B_{2R})}^2 \nonumber\\
		&  \leq \sup_{t \in [0,R]} \left[  \|\bar{\mathfrak{u}}_n(t) \|_{H^2(B_{\mathcal{T}-t})}^2 + \| \bar{\mathfrak{v}}_n(t) \|_{H^1 (B_{\mathcal{T}-t})}^2 \right] . \nonumber
	\end{align}
So we conclude  that  in order to  prove  Proposition \ref{prop-strngConv} it is enough to show \eqref{prop-strngConv-LimitBumpi}.

\begin{proof}[Proof of claim \eqref{prop-strngConv-LimitBumpi}]
	Let us set $ l(t,z) := \frac{1}{2} \|z\|_{\mathcal{H}_{\mathcal{T}-t}}^2,$ for $z= (u,v) \in \mathcal{H}_{\textrm{loc}}$ and $t \in [0,R]$. Invoking Proposition \ref{prop_magic}, with null diffusion part and $k=1,L=I, x=0$, gives, for every $t \in [0,R]$,
	\begin{align}\label{prop-strngConv-t2i}
		l(t,\bar{\mathfrak{z}}_n(t)) & \leq  \int_{0}^{t} \mathbb{V}(r,\bar{\mathfrak{z}}_n(r)) \, dr ,
	\end{align}
	where $\bar{\mathfrak{z}}_n(t) = (\bar{\mathfrak{u}}_n(t), \bar{\mathfrak{v}}_n(t))$ and
	\begin{align}
		\mathbb{V}(t,\bar{\mathfrak{z}}_n(t)) & = \langle \bar{\mathfrak{u}}_n(t), \bar{\mathfrak{v}}_n(t) \rangle_{L^2(B_{\mathcal{T}-t})} + \langle \bar{\mathfrak{v}}_n(t), \bar{f}_n(t) \rangle_{L^2(B_{\mathcal{T}-t})} \nonumber\\
		&  \quad + \langle \partial_x \bar{\mathfrak{v}}_n(t), \partial_x  \bar{f}_n(t) \rangle_{L^2(B_{\mathcal{T}-t})}  + \langle \bar{\mathfrak{v}}_n(t), \bar{g}_n(t) \rangle_{L^2(B_{\mathcal{T}-t})}  \nonumber\\
		& \quad +  \langle \partial_x \bar{\mathfrak{v}}_n(t), \partial_x \bar{g}_n(t) \rangle_{L^2(B_{\mathcal{T}-t})} \nonumber\\
		& = : \mathbb{V}_f(t,\bar{\mathfrak{z}}_n(t)) + \mathbb{V}_g(t,\bar{\mathfrak{z}}_n(t)). \nonumber
	\end{align}
	We estimate $\mathbb{V}_f(t,\bar{\mathfrak{z}}_n(t))$ and  $\mathbb{V}_g(t,\bar{\mathfrak{z}}_n(t))$ separately as follows.
	Since $\mathcal{T}-t >  2R$, for every $t \in[0,R]$ and $\varphi(y), \varphi^\prime(y) =0$ for $y \notin \overline{B_{2R}}$, we have
	\begin{align}
		\int_{0}^{t} \mathbb{V}_f(r,\bar{\mathfrak{z}}(r)) \, dr &   = \int_{0}^{t} \bigg[  \int_{B_{2R}}  \left\{  \varphi(y) \mathfrak{u}_n(r,y)   \varphi(y) \mathfrak{v}_n(r,y)  +  \varphi(y) \mathfrak{v}_n(r,y)   \bar{f}_n(r,y)   \right. \nonumber\\
		&  \quad \left.  +  \varphi^\prime(y) \mathfrak{v}_n(r,y)   \partial_x \bar{f}_n(r,y) + \varphi(y) \partial_x \mathfrak{v}_n(r,y)    \partial_x \bar{f}_n(r,y) \right\} \, dy       \bigg] \, dr \nonumber\\
		& \lesssim_{\varphi, \varphi^\prime}  \int_{0}^{t} l(r,\bar{\mathfrak{z}}_n(r)) \, dr   +  \int_{0}^{t}  \| \bar{f}_n(r) \|_{H^1(B_{2R})}^2  \, dr, \nonumber
	\end{align}
	and
	\begin{align}
		\int_{0}^{t} \mathbb{V}_g(r,\bar{\mathfrak{z}}(r)) \, dr & = \int_{0}^{t} \left(  \langle \bar{\mathfrak{v}}_n(r), \bar{g}_n(r) \rangle_{L^2(B_{\mathcal{T}-r})} +  \langle \partial_x \bar{\mathfrak{v}}_n(r), \partial_x \bar{g}_n(r) \rangle_{L^2(B_{\mathcal{T}-r})} \right) \, dr  \nonumber\\
		&   = \int_{0}^{t} \left(  \langle \bar{\mathfrak{v}}_n(r), \bar{g}_n(r) \rangle_{L^2(B_{2R})} +  \langle \partial_x \bar{\mathfrak{v}}_n(r), \partial_x \bar{g}_n(r) \rangle_{L^2(B_{2R})} \right) \, dr. \nonumber
	\end{align}
	Let us estimate the terms involving $\bar{f}_n$ first. Since $u_n,u_h$ takes values on manifold $M$, by using the properties of $\varphi$ and invoking interpolation inequality \eqref{interpolation1},  as pursued in Lemma \ref{lem-lip}, followed  by Lemma \ref{lem-aprioriEnergyEst}  we deduce that
	\begin{align}\label{prop-strngConv-t3i}
		\|\bar{f}_n(r)\|_{L^2(B_{2R})}^2 & \lesssim_{\varphi,\varphi^\prime, \varphi^{\prime\prime}} \| A_{u_n(r)}(v_n(r), v_n(r)) - A_{u_h(r)}(v_n(r), v_n(r)) \|_{L^2(B_{2R})}^2 \nonumber\\
		& \quad + \| A_{u_h(r)}(v_n(r), v_n(r)) - A_{u_h(r)}(v_n(r), v_h(r)) \|_{L^2(B_{2R})}^2 \nonumber\\
		& \quad + \| A_{u_h(r)}(v_n(r), v_h(r)) - A_{u_h(r)}(v_h(r), v_h(r)) \|_{L^2(B_{2R})}^2 \nonumber\\
		& \quad + \| A_{u_n(r)}(\partial_x u_n(r), \partial_x u_n(r))  -  A_{u_h(r)}(\partial_x u_n(r), \partial_x u_n(r)) \|_{L^2(B_{2R})}^2  \nonumber\\
		& \quad + \| A_{u_h(r)}(\partial_x u_n(r), \partial_x u_n(r))  -  A_{u_h(r)}(\partial_x u_n(r), \partial_x u_h(r)) \|_{L^2(B_{2R})}^2  \nonumber\\
		& \quad + \| A_{u_h(r)}(\partial_x u_n(r), \partial_x u_h(r))  -  A_{u_h(r)}(\partial_x u_h(r), \partial_x u_h(r)) \|_{L^2(B_{2R})}^2  \nonumber\\
		& \quad + \| u_n(r) - u_h(r) \|_{L^2(B_{2R})}^2  + 2 \| \partial_x u_n(r) - \partial_x u_h(r) \|_{L^2(B_{2R})}^2  \nonumber\\
		& \lesssim_{ L_{A}, B_{A}, R   }  \| u_n(r) - u_h(r)\|_{L^2(B_{2R})}^2 \| v_n(r)\|_{L^\infty(B_{2R})}^4 \nonumber\\
		& \quad + \| v_n(r) - v_h(r)\|_{L^2(B_{2R})}^2 \left( \| v_n(r)\|_{L^\infty(B_{2R})}^2 +  \| v_h(r)\|_{L^\infty(B_{2R})}^2   \right)\nonumber\\
		& \quad + \| u_n(r) - u_h(r)\|_{L^2(B_{2R})}^2 \| \partial_x u_n(r)\|_{L^\infty(B_{2R})}^4 \nonumber\\
		& \quad + \| \partial_x u_n(r) - \partial_x u_h(r)  \|_{L^2(B_{2R})}^2 \left( \| \partial_x u_n(r)\|_{L^\infty(B_{2R})}^2 +  \| \partial_x u_h(r)\|_{L^\infty(B_{2R})}^2   \right) \nonumber\\
		& \quad + \| u_n(r) - u_h(r) \|_{L^2(B_{2R})}^2  + 2 \| \partial_x u_n(r) - \partial_x u_h(r) \|_{L^2(B_{2R})}^2  \nonumber\\
		& \lesssim_{ L_{A}, B_{A}, R ,k_e,\mathcal{B} } ~ \| \mathfrak{z}_n(r)\|_{\mathcal{H}(B_{2R})}^2  \lesssim l(r, \mathfrak{z}_n(r)).
	\end{align}
	Similarly by using the interpolation inequality \eqref{interpolation1} and Lemma \ref{lem-aprioriEnergyEst}, based on the computation of \eqref{lip-t3},  we get
	\begin{equation}\nonumber
		\| \partial_x \bar{f}_n(r)\|_{L^2(B_{2R})}^2 \lesssim_{ L_{A}, B_{A}, R ,k_e,\mathcal{B} } ~ l(r, \mathfrak{z}_n(r)),
	\end{equation}
	where the constant of inequality is independent of $n$ but depends on the properties of $\varphi$ and its first two derivatives,
	consequently, we have, for some $C_{\bar{f}} >0$,
	\begin{align}\label{prop-strngConv-t4}
		& \int_{0}^{t} \| \bar{f}_n(r) \|_{H^1(B_{2R})}^2 \, dr  \leq C_{\bar{f}}  \int_{0}^{t} l(r, \mathfrak{z}_n(r)) \, dr, \quad \forall t \in [0,R].
	\end{align}
	Now we move to the crucial estimate of integral involving $\bar{g}_n$. It is the part where we  follow the idea of \cite[Proposition 3.4]{Chueshov+Millet_2010} and \cite[Proposition 4.4]{Duan+Millet_2009}.
	Let $m$ be a natural number, whose value will be set later. Define the following partition of $[0,R]$,
	\begin{equation}\nonumber
		\left\{ 0, \frac{1 \cdot \rT}{2^m}, \frac{2 \cdot \rT}{2^m},\cdots, \frac{2^m \cdot \rT}{2^m}   \right\},
	\end{equation}
	and set
	\begin{equation}\nonumber
		r_m := \frac{(k+1) \cdot \rT}{2^m} \textrm{ and } t_{k+1} := \frac{(k+1) \cdot \rT}{2^m} \textrm{ if } r \in \left[ \frac{k \cdot \rT}{2^m}, \frac{(k+1) \cdot \rT}{2^m} \right).
	\end{equation}
	 Now  observe that, for every $t \in [0,R]$,
	\begin{align}\label{prop-strngConv-t5}
		& \int_{0}^{t} \langle \bar{\mathfrak{v}}_n(r), \bar{g}_n(r) \rangle_{H^1(B_{2R})}  \, dr  \nonumber\\
		& \quad =  \int_{0}^{t}  \langle \bar{\mathfrak{v}}_n(r) , \varphi (Y(u_n(r)) - Y(u_h(r)) )\dot{h}_n(r)  \rangle_{H^1(B_{2R})} \, dr \nonumber \\
		& \quad \quad + \int_{0}^{t}  \langle \bar{\mathfrak{v}}_n(r) - \bar{\mathfrak{v}}_n(r_m), \varphi Y(u_h(r)) (\dot{h}_n(r) - \dot{h}(r)) \rangle_{H^1(B_{2R})} \, dr \nonumber\\
		& \quad \quad + \int_{0}^{t}  \langle \bar{\mathfrak{v}}_n(r_m), \varphi (Y(u_h(r)) - Y(u_h(r_m)) )  (\dot{h}_n(r) - \dot{h}(r)) \rangle_{H^1(B_{2R})} \, dr \nonumber\\
		& \quad \quad + \int_{0}^{t}  \langle \bar{\mathfrak{v}}_n(r_m),  \varphi Y(u_h(r_m))  (\dot{h}_n(r) - \dot{h}(r)) \rangle_{H^1(B_{2R})} \, dr \nonumber\\
		& \quad =: G^{n,m}_1(t) + G^{n,m}_2(t) + G^{n,m}_3(t) + G^{n,m}_4(t).
	\end{align}
	For $G^{n,m}_1$, Lemmata \ref{hsop},  \ref{lem-Y-LDP}  and \ref{lem-aprioriEnergyEst}   followed by \eqref{prop-strngConv-normChngVarphi} imply
	\begin{align}\label{prop-strngConv-t6}
		& |G^{n,m}_1(t)|   \lesssim_{\varphi} \int_{0}^{t} \|\bar{\mathfrak{v}}_n(r) \|_{H^1(B_{2R})}^2 \, dr + \int_{0}^{t} \|Y(u_n(r)) - Y(u_h(r))\|_{H^1(B_{2R})}^2  \|\dot{h}_n(r)\|_{\rkhs}^2  \, dr \nonumber\\
		& \quad \lesssim_R \int_{0}^{t} \|\bar{\mathfrak{v}}_n(r) \|_{H^1(B_{2R})}^2 \, dr  \nonumber\\
		& \quad \quad + \int_{0}^{t} \|u_n(r) - u_h(r)\|_{H^1(B_{2R})}^2  \left(1+ \|u_n(r) \|_{H^1(B_{2R})}^2 + \|u_h(r)\|_{H^1(B_{2R})}^2 \right) \|\dot{h}_n(r)\|_{\rkhs}^2  \, dr \nonumber\\
		& \quad \lesssim_{\mathcal{B}} \int_{0}^{t}  \left( 1+ l(r,\mathfrak{z}_n(r)) \right)  \left( 1+ \|\dot{h}_n(r)\|_{\rkhs}^2 \right)   \, dr, \qquad \forall t \in [0,R].
	\end{align}
	To  estimate  $G^{n,m}_2(t)$  we invoke $\langle h,k \rangle_{H^1(B_{2R})}  \leq \|h\|_{L^2(B_{2R})} \|k\|_{H^2(2R))}$  followed by the H\"older inequality and  Lemmata \ref{hsop}, \ref{lem-Y-LDP}, and  \ref{lem-vnLips}     to get, for every $t \in [0,R]$,
	\begin{align}
		|G^{n,m}_2(t)| & \lesssim_{R, \varphi} \int_{0}^{t}  \| \mathfrak{v}_n(r) - \mathfrak{v}_n(r_m)\|_{L^2(B_{2R})} \| Y(u_h(r)) \|_{H^2(B_{2R})} \| \dot{h}_n(r) - \dot{h}(r)\|_{\rkhs} \, dr   \nonumber\\
		& \lesssim_R \left( \int_{0}^{t}   \|\mathfrak{v}_n(r) - \mathfrak{v}_n(r_m)\|_{L^2(B_{2R})}^2 \, dr \right)^{\frac{1}{2}}   \nonumber\\
		& \quad \times \left( \int_{0}^{t}  \left[ 1 + \|u_h(r)\| _{H^2(B_{2R})} ^4  \right] \| \dot{h}_n(r) - \dot{h}(r)\|_{\rkhs}^2 \, dr  \right)^{\frac{1}{2}}  \nonumber\\
		& \lesssim  \sqrt{M_{\mu}}  \left( \int_{0}^{t}   |r-r_m|  \, dr \right)^{\frac{1}{2}} \sup_{r\in [0,\frac{\mathcal{T}}{2}]} \left[ 1 + \|u_h(r)\| _{H^2(B_{\mathcal{T}-r})}^4  \right]  \nonumber\\
		& \lesssim  \frac{\rT  \sqrt{M_{\mu}}}{2^{m/2}}   \sup_{r\in [0,\frac{\mathcal{T}}{2}]} \left[  1+ (l(r,z_h(r)))^2 \right]  \leq  \frac{\rT \sqrt{M_{\mu}}}{2^{m/2}}  ~ (1+\mathcal{B}^2), \nonumber
	\end{align}
	where  in the last  and the second last step we have used, respectively, Lemma \ref{lem-aprioriEnergyEst} for $\mathcal{T}$ instead of $T$ and
	\begin{align}
		& \left( \int_{0}^{t}  |r-r_m|  \, dr \right)^{\frac{1}{2}} \leq \left( \int_{0}^{\rT}  |r-r_m|  \, dr \right)^{\frac{1}{2}} = \left(  \sum_{k=1}^{2^m} \int_{t_{k-1}}^{t_k}  \biggl\vert r-\frac{k \rT}{2^m} \biggl\vert  \, dr \right)^{\frac{1}{2}} \leq \frac{\rT}{2^{m/2}}. \nonumber
	\end{align}
	Moreover, in the third last step we have also applied the following: since $R=T$ and  $\dot{h}_n \to \dot{h}$ weakly in $L^2(0,T;\rkhs)$,  the sequence  $\dot{h}_n - \dot{h} $ is bounded in $L^2(0,T;\rkhs)$ i.e. $ \exists M_{\mu} >0$ such that
	\begin{equation}\label{hBound}
		\int_{0}^{T}  \| \dot{h}_n(r) - \dot{h}(r)\|_{\rkhs}^2 \, dr  \leq M_{\mu}, \quad \mbox{ for all }n.
	\end{equation}
	Before moving to $G^{n,m}_3(t)$  note that, since $2R = \frac{\mathcal{T}}{2}$, due to  Remark  \ref{rem-aprioriEnergyEst}, for every  $s , t \in [0,\frac{\mathcal{T}}{2}]$,
	\begin{align}
		& \|u_h(t) - u_h(s)\|_{H^1(B_{2R})}  \leq \int_{s}^{t} \|v_h(r)\|_{H^1(B_{2R})}  \, dr   \lesssim \sqrt{\mathcal{B}} |t-s|. \nonumber
	\end{align}
	Consequently, by the H\"older inequality followed by Lemmata \ref{hsop}, \ref{lem-vnLips}, and  \ref{lem-Y-LDP} we obtain
	\begin{align}
		|G^{n,m}_3(t)|  &  \lesssim_{\varphi}   \left( \int_{0}^{t} \left[ \| v_n(r_m) \|_{H^1(B_{2R})}^2 + \| v_h(r_m)\|_{H^1(B_{2R})}^2 \right]   \, dr \right)^{\frac{1}{2}}   \nonumber\\
		& \qquad \qquad \times \left( \int_{0}^{t}  \| Y(u_h(r)) - Y(u_h(r_m)) \|_{H^1(B_{2R})}^2   \|\dot{h}_n(r) - \dot{h}(r)) \|_{\rkhs}^2 \, dr  \right)^{\frac{1}{2}}  \nonumber\\
		& \lesssim_{T, \mathcal{B}} \left( \int_{0}^{t} \| u_h(r) -   u_h(r_m) \|_{H^1(B_{2R})}^2 \left[ 1+ \| u_h(r)  \|_{H^1(B_{2R})}^2 + \|   u_h(r_m) \|_{H^1(B_{2R})}^2 \right] \right. \nonumber\\
		& \qquad \left. \times  \|\dot{h}_n(r) - \dot{h}(r) \|_{\rkhs}^2 \, dr  \right)^{\frac{1}{2}}   \nonumber\\
		& \lesssim_{T, \mathcal{B}}  \left( \int_{0}^{t}  | r-r_m | ~ \|\dot{h}_n(r) - \dot{h}(r) \|_{\rkhs}^2 \, dr  \right)^{\frac{1}{2}} \nonumber\\
		& \leq  \left( \sum_{k=1}^{2^m} \int_{t_{k-1}}^{t_k}  \biggl\vert r- \frac{k \rT}{2^m} \biggl\vert ~ \|\dot{h}_n(r) - \dot{h}(r) \|_{\rkhs}^2 \, dr  \right)^{\frac{1}{2}}      \nonumber\\
		& \leq \sqrt{\frac{\rT}{2^m}} \left( \int_{0}^{t}  \|\dot{h}_n(r) - \dot{h}(r) \|_{\rkhs}^2 \, dr  \right)^{\frac{1}{2}}  \leq \sqrt{R \frac{M_{\mu}}{2^m}}, \quad t \in [0,R].
\label{eqn-G_n^3}
	\end{align}
	Finally we start estimating  $G^{n,m}_4(t)$ by noting that for every $t \in [0,\rT]$,
	$$ \textrm{there exists} \quad  k_t \leq 2^m\quad \textrm{such that} \quad t \in \left[   \frac{(k_t-1) \cdot \rT}{2^m},  \frac{k_t \cdot \rT}{2^m} \right). $$
	Note that  on such interval $r_m  = \frac{k_t \cdot \rT}{2^m}$. Then by Lemma \ref{lem-aprioriEnergyEst}  we have
	\small \begin{align}\label{prop-strngConv-t9}
		& |G^{n,m}_4(t)|    \leq      \biggl\vert \sum_{k=1}^{k_t-1} \int_{t_{k-1}}^{t_k}  \left\langle \bar{\mathfrak{v}}_n \left( \frac{k \cdot \rT}{2^m}\right),  \varphi Y\left(u_h\left( \frac{k \cdot \rT}{2^m}\right) \right)  (\dot{h}_n(r) - \dot{h}(r)) \right\rangle_{H^1(B_{2R})} \, dr  \nonumber\\
		& \quad + \int_{t_{k_t-1}}^{t}  \left\langle \bar{\mathfrak{v}}_n \left( \frac{ (k_t-1) \cdot \rT}{2^m}\right),  \varphi Y\left(u_h\left( \frac{ (k_t-1) \cdot \rT}{2^m}\right) \right)  (\dot{h}_n(r) - \dot{h}(r)) \right\rangle_{H^1(B_{2R})} \, dr  \biggl\vert \nonumber\\
		& \leq \sum_{k=1}^{2^m}   \biggl\vert   \left\langle \bar{\mathfrak{v}}_n \left( \frac{k \cdot \rT}{2^m}\right),  \varphi Y\left(u_h\left( \frac{k \cdot \rT}{2^m}\right) \right)  \int_{t_{k-1}}^{t_k}  (\dot{h}_n(r) - \dot{h}(r)) \, dr  \right\rangle_{H^1(B_{2R})}   \biggl\vert \nonumber\\
		&  + \sup_{1 \leq k \leq 2^m} \sup_{t_k \leq t \leq t_{k-1} } \biggl\vert  \left\langle \bar{\mathfrak{v}}_n \left( \frac{ (k-1) \cdot \rT}{2^m}\right),  \varphi Y\left(u_h\left( \frac{ (k-1) \cdot \rT}{2^m}\right) \right) \int_{t_{k-1}}^{t}  (\dot{h}_n(r) - \dot{h}(r)) \, dr  \right\rangle_{H^1(B_{2R})}  \biggl\vert \nonumber\\
		& \leq \sum_{k=1}^{2^m}   \biggl\Vert  \bar{\mathfrak{v}}_n \left( \frac{k \cdot \rT}{2^m}\right) \biggl\Vert_{H^1(B_{2R})} \biggl\Vert  \varphi Y\left(u_h\left( \frac{k \cdot \rT}{2^m}\right) \right)  \int_{t_{k-1}}^{t_k}  (\dot{h}_n(r) - \dot{h}(r)) \, dr  \biggl\Vert_{H^1(B_{2R})}   \nonumber\\
		&  + \sup_{1 \leq k \leq 2^m} \sup_{t_k \leq t \leq t_{k-1} } \biggl\Vert \bar{\mathfrak{v}}_n \left( \frac{ (k-1) \cdot \rT}{2^m}\right)\biggl\Vert_{H^1(B_{2R})}  \biggl\Vert  \varphi Y\left(u_h\left( \frac{ (k-1) \cdot \rT}{2^m}\right) \right) \int_{t_{k-1}}^{t}  (\dot{h}_n(r) - \dot{h}(r)) \, dr \biggl\Vert_{H^1(B_{2R})} \nonumber\\
		& \lesssim_{\varphi, \mathcal{B}} \sum_{k=1}^{2^m}   \biggl\Vert   Y\left(u_h\left( \frac{k \cdot \rT}{2^m}\right) \right)  \int_{t_{k-1}}^{t_k}  (\dot{h}_n(r) - \dot{h}(r)) \, dr  \biggl\Vert_{H^1(B_{2R})} \nonumber\\
		&  + \sup_{1 \leq k \leq 2^m} \sup_{t_k \leq t \leq t_{k-1} }   \biggl\Vert   Y\left(u_h\left( \frac{ (k-1) \cdot \rT}{2^m}\right) \right) \int_{t_{k-1}}^{t}  (\dot{h}_n(r) - \dot{h}(r)) \, dr \biggl\Vert_{H^1(B_{2R})} \nonumber\\
		& =:  G_4^{n,m,1} + G_4^{n,m,2},
	\end{align}  \normalsize
	where the  right hand side does not depend on $t$.  By  invoking Lemmata \ref{hsop}, \ref{lem-Y-LDP},  the H\"older inequality,  Lemma \ref{lem-aprioriEnergyEst} and \eqref{hBound} we estimate $G_4^{n,m,2}$  as
	\begin{align}
	G_4^{n,m,2} &\lesssim_{R} \sup_{1 \leq k \leq 2^m} \sup_{t_k \leq t \leq t_{k-1} }   \biggl\Vert  Y\left(u_h\left( \frac{ (k-1) \cdot \rT}{2^m}\right) \right)  \biggl\Vert_{ H^1(B_{2R})}   \left( \int_{t_{k-1}}^{t}  \| \dot{h}_n(r) - \dot{h}(r)\|_{\rkhs}  \, dr \right)  \label{eqn-G_n^4,2}\\
	& \leq \sup_{1 \leq k \leq 2^m} \sup_{t_k \leq t \leq t_{k-1} } \left[ 1+       \biggl\Vert  u_h\left( \frac{ (k-1) \cdot \rT}{2^m}\right)   \biggl\Vert_{ H^1(B_{2R})}  \right]  \left( \int_{t_{k-1}}^{t}  \| \dot{h}_n(r) - \dot{h}(r)\|_{\rkhs}  \, dr \right) \nonumber\\
	& \lesssim_{\mathcal{B}} \sup_{1 \leq k \leq 2^m}  \left(\frac{R}{2^m} \right)^{\frac{1}{2}} \left( \int_{t_{k-1}}^{t_k}  \| \dot{h}_n(r) - \dot{h}(r)\|_{\rkhs}^2  \, dr \right)^{\frac{1}{2}}  \leq  \sqrt{R \frac{M_\mu}{2^m}}.  \nonumber
	\end{align}
	For $G_4^{n,m,1}$ recall that,  by Lemma \ref{hsop}, for every $\phi \in H^1(B(x,r))$ the multiplication operator
	$$ Y(\phi)  : K\ni k \mapsto  Y(\phi) \cdot k \in H^1(B(x,r)), $$
	is $\gamma$-radonifying and hence compact. Hence by   Lemma \ref{lem-compactness} we infer that for every $k$,
	\begin{align}\label{prop-strngConv-t11}
		& \biggl\Vert  Y\left(u_h\left( \frac{k \cdot \rT}{2^m}\right) \right)  \int_{t_{k-1}}^{t_k}  (\dot{h}_n(r) - \dot{h}(r)) \, dr  \biggl\Vert_{H^1(B_{2R})}  \to 0 \textrm{ as } n \to 0.
	\end{align}
	Hence each term of the sum in $G_4^{n,m,1}$ converges to $0$ as $n \to \infty$.
	Consequently, by substituting the computation between \eqref{prop-strngConv-t6} and \eqref{prop-strngConv-t9} into \eqref{prop-strngConv-t5} we obtain
	 \begin{align}
		\int_{0}^{t} &  \langle \bar{\mathfrak{v}}_n(r), \bar{g}_n(r) \rangle_{H^1(B_{2R})}  \, dr  \lesssim_{R, L_{A}, B_{A}, \varphi, \mathcal{B}}  \int_{0}^{t} \left( 1+  l(r,\mathfrak{z}_n(r)) \right) \left( 1+ \|\dot{h}_n(r) \|_{\rkhs}^2  \right)  \, dr \nonumber\\
		& \quad +  \sqrt{ R\frac{M_{\mu}}{2^m}} +  \sum_{k=1}^{2^m}  \biggl\Vert  Y\left(u_h\left( \frac{k \cdot \rT}{2^m}\right) \right)  \int_{t_{k-1}}^{t_k}  (\dot{h}_n(r) - \dot{h}(r)) \, dr  \biggl\Vert_{H^1(B_{2R})}.  \nonumber
	\end{align}
	Therefore, with \eqref{prop-strngConv-t4} and \eqref{prop-strngConv-normChngVarphi},  from \eqref{prop-strngConv-t2i}  we have
	\begin{align}
		 l(t,\mathfrak{z}_n(t)) & \lesssim  \int_{0}^{t}\left( 1+  l(r,\mathfrak{z}_n(r)) \right)  \left( 1+ \|\dot{h}_n(r) \|_{\rkhs}^2  \right)  \, dr  +    \sqrt{R \frac{M_{\mu}}{2^m}}   \nonumber\\
		& \quad +  \sum_{k=1}^{2^m}  \biggl\Vert  Y\left(u_h\left( \frac{k \cdot \rT}{2^m}\right) \right)  \int_{t_{k-1}}^{t_k}  (\dot{h}_n(r) - \dot{h}(r)) \, dr  \biggl\Vert_{H^1(B_{2R})}, \quad  t \in [0,R], \nonumber
	\end{align}
	and by the Gronwall Lemma, with the observation that all the terms in right hand side except the first are independent of $t$,  and $h_n \in S_{\mathcal{M}}$ further we get
	 \begin{align}\label{prop-strngConv-t10}
		 \sup_{t \in [0,R]} l(t,\mathfrak{z}_n(t)) & \lesssim   e^{T+\mathcal{M}}  \left\{ \sqrt{R \frac{M_{\mu}}{2^m}}   + \sum_{k=1}^{2^m}  \biggl\Vert  Y\left(u_h\left( \frac{k \cdot \rT}{2^m}\right) \right)  \int_{t_{k-1}}^{t_k}  (\dot{h}_n(r) - \dot{h}(r)) \, dr  \biggl\Vert_{H^1(B_{2R})}  \right\} .
	\end{align}
	Hence, for given any $\alpha >0$ we can choose $m$ such that
	\begin{equation}\nonumber
		\sqrt{R \frac{M_\mu}{2^m}}    < \alpha, \;\; \mbox{ for every } n \in \mathbb{N}.
	\end{equation}
	Thus, for such chosen $m$, due to \eqref{prop-strngConv-t11} by taking  $n \to \infty$ in \eqref{prop-strngConv-t10} we conclude that, for every $\alpha >0$,
	\begin{equation}\label{eqn-5.26}
		0 <  \limsup\limits_{n \to \infty} \sup_{t \in [0,R]} l(t,\mathfrak{z}_n(t)) < \alpha.
	\end{equation}
	Therefore, due to \eqref{prop-strngConv-normChngVarphi} we  conclude the proof of assertion \eqref{prop-strngConv-LimitBumpi}.
\end{proof}
Hence, the Proposition \ref{prop-strngConv} follows.
\end{proof}
Now we come back to the proof of \textbf{Statement 1}. Previous proposition shows that every sequence in $K_{\mathcal{M}}$ has a convergent subsequence. Hence $K_{\mathcal{M}}$ is sequentially relatively compact subset of $\mathcal{X}_{T}$. Let $\{ z_n\}_{n \in \mathbb{N}} \subset K_{\mathcal{M}}$ which converges to $z \in \mathcal{X}_{T}$. But Proposition \ref{prop-strngConv} shows that there exists a subsequence $\{ u_{n_k} \}_{k \in \mathbb{N}}$ which converges to some element $z_h$ of $K_{\mathcal{M}}$ in the strong topology of $\mathcal{X}_{T}$. Hence $z= z_h$ and $K_{\mathcal{M}}$ is a closed subset of $\mathcal{X}_{T}$. This completes the proof of \textbf{Statement 1}.
Below is a basic result that we have used in the proof of Proposition \ref{prop-strngConv}. A statement of this sort can be found in \cite{Chueshov+Millet_2010}, see the proof of Proposition 3.4.

\begin{Lemma}\label{lem-compactness}
Let $X,Y$ be separable Hilbert spaces and let $C:X\to Y$ be a compact operator. Then the operator $K:L^2(0,T;X)\to C([0,T];Y)$ defined as
\[Kg(t)=C\int_0^tg(s)\,ds\,,\]
where the integral $\int_0^tg(s)\,ds$ is meant in the Bochner sense,
is compact. In particular, if $g_n\to g$ weakly in $L^2(0,T;X)$ then $Kg_n$ converges to $Kg$ strongly in $C([0,T];Y)$.
\end{Lemma}
\begin{proof}[\textbf{Proof of Lemma \ref{lem-compactness}}]
Clearly the operator $K$ is bounded.
Let $B_{L^2_TX}$ stand for the centered unit ball in $L^2(0,T;X)$. In order to prove compactness of $K$, in view of the Arzel{\`a}-Ascoli Theorem, see  \cite[Lemma 1]{Simon_1987} (and, for a very general formulation, \cite[Theorem 8.2.10]{Engelking_1989}),   we only need to show that the following two conditions hold. \\
(1) for every fixed $t\in[0,T]$ the set
\[\left\{Kg(t): \,g\in B_{L^2_TX}\right\}\subset Y\quad\mathrm{is\,\,relatively\,\,compact\,\,in\,\,}Y;\]
(2) the set of function
\[\left\{Kg:\,g\in B_{L^2_TX}\right\}\subset C([0,T];Y)\]
is uniformly equi-continuous.\\
To prove (1) we note first that for $t\in[0,T]$ fixed
\[\left|\int_0^tg(s)\,ds\right|_X\le \sqrt{T}, \; g\in B_{L^2_TX}\,.\]
Since $C:X\to Y$ is compact, the set
\[\left\{C\int_0^tg(s)\,ds:\,g\in B_{L^2_TX}\right\},\]
being an image of a bounded set in $X$, is relatively compact in $Y$. \\
To prove (2) it is enough to note that for any $g\in B_{L^2_TX}$ and $s,t\in[0,T]$
\[|Kg(t)-Kg(s)|\le \|C\|\int_s^t|g(r)|\,dr\le \|C\|\sqrt{|t-s|}\,.\]
Thus the proof of Lemma \ref{lem-compactness} is complete.
\end{proof}
The following Lemma is about the Lipschitz property of the difference of solutions that we have used in proving Proposition \ref{prop-strngConv}.
\begin{Lemma}\label{lem-vnLips}
	Let $R>0$, $I=[-a,a]$ and  $h_n, h \in S_{\mathcal{M}}$. There exists a positive constant $C := C(R,\mathcal{B}, M,a)$ such that for   $t,s \in [0,R] $  the following holds
	\begin{equation}\label{lem-vnLips-res}
		\sup_{x \in I} \| \mathfrak{v}_n(t) - \mathfrak{v}_n(s) \|_{L^2(B(x,2R))} \lesssim C~ |t-s|^{\frac{1}{2}},
	\end{equation}
	where $\mathfrak{v}_n$ is defined just after \eqref{prop-strngConv-un}.
\end{Lemma}
\begin{proof}[\textbf{Proof of Lemma \ref{lem-vnLips}}]
	Due to triangle inequality it is sufficient to show
	\begin{equation}\nonumber
	\sup_{x \in I} \|v_h(t) - v_h(s) \|_{L^2(B(x,2R))} \lesssim C |t-s|^{\frac{1}{2}}, \quad t,s \in [0,R].
	\end{equation}
	From the proof of existence part in Theorem \ref{thm-skeleton} we have, for $t,s \in [0,R]$,
	\begin{align}\label{lem-vnLips-t2}
		\|v_h(t) - v_h(s)\|_{L^2(B(x,2R))} & \leq \int_{s}^t  \|\partial_{xx} u_h(r)\|_{L^2(B(x,2R))} \, dr \nonumber\\
		& \quad + \int_{s}^t \left[  \|f_h(r)\|_{L^2B(x,2R))} + \|g_h(r)\|_{L^2(B(x,2R))} \right] \, dr ,
	\end{align}
	where
	\begin{align}
		f_h(r) & :=  A_{u_h(r)}(v_h(r), v_h(r))   - A_{u_h(r)}(\partial_x u_h(r), \partial_x u_h(r)) , \textrm{ and } g_h(r) := Y(u_h(r)) \dot{h}(r) . \nonumber
	\end{align}
	But, since $h \in S_{\mathcal{M}}$, the H\"older inequality followed by Lemmata \ref{hsop} and  \ref{lem-Y-LDP} yield
	\begin{align}
		 \sup_{x \in I} \int_{s}^t  \|g_h(r)\|_{L^2(B(x,2R))}  \, dr  & \leq |t-s|^{\frac{1}{2}}  \left( \int_{s}^t  \sup_{x \in I} \| Y(u_h(r)) \|_{L^2(B(x,2R))}^2  \|\dot{h}(r)\|_{\rkhs}^2  \, dr  \right)^{\frac{1}{2}} \nonumber\\
		& \lesssim_{R,\mathcal{B}, M}  |t-s|^{\frac{1}{2}}, \qquad\textrm{for}\quad  t,s \in [0,R], \nonumber
	\end{align}
	where we also applied \ref{lem-aprioriEnergyEst}  with $2R$ instead of $T$ and, based on \eqref{prop-strngConv-t3i}, we also have
	\begin{align}
		& \sup_{x \in I} \int_{s}^t \|f_h(r)\|_{L^2(B(x,2R))}  \, dr  \leq |t-s|^{\frac{1}{2}}  \left(   \int_{s}^t \sup_{x \in I} \| A_{u_h(r)} (v_h(r), v_h(r) )  \|_{L^2(B(x,2R))}^2  \, dr  \right)^{\frac{1}{2}} \nonumber\\
		& \quad + |t-s|^{\frac{1}{2}} \left(   \int_{s}^t  \sup_{x \in I} \| A_{u_h(s)} (\partial_x u_h(r), \partial_x u_h(r) )  \|_{L^2(B(x,2R))}^2  \, dr \right)^{\frac{1}{2}} \nonumber\\
		& \lesssim |t-s|^{\frac{1}{2}} \left(   \int_{s}^t  \sup_{x \in I}  \|u_h(r)\|_{L^2(B(x,2R))}^2  \{ \|v_h(s)\|_{L^2(B(x,2R))}^4   +   \|\partial_x u_h(s)\|_{L^2(B(x,2R))}^4 \}\, ds   \right)^{\frac{1}{2}} \nonumber\\
		& \lesssim  |t-s|~ \mathcal{B}^{\frac{3}{2}} \qquad\textrm{for} \quad t,s \in [0,R].  \nonumber
	\end{align}
	Finally, by the H\"older inequality and Lemma \ref{lem-aprioriEnergyEst}, we obtain, for $t,s \in [0,R]$,
	\begin{align}
		\sup_{x \in I}  \int_{s}^t   \|\partial_{xx} u_h(s)\|_{L^2(B(x,2R))} \, dr  & \leq \left( \int_{s}^t   1 \, dr \right)^{\frac{1}{2}} \left( \int_{s}^t   \sup_{x \in I}  \|u_h(r)\|_{H^2(B(x,2R))}^2  \, dr \right)^{\frac{1}{2}}  \nonumber\\
		&  \lesssim   \sqrt{\mathcal{B}} |t-s|. \nonumber
	\end{align}
	Therefore, by collecting the estimates in \eqref{lem-vnLips-t2} we get the required inequality \eqref{lem-vnLips-res} and we are done with the proof of Lemma \ref{lem-vnLips}.
\end{proof}

\subsection{Proof of Statement 2}\label{subsec:Statement 2}
Recall that $\mathcal{M}>0$ is given and a sequence $\{h_n \}_{n \in \mathbb{N}} \subset \mathscr{S}_{\mathcal{M}}$ is also given which converges in law to $h \in\mathscr{S}_{\mathcal{M}}$ as $\varepsilon_n \to 0$.
It will be useful to introduce the following notation for the processes
\begin{equation}\nonumber
	Z_n := (U_n,V_n) = J^{\varepsilon_n}\left( W + \frac{1}{\sqrt{\varepsilon_n}} h_n  \right), \quad z_n := (u_n,v_n) =  J^0(h_n).
\end{equation}
Let us fix any $x \in \mathbb{R}$. Then set $N$ a natural number such that
$$N > \| (u_0,v_0)\|_{\mathcal{H}(B(x, T))}.$$
For each $n \in \mathbb{N}$ we define an $ \mathfrak{F}_t $-stopping time
\begin{equation}\label{defn-taun}
	\tau_n(\omega) := \inf\{ t >0: \| Z_n(t,\omega) \|_{\mathcal{H}(B(x,T-t))} \geq N \} \wedge T, \quad \omega \in \Omega.
\end{equation}
Recall that  for $z = (u,v) \in \mathcal{H}_{loc}$, we set
\begin{align*}
\mathbf{e}(t,T;x,z)  =\frac{1}{2}\left\{ \| u\|^2_{H^2(B(x,T-t))} + \| v \|^2_{H^1(B(x,T-t))} \right\} = \frac{1}{2} \| z\|_{\mathcal{H}(B(x,T-t))}^2, \quad t \in [0,T] .
\end{align*}

In this framework we prove the following key result.
\begin{Proposition}\label{prop-weakConv}
	Let us define $\mathcal{Z}_n := Z_n - z_n$. For $\tau_n$ defined in \eqref{defn-taun} we have
	\begin{equation}\nonumber
	\lim\limits_{n \to \infty}\sup_{x \in [-a,a]}   \mathbb{E}\left[ \sup_{t \in [0,\frac{T}{2}]} \mathbf{e}(t \wedge \tau_n ,T; x, \mathcal{Z}_n(t \wedge \tau_n)) \right] =0.
	\end{equation}
\end{Proposition}
\begin{proof}[\textbf{Proof of Proposition \ref{prop-weakConv}}]
	Let us fix any $n \in \mathbb{N}$.
	To avoid complexity of notation we use an abuse of notation and write all the norms without reference of the centre of the ball $x$ and we will write $\mathbf{e}(t,z)$ in place of $\mathbf{e}(t,T;x,z)$ unless any conflict arises.
	First note that under our notation $Z_n = (U_n,V_n)$ and $z_n= (u_n,v_n)$, respectively, are the unique global strong solutions to the Cauchy problem
	\begin{equation}\nonumber
		\left\{
		\begin{aligned}
			&  \partial_{tt}U_n =  \partial_{xx}U_n +  A_{U_n}(\partial_t U_n, \partial_t U_n) - A_{U_n}(\partial_x U_n, \partial_x U_n) + Y(U_n) \dot{h}_n  ,   \\
			& \hspace{2cm} + \sqrt{\varepsilon_n} Y(U_n)\dot{W}, \\
			& \left( U_n(0), \partial_t U_n(0) \right) = \left( u_0 , v_0 \right), \quad \textrm{ where } V_n := \partial_t U_n,
		\end{aligned}\right.
	\end{equation}
	and
	\begin{equation}\nonumber
		\left\{
		\begin{aligned}
			&  \partial_{tt}u_n =  \partial_{xx}u_n +  A_{u_n}(\partial_t u_n, \partial_t u_n) - A_{u_n}(\partial_x u_n, \partial_x u_n)   +  Y(u_n) \dot{h}_n,   \\
			& \left( u_n(0), \partial_t u_n(0) \right) = \left( u_0 , v_0 \right),  \quad \textrm{ where } v_n := \partial_t u_n.
		\end{aligned}\right.
	\end{equation}
	Hence $\mathcal{Z}_n$ solves uniquely the Cauchy problem, with null initial data,
	\begin{align}
		\partial_{tt}\mathcal{U}_n & =  \partial_{xx}\mathcal{U}_n - A_{U_n}(\partial_x U_n, \partial_x U_n) +  A_{u_n}(\partial_x u_n, \partial_x u_n) + A_{U_n}(\partial_t U_n, \partial_t U_n)  \nonumber\\
		& \quad - A_{u_n}(\partial_t u_n, \partial_t u_n)  +  Y(U_n)\dot{h}_n  - Y(u_n)\dot{h}_n  + \sqrt{\varepsilon_n} Y(U_n) \dot{W}, \nonumber
	\end{align}
	where $\mathcal{V}_n := \partial_t \mathcal{U}_n$.  This is equivalent to say, for all $t \in [0,\frac{T}{2}]$,
	\begin{align}
		\mathcal{Z}_n(t) = \int_{0}^{t} S_{t-s}\left(\begin{array}{c}0\\ f_n(s)\end{array}\right)\, ds + \int_{0}^{t} S_{t-s} \left(\begin{array}{c}0\\g_n(s)\end{array}\right)\, dW(s).
	\end{align}
	Here
	\begin{align}
		f_n(s) & := - A_{U_n(s)}(\partial_x U_n(s), \partial_x U_n(s)) +  A_{u_n(s)}(\partial_x u_n(s), \partial_x u_n(s)) + A_{U_n(s)}(V_n(s), V_n(s))  \nonumber\\
		& \quad - A_{u_n(s)}(v_n(s), v_n(s))  +  Y(U_n(s)) \dot{h}_n(s)  - Y(u_n(s))\dot{h}_n(s), \nonumber
	\end{align}
	and
	\begin{equation}\nonumber
		g_n(s) := \sqrt{\varepsilon_n} Y(U_n(s)).
	\end{equation}
	Invoking Proposition \ref{prop_magic}, with that by taking $k=1,L=I,$ implies for every $t \in [0,\frac{T}{2}]$ and $x  \in [-a,a]$,
	\begin{align}\label{LDP-energyEstimate}
	\mathbf{e}(t,T;x,\mathcal{Z}_n(t)) & \leq  \int_{0}^{t} \mathbb{V}(r,\mathcal{Z}_n(r)) \, dr  + \int_{0}^{t} \langle \mathcal{V}_n(r), g_n(r)dW(r) \rangle_{L^2(B_{T-r})} \nonumber\\
	& \quad + \int_{0}^{t} \langle \partial_x \mathcal{V}_n(r), \partial_x [g_n(r)dW(r)] \rangle_{L^2(B_{T-r})},
	\end{align}
	with
	\begin{align}
	 \mathbb{V}(r,\mathcal{Z}_n(r))  & = \langle \mathcal{U}_n(r), \mathcal{V}_n(r) \rangle_{L^2(B_{T-r})} + \langle \mathcal{V}_n(r), f_n(r) \rangle_{L^2(B_{T-r})} + \langle \partial_x \mathcal{V}_n(r), \partial_x f_n(r) \rangle_{L^2(B_{T-r})}  \nonumber\\
	&  \quad  + \frac{1}{2} \sum_{j=1}^\infty \|g_n(r)e_j \|_{L^2(B_{T-r})}^2 +  \frac{1}{2} \sum_{j=1}^\infty \|\partial_x [g_n(r)e_j] \|_{L^2(B_{T-r})}^2, \nonumber
	\end{align}
	for a given sequence $\{e_j\}_{j \in \mathbb{N}}$ of orthonormal basis of $\rkhs$.

	Observe that, for any $\tau \in [0,T]$, by the Cauchy-Schwartz inequality
	\begin{align}\label{t1}
	 \sup_{0 \leq t \leq \tau} \int_{0}^{t \wedge \tau_n} \mathbb{V}(r,\mathcal{Z}_n(r)) \, dr &  \leq  2 \int_{0}^{\tau \wedge \tau_n} \mathbf{e}(r,\mathcal{Z}_n(r)) \, dr \\
	 &   + \frac{1}{2}\int_{0}^{\tau \wedge \tau_n} \left(  \| f_n(r) \|_{H^1(B_{T-r})}^2    +  \|g_n(r) \cdot\|_{\mathscr{L}_2(\rkhs,H^1(B_{T-r}))}^2  \right)  \, dr, \nonumber
	\end{align}
	where $g_n(r) \cdot$ denotes the multiplication operator in $\mathscr{L}_2(\rkhs,H^1(B_{T-r}))$, see Lemma~\ref{hsop}.

	Next, we define the process
	\begin{equation}\label{Yt}
		\mathcal{Y}(t) := \int_{0}^{t} \langle \mathcal{V}_n(r), g_n(r)dW(r) \rangle_{H^1(B_{T-r})}.
	\end{equation}
	By taking $\int_{0}^{t} \xi(r)\, dW(r)$ with
	$$ \xi(r): \rkhs \ni k \mapsto \langle \mathcal{V}_n(r), g_n(r)(k) \rangle_{H^1(B_{T-r})} \in \mathbb{R}, $$
	a Hilbert-Schmidt operator, note that
	\begin{equation}\nonumber
		\mathcal{Q}(t) := \int_{0}^{t} \xi(r) \circ \xi(r)^\star \, dr,
	\end{equation}
	is quadratic variation of $\mathbb{R}$-valued martingale $\mathcal{Y}$.  Thus
	\begin{align}\label{t2}
		\mathcal{Q}(t) & \leq \int_{0}^{t} \|\xi(r)\|_{\mathscr{L}_2(\rkhs,\mathbb{R})} \|\xi(r)^\star\|_{\mathscr{L}_2(\mathbb{R}, \rkhs)}\, dr =  \int_{0}^{t} \|\xi(r)\|_{\mathscr{L}_2(\rkhs,\mathbb{R})}^2\, dr \\
		& =  \int_{0}^{t} \sum_{j=1}^\infty |\xi(r)(e_j)|^2 \, dr =  \int_{0}^{t} \sum_{j=1}^\infty |\langle \mathcal{V}_n(r), g_n(r)(e_j) \rangle_{H^1(B_{T-r})} |^2 \, dr, \quad  t \in [0,\frac{T}{2}]. \nonumber
	\end{align}
	On the other hand by the Cauchy-Schwartz inequality
	\begin{align}
		& \sum_{j=1}^\infty |\langle \mathcal{V}_n(r), g_n(r)(e_j) \rangle_{H^1(B_{T-r})} |^2 \leq \|\mathcal{V}_n(r)\|_{H^1(B_{T-r})}^2 \| g_n(r) \cdot\|_{\mathscr{L}_2(\rkhs, H^1(B_{T-r})  )}^2. \nonumber
	\end{align}
	Therefore,
	\begin{equation}\label{QtEstimate}
		\mathcal{Q}(t) \leq \int_{0}^{t} \|\mathcal{V}_n(r)\|_{H^1(B_{T-r})}^2 \| g_n(r) \cdot \|_{\mathscr{L}_2(\rkhs, H^1(B_{T-r})  )}^2 \, dr, \quad  t \in [0,\frac{T}{2}].
	\end{equation}
	Invoking the Davis inequality with \eqref{QtEstimate} followed by the Young inequality gives
	\begin{align}\label{t3}
		\mathbb{E}& \left[ \sup_{0 \leq t \leq \tau} |\mathcal{Y}(t \wedge \tau_n)|\right]  \leq 3 \mathbb{E}\left[ \sqrt{\mathcal{Q}(\tau \wedge \tau_n)} \right] \nonumber\\
		& \qquad \leq   3 \mathbb{E} \left[ \sup_{0 \leq t \leq \tau \wedge \tau_n} \|\mathcal{V}_n(t \wedge \tau_n )\|_{H^1(B_{T-t})} \left\{  \int_{0}^{\tau \wedge \tau_n}  \| g_n(r) \cdot\|_{\mathscr{L}_2(\rkhs, H^1(B_{T-r})  )}^2 \, dr \right\}^{\frac{1}{2}} \right] \nonumber\\
		& \qquad \leq  3 \mathbb{E} \left[ \varepsilon \sup_{0 \leq t \leq \tau \wedge \tau_n} \|\mathcal{V}_n(t)\|_{H^1(B_{T-t})}^2 + \frac{1}{4 \varepsilon} \int_{0}^{\tau \wedge \tau_n}  \| g_n(r) \cdot\|_{\mathscr{L}_2(\rkhs, H^1(B_{T-r})  )}^2 \, dr  \right] \nonumber\\
		&\qquad \leq  6\varepsilon ~ \mathbb{E} \left[  \sup_{0 \leq t \leq \tau \wedge \tau_n} \mathbf{e}(t, \mathcal{Z}_n(t)) \right]  +  \frac{3}{4 \varepsilon} \mathbb{E}\left[ \int_{0}^{\tau \wedge \tau_n}  \| g_n(r) \cdot\|_{\mathscr{L}_2(\rkhs, H^1(B_{T-r})  )}^2 \, dr  \right].
	\end{align}
	By choosing $\varepsilon$ such that $6 \varepsilon = \frac{1}{2}$ and taking $\sup_{0 \leq s \leq t}$ followed by expectation $\mathbb{E}$ on the both sides of \eqref{LDP-energyEstimate} after evaluating it at $\tau \wedge \tau_n$ we obtain
	\begin{align}
		\mathbb{E}& \left[ \sup_{0 \leq s \leq t \wedge \tau_n} \mathbf{e}(s,\mathcal{Z}_n(s)) \right] \leq  \mathbb{E}\left[ \sup_{0 \leq s \leq t}\int_{0}^{s \wedge \tau_n} \mathbb{V}(r,\mathcal{Z}_n(r)) \, dr  \right]  + \mathbb{E}\left[ \sup_{0 \leq s \leq t} \mathcal{Y}(s \wedge \tau_n) \right]. \nonumber
	\end{align}
	Consequently, using \eqref{t1} and \eqref{t3} we infer that
	\begin{align}\label{t5}
		\mathbb{E} \left[ \sup_{0 \leq s \leq t \wedge \tau_n} \mathbf{e}(s,\mathcal{Z}_n(s)) \right]  &  \leq 4 \mathbb{E} \left[ \int_{0}^{t \wedge \tau_n}   \mathbf{e}(r,\mathcal{Z}_n(r)) \, dr \right]  + \mathbb{E}\left[\int_{0}^{t \wedge \tau_n} \| f_n(r) \|_{H^1(B_{T-r})}^2 \, dr \right] \nonumber\\
		& \quad + 19 \mathbb{E} \left[\int_{0}^{t \wedge \tau_n} \| g_n(r) \cdot\|_{\mathscr{L}_2(\rkhs, H^1(B_{T-r})  )}^2 \, dr \right].
	\end{align}
	Now since the Hilbert-Schmidt operator $g_n(r) \cdot$ is defined as
	\begin{equation}\nonumber
		\rkhs \ni k \mapsto g_n(r)\cdot k \in H^1(B_{T-r}),
	\end{equation}
	Lemmata \ref{hsop} and \ref{lem-Y-LDP} give,
	\begin{align}\label{t6}
	& \sup_{x\in[-a,a]} \mathbb{E} \left[\int_{0}^{t \wedge \tau_n} \| g_n(r) \cdot \|_{\mathscr{L}_2(\rkhs, H^1(B_{T-r})  )}^2 \, dr \right]   \lesssim_T   ~ \mathbb{E} \left[\int_{0}^{t \wedge \tau_n} \| \sqrt{\varepsilon_n} Y(U_n(r)) \|_{H^1(B_{T-r})}^2 \, dr \right] \nonumber\\
	&  \leq C_{Y,T}^2 \varepsilon_n ~  \mathbb{E} \left[\int_{0}^{t \wedge \tau_n} \left( 1+ \|Z_n(r)\|_{\mathcal{H}_{T-r}}^2  \right)    \, dr \right]   \lesssim_T  \varepsilon_n   ~ (1+N^2).
	\end{align}
	Here we observe that the constant in inequality \eqref{t6} does not depend on $a$ due to Lemma \ref{hsop}.
	To estimate the terms involving $f_n$ we have
	\begin{align}\label{t7}
		\|f_n(r)\|_{H^1(B_{T-r})}^2  & \lesssim \|A_{U_n(r)}(\partial_x U_n(r), \partial_x U_n(r)) -  A_{u_n(r)}(\partial_x u_n(r), \partial_x u_n(r)) \|_{H^1(B_{T-r})}^2 \nonumber\\
		& \quad  + \| A_{U_n(r)}(V_n(r), V_n(r)) - A_{u_n(r)}(v_n(r), v_n(r)) \|_{H^1(B_{T-r})}^2  \nonumber\\
		& \quad   +  \| Y(U_n(r))\dot{h}_n(r)   - Y(u_n(r))\dot{h}_n(r) \|_{H^1(B_{T-r})}^2 \nonumber\\
		& =: f_n^1 + f_n^2 + f_n^3.
	\end{align}
	By doing the computation based on Lemmata \ref{lem-lip}  and \ref{lem-Y-LDP} we obtain
	\begin{align}\label{t7-1}
		 f_n^1 & \lesssim \|A_{U_n(r)}(\partial_x U_n(r), \partial_x U_n(r)) -  A_{u_n(r)}(\partial_x U_n(r), \partial_x U_n(r)) \|_{H^1(B_{T-r})}^2 \nonumber\\
		& \quad + \|A_{u_n(r)}(\partial_x U_n(r), \partial_x U_n(r)) -  A_{u_n(r)}(\partial_x u_n(r), \partial_x U_n(r)) \|_{H^1(B_{T-r})}^2 \nonumber\\
		& \quad + \|A_{u_n(r)}(\partial_x u_n(r), \partial_x U_n(r)) -  A_{u_n(r)}(\partial_x u_n(r), \partial_x u_n(r)) \|_{H^1(B_{T-r})}^2 \nonumber\\
		& \lesssim_{T,x} \|U_n(r) - u_n(r)\|_{H^2(B_{T-r})}^2 \left(1 +\|\partial_x U_n(r)\|_{H^1(B_{T-r})}^2 +\|\partial_x U_n(r)\|_{H^1(B_{T-r})}^2  \right) \times \nonumber\\
		& \quad \times \left( 1+ \|u_n(r)\|_{H^2(B_{T-r})}^2 \right)  \nonumber\\
		& \quad + \|u_n(r)\|_{H^2(B_{T-r})}^2   \|\partial_x [U_n(r) - u_n(r)]\|_{H^1(B_{T-r})}^2\|\partial_x[u_n(r)]\|_{H^1(B_{T-r})}^2 \nonumber\\	
		& \lesssim  ~ \|\mathcal{Z}_n(r)\|_{\mathcal{H}_{T-r}}^2 ~ \left[  \left( 1 + \|Z_n(r)\|_{\mathcal{H}_{T-r}}^2 \right)  \left( 1 +  \|z_n(r)\|_{\mathcal{H}_{T-r}}^2 \right) + \|z_n(r)\|_{\mathcal{H}_{T-r}}^4 \right],
	\end{align}
	and, by similar calculations,
	\begin{align}\label{t7-2}
		& f_n^2  \lesssim_{T,x}  ~ \|\mathcal{Z}_n(r)\|_{\mathcal{H}_{T-r}}^2 ~ \left[  \left( 1 + \|Z_n(r)\|_{\mathcal{H}_{T-r}}^2 \right)  \left( 1 +  \|z_n(r)\|_{\mathcal{H}_{T-r}}^2 \right) + \|z_n(r)\|_{\mathcal{H}_{T-r}}^4 \right].
	\end{align}
	Furthermore, Lemmata \ref{lem-Y-LDP} and \ref{hsop} implies
	\begin{align}\label{t7-3}
		f_n^3 & \lesssim_{T,x} \| U_n(r)  - u_n(r)  \|_{H^1(B_{T-r})}^2 \left[ 1+ \| U_n(r)   |_{H^1(B_{T-r})}^2 + \| u_n(r)  |_{H^1(B_{T-r})}^2  \right] \|\dot{h}_n(r)\|_{\rkhs}^2 \nonumber\\
		& \lesssim  \|\mathcal{Z}_n(r)\|_{\mathcal{H}_{T-r}}^2 ~  \left( 1 + \|Z_n(r)\|_{\mathcal{H}_{T-r}}^2 + \|z_n(r)\|_{\mathcal{H}_{T-r}}^2  \right)   \|\dot{h}_n(r)\|_{\rkhs}^2.
	\end{align}
	Hence by substituting \eqref{t7-1}-\eqref{t7-3} in \eqref{t7} we get
	\begin{align}
		\|f_n(r)\|_{H^1(B_{T-r})}^2 & \lesssim_{T,x}   \|\mathcal{Z}_n(r)\|_{\mathcal{H}_{T-r}}^2 ~ \left[  \left( 1 +  \|Z_n(r)\|_{\mathcal{H}_{T-r}}^2 \right)  \left( 1 +  \|z_n(r)\|_{\mathcal{H}_{T-r}}^2 \right) + \|z_n(r)\|_{\mathcal{H}_{T-r}}^4 \right] \nonumber\\
		& \quad + \|\mathcal{Z}_n(r)\|_{\mathcal{H}_{T-r}}^2 ~  \left( 1 + \|Z_n(r)\|_{\mathcal{H}_{T-r}}^2 + \|z_n(r)\|_{\mathcal{H}_{T-r}}^2  \right)   \|\dot{h}_n(r)\|_{\rkhs}^2, \nonumber
	\end{align}
	consequently, the definition of $\tau_n$ and Lemma \ref{lem-aprioriEnergyEst} suggest
	\begin{align}\label{t8}
		 \mathbb{E}\left[\int_{0}^{t \wedge \tau_n} \| f_n(r) \|_{H^1(B_{T-r})}^2 \, dr \right] & \lesssim \mathbb{E}\left[\int_{0}^{t \wedge \tau_n} \left\{  \|\mathcal{Z}_n(r)\|_{\mathcal{H}_{T-r}}^2 ~ \left[  \left( 1 +  N^2 \right)  \left( 1 +  \mathcal{B}^2 \right) + \mathcal{B}^4 \right] \right. \right.  \nonumber\\
		& \quad \left. \left. +  \|\mathcal{Z}_n(r)\|_{\mathcal{H}_{T-r}}^2 ~  \left( 1 + N^2 + \mathcal{B}^2  \right)  ~ \left( 1 + \mathcal{B}^2 \right)  \|\dot{h}_n(r)\|_{\rkhs}^2   \right\} \, dr \right] \nonumber\\
		& \lesssim \mathbb{E}\left[\int_{0}^{t \wedge \tau_n}  ~ \mathbf{e}(r,T;x, \mathcal{Z}_n(r)) ~ C_{N,\mathcal{B}} \left(1  +  \|\dot{h}_n(r)\|_{\rkhs}^2  \right)  \, dr \right],
	\end{align}
	for some constant $C_{N,\mathcal{B}}>0$ depends on $N,\mathcal{B}$, where $\mathcal{B}$ is a function of $x$ which is bounded on compact sets.
	Then substitution of \eqref{t6} and \eqref{t8} in \eqref{t5} implies, here we write dependency of $\textbf{e}$ on $x$ and $T$ explicitly,
	\begin{align}
		\mathbb{E} & \left[ \sup_{0 \leq s \leq t \wedge \tau_n} \mathbf{e}(s,T;x,\mathcal{Z}_n(s)) \right]   \lesssim_{T,x}  \varepsilon_n ~ (1+N^2) \nonumber\\
		& \qquad + C_{N,\mathcal{B}}  \mathbb{E}  \left[ \int_{0}^{t \wedge \tau_n} [\sup_{0 \leq s \leq r \wedge \tau_n} \mathbf{e}(s, T;x,\mathcal{Z}_n(s))]  \left( 1+  \|\dot{h}_n(r)\|_{\rkhs}^2 \right)  \, dr  \right]. \nonumber
	\end{align}
	Therefore, invoking the stochastic Gronwall Lemma, see \cite[Lemma 3.9]{Duan+Millet_2009}, gives,
	\begin{align}\label{t9}
		\sup_{x \in [-a,a]} \mathbb{E}& \left[ \sup_{0 \leq s \leq t \wedge \tau_n} \mathbf{e}(s,T;x,\mathcal{Z}_n(s)) \right] \lesssim_{T,a}  \varepsilon_n  ~ (1+N^2)  \exp\left[ C_{N,\mathcal{B}} (T + \mathcal{M} ) \right].
	\end{align}
	Since $\varepsilon_n \to 0$ as $n \to \infty$ and
	$$ \mathbb{E} \left[ \sup_{0 \leq s \leq t \wedge \tau_n} \mathbf{e}(s,T;x,\mathcal{Z}_n(s)) \right] = \mathbb{E} \left[ \sup_{0 \leq s \leq t} \mathbf{e}(s\wedge \tau_n,T;x,\mathcal{Z}_n(s\wedge \tau_n)) \right], $$
	inequality \eqref{t9} gives $\lim\limits_{n \to \infty} \sup_{x \in [-a,a]} \mathbb{E} \left[ \sup_{0 \leq t \leq T} \mathbf{e}(t\wedge \tau_n,T;x,\mathcal{Z}_n(t\wedge \tau_n)) \right] =0$.   Hence we are done with the proof of Proposition \ref{prop-weakConv}.
\end{proof}

To proceed further we also need the following stochastic analogue of Lemma \ref{lem-aprioriEnergyEst}.
\begin{Lemma}\label{lem-stocAprioriEnergyEst}
	There exists a constant $\mathscr{B}: = \mathscr{B}(N,T,\mathcal{M}) > 0$ such that
	\begin{equation}\nonumber
	\limsup_{n \to\infty} ~ \sup_{x \in [-a,a]} \mathbb{E}\left[ \sup_{t \in [0,\frac{T}{2}]} \mathbf{e}(t \wedge \tau_n ,T;x, Z_n(t \wedge \tau_n)) \right]  \leq \mathscr{B}.
	\end{equation}
\end{Lemma}
\begin{proof}[\textbf{Proof of Lemma \ref{lem-stocAprioriEnergyEst}}]
	Let us fix sequence $\{e_j\}_{j \in \mathbb{N}}$ of orthonormal basis of $\rkhs$. Let us also fix any $n \in \mathbb{N}$. With the notation of this subsection, Proposition \ref{prop_magic}, with  $k=1,L=I$, implies for every $t \in [0,\frac{T}{2}]$ and $x \in [-a,a]$,
	\begin{align}
	\mathbf{e}(t,T;x, Z_n(t)) & \leq  \int_{0}^{t} \mathbb{V}(r,x,Z_n(r)) \, dr  + \int_{0}^{t} \langle V_n(r), g_n(r)dW(r) \rangle_{H^1(B(x,T-r))}, \nonumber
	\end{align}
with
\begin{align}
 \mathbb{V}(r,x,Z_n(r))  & := \langle U_n(r), V_n(r) \rangle_{L^2(B(x,T-r))} + \langle V_n(r), f_n(r) \rangle_{H^1(B(x,T-r))}  \nonumber\\
& \qquad  + \frac{1}{2} \sum_{j=1}^\infty \|g_n(r)e_j \|_{H^1(B(x,T-r))}^2 ,\nonumber
\end{align}
where for simplification we avoid writing the dependency of l.h.s on $T$ explicitly, and
\begin{align}
& f_n(r) := A_{U_n(r)}(V_n(r), V_n(r)) - A_{U_n(r)}(\partial_x U_n(r), \partial_x U_n(r))    +  Y(U_n(r)) \dot{h}_n(r), \nonumber\\
& g_n(r) := \sqrt{\varepsilon_n} Y(U_n(r)). \nonumber
\end{align}
Next, we set
\begin{equation}\nonumber
\psi_n (t,x) := \mathbb{E} \left[ \sup_{0 \leq s \leq t} \mathbf{e}(s \wedge \tau_n,T;x,Z_n(s \wedge \tau_n))\right], \quad t \in [0,T].
\end{equation}
Now, we intent to follow the procedure of Proposition \ref{prop-weakConv}.
By the Cauchy-Schwartz inequality, for $\tau \in [0,\frac{T}{2}]$ and $x \in [-a,a]$, we have
\begin{align}
\sup_{0 \leq t \leq \tau} & \int_{0}^{t \wedge \tau_n} \mathbb{V}(r,x,Z_n(r)) \, dr   \leq  2 \int_{0}^{\tau \wedge \tau_n} \mathbf{e}(r,T;x,Z_n(r)) \, dr   \nonumber\\
& \qquad \qquad + \frac{1}{2}\int_{0}^{\tau \wedge \tau_n} \left(  \| f_n(r) \|_{H^1(B(x,T-r))}^2    +  \|g_n(r) \cdot\|_{\mathscr{L}_2(\rkhs,H^1(B(x,T-r)))}^2  \right)  \, dr. \nonumber
\end{align}
Since the $g_n$ here is same as in Proposition \ref{prop-weakConv}, the computation of \eqref{Yt}-\eqref{t6} fits here too and we have
\begin{align}\label{lem-stocAprioriEnergyEst-t2}
\mathbb{E} \left[ \sup_{0 \leq s \leq t \wedge \tau_n} \mathbf{e}(s,T;x,Z_n(s)) \right]  &  \lesssim_T  \mathbb{E} \left[ \int_{0}^{t \wedge \tau_n}   \mathbf{e}(r,T;x,Z_n(r)) \, dr \right]  \\
& \quad  + \mathbb{E}\left[\int_{0}^{t \wedge \tau_n} \| f_n(r) \|_{H^1(B(x,T-r))}^2 \, dr \right] + \varepsilon_n (1+ N^2). \nonumber
\end{align}
Invoking Lemmata \ref{hsop} and \ref{lem-Y-LDP} implies, to save space we write $B_{T-r}$ instead of $B(x,T-r)$,
\begin{align}
& \|f_n(r)\|_{H^1(B_{T-r})}^2   \lesssim \|A_{U_n(r)}(\partial_x U_n(r), \partial_x U_n(r)) \|_{H^1(B_{T-r})}^2 + \| A_{U_n(r)}(V_n(r), V_n(r))  \|_{H^1(B_{T-r})}^2  \nonumber\\
& \qquad  +  \| Y(U_n(r))\dot{h}_n(r)  \|_{H^1(B_{T-r})}^2 \nonumber\\
& \quad \lesssim_{T,x} \left( 1 + \| U_n(r)\|_{H^1(B_{T-r})}^2 \right) \left[ 1+  \| \partial_x U_n(r) \|_{H^1(B_{T-r})}^2 + \| V_n(r) \|_{H^1(B_{T-r})}^2  +  \| \dot{h}_n(r)\|_{\rkhs}^2  \right] \nonumber\\
& \quad \lesssim  \left( 1 + \| Z_n(r)\|_{\mathcal{H}_{T-r}}^2 \right) \left[ 1+  \| Z_n(r) \|_{\mathcal{H}_{T-r}}^2 +  \| \dot{h}_n(r)\|_{\rkhs}^2  \right]. \nonumber
\end{align}
So from \eqref{lem-stocAprioriEnergyEst-t2} and the definition \eqref{defn-taun} we get
\begin{align}
\sup_{x \in [-a,a]}\mathbb{E} \left[ \sup_{0 \leq s \leq t \wedge \tau_n} \mathbf{e}(s,T;x,Z_n(s)) \right]  &  \lesssim_{T,a} N^2 \mathbb{E} \left[ t \wedge \tau_n \right]    + \varepsilon_n (1+ N^2) \nonumber\\
& \quad + (1+ N^2) \mathbb{E}\left[\int_{0}^{t \wedge \tau_n}   \left( 1 + N^2+ \dot{h}_n(r)\|_{\rkhs}^2  \right)  \, dr \right] \nonumber\\
& \lesssim_T  N^2 T + (1+N^2)T + \mathcal{M} +  \varepsilon_n (1+ N^2). \nonumber
\end{align}
Since $\lim\limits_{n \to \infty}\varepsilon_n  =0$, taking $\limsup_{n \to\infty}$ on both the sides we get the required bound, and hence, the Lemma \ref{lem-stocAprioriEnergyEst}.
\end{proof}

\begin{Lemma}\label{lem-ProbConv}
	The sequence of $\mathcal{X}_{T}$-valued process $\{ \mathcal{Z}_n \}_{n \in \mathbb{N}}$ converges in probability to $0$.
\end{Lemma}
\begin{proof}[\textbf{Proof of Lemma \ref{lem-ProbConv}}]
	We aim to show that for every  $x \in \mathbb{R}$ and $ R,\delta, \alpha >0$ there exists a natural number $n_0$ such that
	\begin{equation}\label{lem-ProbConv-t0}
	\mathbb{P} \left[ \sup_{t \in [0,T]} \| \mathcal{Z}_n(t)\|_{\mathcal{H}_{B(x,R)}} > \delta   \right]  <\alpha \quad \textrm{ for all } \quad n \geq n_0.
	\end{equation}
	Let us choose and fix $x \in \mathbb{R}, \delta>0,\alpha>0$. In first step, we prove  \eqref{lem-ProbConv-t0} for the case when $R$ is set to be $T$. Let us also set $\mathcal{T} = 2T$.
	%
	%
	Then,  since $\| \cdot\|_{\mathcal{H}_{B(x,r)}}$ is increasing in $r$ and for $t \in [0,T]$ we have $\mathcal{T}- t \geq T=R$, and
	%
	%
	\begin{equation}\label{lem-ProbConv-t0c}
	\mathbb{P}\left[   \sup_{t \in [0, T]} \| \mathcal{Z}_n(t)\|_{\mathcal{H}_{B(x,R)}} > \delta   \right] \leq \mathbb{P}\left[  \sup_{t \in [0,T]} \| \mathcal{Z}_n(t)\|_{\mathcal{H}_{B(x,\mathcal{T}-t)}} > \delta   \right].
	\end{equation}
	Further note that, since $0 \leq \mathbf{e}(t,\mathcal{T};x,\mathcal{Z}_n(t,\omega)) = \frac{1}{2} \| \mathcal{Z}_n(t,\omega)\|_{\mathcal{H}_{B(x,\mathcal{T}-t)}}^2$, due to \eqref{lem-ProbConv-t0c} instead of showing \eqref{lem-ProbConv-t0}, in the setting $R=T$,  it is enough to show that there exists $n_0 \in \mathbb{N}$ such that
	\begin{equation}\label{lem-ProbConv-t0d}
	\mathbb{P} \left[ \sup_{t \in [0,T]} \mathbf{e}(t,\mathcal{T};x,\mathcal{Z}_n(t,\omega)) >  \delta^2/2   \right]  <\alpha \quad \textrm{ for all } \quad n \geq n_0.
	\end{equation}
	But, since $x$ is fix in the argument now, there exists $a >0$ such that $x \in [-a,a]$ and the following holds
	$$\mathbb{P} \left[  \sup_{t \in [0,T]} \mathbf{e}(t,\mathcal{T};x,\mathcal{Z}_n(t)) >  \delta^2/2   \right] \leq \sup_{x \in [-a,a]}  \mathbb{P} \left[  \sup_{t \in [0,T]} \mathbf{e}(t,\mathcal{T};x,\mathcal{Z}_n(t)) >  \delta^2/2   \right].$$
	Consequently instead of \eqref{lem-ProbConv-t0d} it is sufficient to show that the existence of $n_0 \in \mathbb{N}$ such that
	\begin{equation}\label{lem-ProbConv-t0e}
	\sup_{x \in [-a,a]}  \mathbb{P} \left[ \sup_{t \in [0,T]} \mathbf{e}(t,\mathcal{T};x,\mathcal{Z}_n(t,\omega)) >  \delta^2/2   \right]  <\alpha \quad \textrm{ for all } \quad n \geq n_0.
	\end{equation}
	To prove \eqref{lem-ProbConv-t0e}, let us define a sequence $\{\kappa_n\}_{n \in \mathbb{N}}$ of stopping time via replacing $T$ by $\mathcal{T}$ in \eqref{defn-taun}.
	Now choose  $N > \|(u_0,v_0)\|_{\mathcal{H}_{a+T}}$  and $n_0 \in \mathbb{N}$ such that, based on Lemma \ref{lem-stocAprioriEnergyEst} for $\mathcal{T}$ instead of $T$,
	\begin{equation}\label{lem-ProbConv-t1}
	\frac{2}{N^2} \sup_{n \in \mathbb{N}} \sup_{x \in [-a,a]}  \mathbb{E} \left[ \sup_{t \in [0,T]} \mathbf{e}(t\wedge \kappa_n,\mathcal{T};x, Z_n(t\wedge \kappa_n)) \right] < \frac{\alpha}{2} \textrm{ for all } n \geq n_0,
	\end{equation}
	and, due to Proposition \ref{prop-weakConv} for $\mathcal{T}$ instead of $T$,
	\begin{equation}\label{lem-ProbConv-t2}
	\sup_{x \in [-a,a]}  \mathbb{E} \left[ \sup_{t \in [0,T]} \mathbf{e}(t\wedge \kappa_n,\mathcal{T};x,\mathcal{Z}_n(t\wedge \kappa_n)) \right] < \frac{\delta^2 \alpha}{4} \textrm{ for all } n \geq n_0.
	\end{equation}
	Thus the Markov inequality followed by using of \eqref{lem-ProbConv-t1} and \eqref{lem-ProbConv-t2}, for $n \geq n_0$, gives
	\begin{align}\label{lem-ProbConv-t3}
	\sup_{x \in [-a,a]}  \mathbb{P} & \left[ \sup_{t \in [0,T]} \mathbf{e}(t,\mathcal{T};x,\mathcal{Z}_n(t)) > \delta^2/2 \right] \nonumber\\
	& = \sup_{x \in [-a,a]}  \mathbb{P} \left[  \sup_{t \in [0,T]} \mathbf{e}(t,\mathcal{T};x,\mathcal{Z}_n(t)) > \delta^2/2  \textrm{ and }  \kappa_n = \mathcal{T}  \right]  \nonumber\\
	& \quad  + \sup_{x \in [-a,a]}  \mathbb{P}\left[ \sup_{t \in [0,T]} \mathbf{e}(t,\mathcal{T};x,\mathcal{Z}_n(t)) > \delta^2/2  \textrm{ and }  \mathbf{e}(t,\mathcal{T};x, Z_n(t)) \geq \frac{N^2}{2}  \right] \nonumber\\
	& \leq \frac{2}{\delta^2} \sup_{x \in [-a,a]}  \mathbb{E} \left[  \sup_{t \in [0,T]} \mathbf{e}(t,\mathcal{T};x,\mathcal{Z}_n(t)) \right] \nonumber\\
	& \quad + \frac{2}{N^2} \sup_{x \in [-a,a]}  \mathbb{E} \left[ \sup_{t \in [0,T]} \mathbf{e}(t,\mathcal{T};x, Z_n(t)) \right]  < \alpha.
	\end{align}
	Now we move to prove \eqref{lem-ProbConv-t0} when $R$ is not set to $T$. Since the closure of $B(x,R)$ is compact and $B(x,R) \subset \cup_{y \in B(x,R)} B(y,T) $, we can find finitely many centre $\{x_i\}_{i=1}^m$ such that $B(x,R) \subset \cup_{i=1}^m B(x_i,T).$ Moreover, since $B(x,R)$ is bounded, there exists $a >0$ such that $B(x,R) \in [-a,a]$. In particular, $x_i \in [-a,a]$ for all $i=1,\ldots,m$. Then since $\| \mathcal{Z}_n(t,\omega)\|_{\mathcal{H}_{B(x,R)}} \leq \sum_{i=1}^{m}\| \mathcal{Z}_n(t,\omega)\|_{\mathcal{H}_{B(x_i,T)}}$, we have
	\begin{align}
	& \sup_{x \in [-a,a]}  \mathbb{P}  \left[ \sup_{t \in [0,T]} \| \mathcal{Z}_n(t)\|_{\mathcal{H}_{B(x,R)}} > \delta \right] \leq  \sup_{x \in [-a,a]}  \mathbb{P} \left[  \sup_{t \in [0,T]} \sum_{i=1}^{m}\| \mathcal{Z}_n(t)\|_{\mathcal{H}_{B(x_i,T)}} > \delta \right] \nonumber\\
	& \leq \sum_{i=1}^{m} \sup_{x \in [-a,a]}  \mathbb{P} \left[ \sup_{t \in [0,T]} \| \mathcal{Z}_n(t)\|_{\mathcal{H}_{B(x,T)}} > \delta \right] \leq m \sup_{x \in [-a,a]}  \mathbb{P} \left[   \sup_{t \in [0,T]} \mathbf{e}(t,\mathcal{T};x,\mathcal{Z}_n(t))  > \delta^2/2 \right]. \nonumber
	\end{align}
	Now by taking $\alpha$ as $\alpha/m$ in \eqref{lem-ProbConv-t3}, of course with new $a$, we get that there exists an $n_0 \in\mathbb{N}$ such that, for all $n \geq n_0$,
	\begin{align}
	\sup_{x \in [-a,a]}  \mathbb{P} & \left[ \omega \in \Omega: \sup_{t \in [0,T]} \| \mathcal{Z}_n(t,\omega)\|_{\mathcal{H}_{B(x,R)}} > \delta \right]  < \alpha. \nonumber
	\end{align}
	Hence the Lemma \ref{lem-ProbConv}.
\end{proof}

Now we come back to the proof of \textbf{Statement 2}. Recall that $S_{\mathcal{M}}$ is a separable metric space. Since, by the assumptions, the sequence $\{ \mathscr{L}(h_n) \}_{n \in\mathbb{N}}$ of laws on $S_{\mathcal{M}}$ converges weakly to the law $\mathscr{L}(h)$, the Skorokhod representation theorem, see for example  \cite[Theorem 3.30]{Kallenberg_1997B}, there exists a probability space $(\tilde{\Omega}, \tilde{\mathscr{F}}, \tilde{\mathbb{P}})$, and on this probability space, one can  construct processes $(\tilde{h}_n, \tilde{h}, \tilde{W})$ such that the joint distribution of $(\tilde{h}_n, \tilde{W})$ is same as that of $(h_n,W)$, the distribution of $\tilde{h}$ coincide with that of $h$, and $\tilde{h}_n \xrightarrow[n \to \infty]{} \tilde{h}$, $\tilde{\mathbb{P}}$-a.s. pointwise on $\tilde{\Omega}$, in the weak topology of $S_M$. By Proposition \ref{prop-strngConv} this implies that
\begin{equation}\nonumber
J^0 \circ \tilde{h}_n \to J^0\circ \tilde{h} \quad \textrm{ in } \quad \mathcal{X}_T \;\; \tilde{\mathbb{P}}\textrm{-a.s. pointwise on } \tilde{\Omega}.
\end{equation}
Next, we claim that
\begin{equation}\nonumber
	\mathscr{L}(z_n) = \mathscr{L}(\tilde{z}_n), \quad \textrm{ for all } n
\end{equation}
where
\begin{equation}\nonumber
z_n := J^0 \circ h : \Omega \to \mathcal{X}_T \quad \textrm{ and } \quad \tilde{z}_n := J^0\circ \tilde{h}_n : \tilde{\Omega} \to \mathcal{X}_T.
\end{equation}
To avoid complexity, we will write $J^0(h)$ for $J^0 \circ h$. Let $B$ be an arbitrary Borel subset of $\mathcal{X}_T$. Thus, since from Proposition \ref{prop-strngConv} $J^0: S_{\mathcal{M}} \to  \mathcal{X}_T$ is Borel, $(J^0)^{-1}(B)$ is Borel in $S_{\mathcal{M}}$. So we have
 \begin{align}
 	& \mathscr{L}(z_n)(B) = \mathbb{P} \left[  J^0(h_n)(\omega) \in B\right]  = \mathbb{P} \left[ h_n^{-1} \left( (J^0)^{-1} (B)\right) \right]  = \mathscr{L}(h_n) \left( (J^0)^{-1} (B) \right). \nonumber
 \end{align}
But, since $\mathscr{L}(h_n) = \mathscr{L}(\tilde{h}_n)$ on $\mathcal{X}_T$, this implies $\mathscr{L}(z_n)(B) = \mathscr{L}(\tilde{z}_n)(B)$.
Hence the claim and by a similar argument we also have $\mathscr{L}(z_h) = \mathscr{L}(z_{\tilde{h}})$.

Before moving forward, note that from Lemma \ref{lem-ProbConv}, the sequence of $\mathcal{X}_T$-valued random variables, defined from $\Omega$, $J^{\varepsilon_n}(h_n) - J^0(h_n)$ converges in measure $\mathbb{P}$ to $0$.  Consequently, because $\mathscr{L}(h_n) = \mathscr{L}(\tilde{h}_n)$ and $J^{\varepsilon_n} - J^0$ is measurable, we infer that  $J^{\varepsilon_n}(\tilde{h}_n) - J^0(\tilde{h}_n) \xrightarrow[]{\tilde{\mathbb{P}}} 0$ as $n \to \infty$.
Hence, we can choose a subsequence $\{ J^{\varepsilon_n}(\tilde{h}_n) - J^0(\tilde{h}_n) \}_{n \in \mathbb{N}}$, indexed again by $n$,  of $\mathcal{X}_T$-valued random variables converges to $0$, $\tilde{\mathbb{P}}$-almost surely.

Now we claim to have the proof of \textbf{Statement 2}. Indeed, for any globally Lipschitz continuous and bounded function $\psi : \mathcal{X}_T \to \mathbb{R}$, see \cite[Theorem 11.3.3]{Dudley_1989B}, we have
\begin{align}
	& \left\vert  \int_{\mathcal{X}_T} \psi(x) \, d\mathscr{L}(J^{\varepsilon_n}(h_n)) -  \int_{\mathcal{X}_T} \psi(x) \, d\mathscr{L}(J^0(h)) \right\vert \nonumber\\
	&  = \left\vert  \int_{\mathcal{X}_T} \psi(x) \, d\mathscr{L}(J^{\varepsilon_n}(\tilde{h}_n)) -  \int_{\mathcal{X}_T} \psi(x) \, d\mathscr{L}(J^0(\tilde{h})) \right\vert\nonumber\\
	& \leq  \left\vert  \int_{\tilde{\Omega}} \left\{  \psi \left(  J^{\varepsilon_n}(\tilde{h}_n) \right) -   \psi \left(  J^0(\tilde{h}_n) \right)\right\} \, d\tilde{\mathbb{P}}  \right\vert \nonumber\\
	& \quad + \left\vert  \int_{\tilde{\Omega}} \psi \left(  J^0(\tilde{h}_n) \right)\, d\tilde{\mathbb{P}} -  \int_{\tilde{\Omega}} \psi \left(  J^0(\tilde{h}) \right)\, d\tilde{\mathbb{P}}  \right\vert. \nonumber
\end{align}
Since $J^0(\tilde{h}_n) \xrightarrow[n \to \infty]{} J^0(\tilde{h})$, $\mathbb{P}$-a.s. and $\psi$ is bounded and continuous, we deduce that the 2nd term in right hand side above converges to $0$ as $n \to \infty$. Moreover we claim that the 1st term also goes to $0$. Indeed, it follows from the dominated convergence theorem because the term is bounded by
\begin{equation}\nonumber
	L_{\psi} \int_{\tilde{\Omega}}  |J^{\varepsilon_n}(\tilde{h}_n) -   J^0(\tilde{h}_n)| \, d \tilde{\mathbb{P}},
\end{equation}
where $L_{\psi}$ is Lipschitz constant of $\psi$,  and the sequence $\{ J^{\varepsilon_n}(\tilde{h}_n) -   J^0(\tilde{h}_n) \}_{n \in \mathbb{N}}$ converges to $0$, $\tilde{\mathbb{P}}$-a.s.

Therefore, \textbf{Statement 2} holds true and we complete the proof of Theorem \ref{thm-LDP}.

\appendix
\section{Intrinsic and Extrinsic Formulation}\label{sec:IntAndExtSoln}
Here we recall the intrinsic and extrinsic formulation of SGWE  from  \cite{Brz+Ondr_2007} and state, without proof, the equivalence result between them.
Consider the following SGWE Cauchy problem
\begin{equation}\label{eqn_initial}
\left\{\begin{aligned}
& \mathbf D_t\partial_tu =\mathbf D_x\partial_xu+Y_u(\partial_tu,\partial_xu)\,\dot W,
\\
& u(0,\cdot)=u_0,\\
& \partial_tu(t,\cdot)_{|t=0} =v_0
\end{aligned}
\right.\end{equation}
Assume that  $u_0$, $v_0$ are $\mathfrak F_0$-measurable random variables with values in $H^2_{\textrm{loc}}(\mathbb R,M)$ and $H^1_{\textrm{loc}}(\mathbb R,TM)$ respectively such that $u_0(x,\omega)\in M$ and
$v_0(x,\omega)\in T_{u_0(x,\omega)}M$ hold for every $\omega\in\Omega$ and $x\in\mathbb R$.
\begin{Definition}\cite[Definition 2.3]{Brz+Ondr_2007}\label{def_intrinsic_soln}
	A  process $u:\mathbb  R_+\times\mathbb  R\times\Omega\to M$ is called an {\rm intrinsic solution} of  problem \eqref{eqn_initial} provided the following six conditions are satisfied:
	\begin{trivlist}
		\item[(i)] $u(t,x,\cdot)$ is $\mathfrak F_t$-measurable for every $x\in\mathbb R$ and every $t\ge 0$,
		\item[(ii)] $u(\cdot,\cdot,\omega)$ belongs to $\mathcal{C}^1(\mathbb  R_+\times\mathbb  R,M)$ for every $\omega\in\Omega$,
		\item[(iii)] $\mathbb{R}^+\ni t \mapsto u(t,\cdot,\omega)\in H^2_{\textrm{loc}}(\mathbb R,M)$ is continuous for every $\omega\in\Omega$,
		\item[(iv)] $\mathbb{R}^+\ni t\mapsto\partial_tu(t,\cdot,\omega)\in H^1_{\textrm{loc}}(\mathbb R;TM)$ is continuous for every $\omega\in\Omega$,
		\item[(v)] $u(0,x,\omega)=u_0(x,\omega)$ and $\partial_tu(0,x,\omega)=v_0(x,\omega)$ holds for every $x\in\mathbb R$ almost surely,
		\item[(vi)]   and for every vector field $X$ on $M$, and every $t\ge 0$ and $R>0$
	\end{trivlist}
	\begin{align}
	\langle\partial_tu(t),X(u(t))\rangle_{T_{u(t)}M}&= \langle v_0,X(u_0)\rangle_{T_{u(t)}M}+\int_0^t\langle\mathbf D_x\partial_xu(s),X(u(s))\rangle_{T_{u(s)}M}\,ds
	\nonumber\\
	&\quad + \int_0^t\langle\partial_tu(s),\nabla_{\partial_tu(s)}X\rangle_{T_{u(s)}M}\,ds \nonumber
	\\
	&\quad + \int_0^t\langle X(u(s)),Y_{u(s)}(\partial_tu(s),\partial_xu(s))\,dW(s)\rangle_{T_{u(s)}M},
	\nonumber
	\end{align}
	holds in $L^2(-R,R)$ almost surely.
\end{Definition}

\begin{Definition}\cite[Definition 2.6]{Brz+Ondr_2007}\label{def_extrinsic_soln}
	A  process $u:\mathbb  R_+\times\mathbb  R\times\Omega\to M$ is called an {\rm extrinsic solution} of problem \eqref{eqn_initial} provided the following six conditions are satisfied.
	\begin{trivlist}
		\item[(a)] $u(t,x,\cdot)$ is $\mathfrak F_t$-measurable for every $t\ge 0$ and $x\in\mathbb  R$,
		\item[(b)] $\mathbb{R}^+\ni t \mapsto u(t,\cdot,\omega)\in H^2_{\textrm{loc}}(\mathbb R;\mathbb R^n)$ is continuous for every $\omega\in\Omega$,
		\item[(c)] $\mathbb{R}^+\ni t\mapsto u(t,\cdot,\omega)\in H^1_{\textrm{loc}}(\mathbb R;\mathbb R^n)$ is continuously
		differentiable for every $\omega\in\Omega$,
		\item[(d)] $u(t,x,\omega)\in M$ for every $x\in\mathbb R$ and every $\omega\in\Omega$,
		\item[(e)] $u(0,x,\omega)=u_0(x,\omega)$ and $\partial_tu(0,x,\omega)=v_0(x,\omega)$ holds for every $x\in\mathbb R$ almost surely,
		\item[(f)]   and  for every $t\ge 0$
		and $R>0$
		\begin{align}
		\partial_tu(t)& = v_0 +\int_0^t\left[\partial_{xx}u(s)-A_{u(s)}(\partial_xu(s),\partial_xu(s))+A_{u(s)}(\partial_tu(s),\partial_tu(s))\right]\,ds
		\nonumber\\
		&\qquad + \int_0^rY_{u(s)}(\partial_tu(s),\partial_xu(s))\,dW(s), \nonumber
		\end{align}
		holds in $L^2((-R,R); \mathbb R^n)$ almost surely.
	\end{trivlist}
\end{Definition}

The next result state the equivalence between the intrinsic solution and extrinsic solution to the problem  \eqref{eqn_initial}.
\begin{Theorem}\cite[Theorem 12.1]{Brz+Ondr_2007} \label{thm-equiv}
	Assume that  $u_0$, $v_0$ are $\mathfrak F_0$-measurable random variables with values in $H^2_{\textrm{loc}}(\mathbb R,M)$ and $H^1_{\textrm{loc}}(\mathbb
	R,TM)$ respectively such that $u_0(x,\omega)\in M$ and
	$v_0(x,\omega)\in T_{u_0(x,\omega)}M$ hold for every
	$\omega\in\Omega$ and $x\in\mathbb R$. Suppose also that $M$ is a compact submanifold of $\mathbb{R}^n$ as in Definition \ref{def_extrinsic_soln}. Then a
	process $u:\mathbb  R_+\times\mathbb  R\times\Omega\to M$ is  an {intrinsic solution} of problem \eqref{eqn_initial} if and only if it is an extrinsic solution of the same problem.
\end{Theorem}

\section{Existence and uniqueness result}\label{sec:existUniqResult}
In this part we recall a result about the existence of a uniqueness global solution, in strong sense, to problem  \eqref{eqn_initial}. We ask the reader to refer  \cite[Theorem 11.1]{Brz+Ondr_2007} for the proof.
%
%

%
\begin{Theorem}\label{thm-exist}
	Fix $T>0$ and $R > T$. For every $\mathfrak F_0$-measurable random variable $u_0,v_0$ with values in $H^2_{\textrm{loc}}(\mathbb R,M)$ and $H^1_{\textrm{loc}}(\mathbb R,TM)$, there exists a process $u: [0,T) \times \mathbb{R} \times \Omega \to M$, which we denote by $u = \{u(t), t < T \}$, such that the following hold:
	\begin{enumerate}
		\item  $u(t,x,\cdot): \Omega \to M$ is $\mathfrak{F}_t$-measurable for every $t <T$ and $x \in \mathbb{R}$,
		\item  $[0,T) \ni t \mapsto u(t,\cdot,\omega) \in H^2((-R,R);\mathbb{R}^n)$ is continuous for almost every $\omega \in \Omega$,
		\item  $[0,T) \ni t \mapsto u(t,\cdot,\omega) \in H^1((-R,R);\mathbb{R}^n)$ is continuously differentiable for almost every $\omega \in \Omega$,
		\item  $u(t,x,\omega) \in M$,    for every $t <T, x \in \mathbb{R}$, $\mathbb{P}$-almost surely,
		\item  $u(0,x,\omega)  = u_0(x,\omega)$ and $\partial_t u(0,x,\omega) = v_0(x,\omega)$ holds,   for every $x \in \mathbb{R}$, $\mathbb{P}$-almost surely,
		\item  for every $t \geq 0$ and $R>0$,
		\begin{align}
			\partial_t u(t) & = v_0 + \int_{0}^{t} \left[ \partial_{xx}u(s) - A_{u(s)}(\partial_x u(s), \partial_x u(s)) + A_{u(s)}(\partial_t u(s), \partial_t u(s))   \right] \, ds \nonumber\\
			& \quad + \int_{0}^{t}  Y_{u(s)}(\partial_t u(s), \partial_x u(s)) \, dW(s), \nonumber
		\end{align}
		holds in $L^2((-R,R);\mathbb{R}^n)$,   $\mathbb{P}$-almost surely.
	\end{enumerate}
	Moreover, if there exists another process $U = \{ U(t); t \geq 0\}$ satisfy the above properties,  then $U(t,x,\omega)=u(t,x,\omega)$ for every $|x|<R-t$ and $t \in [0,T)$, $\mathbb{P}$-almost surely.
\end{Theorem}

\section{Energy inequality for stochastic wave equation}\label{sec:EnergyIneqSWE}
Recall the following slightly modified version of \cite[Proposition 6.1]{Brz+Ondr_2007} for a one (spatial) dimensional linear inhomogeneous stochastic wave equation. For $l\in\mathbb  N$, we use the symbol $D^lh$ to denote the $\mathbb R^{n\times 1}$-vector $ \left(\frac{d^lh^1}{d x^l}, \frac{d^lh^2}{d x^l}, \cdots
,\frac{d^lh^n}{d x^l} \right). $
\begin{Proposition}\label{prop_magic}
	Assume that  $T>0$ and  $k\in \mathbb{N}$.  Let $W$ be a cylindrical Wiener process on a Hilbert space $K$. Let $f$ and $g$ be progressively measurable processes with values, respectively,  in $H_{\trm{loc}}^k(\mathbb R;\mathbb R^n)$ and $\mathscr L_2(K,H_{\trm{loc}}^k(\mathbb R;\mathbb R^n))$ such that, for every $R>0$,
	\begin{equation}\nonumber
	\int_0^T\left\{\Vert f(s)\Vert_{H^k((-R,R);\mathbb R^n)}+\Vert g(s)\Vert^2_{\mathscr L_2(K,H^k((-R,R);\mathbb	R^n))}\right\}\,ds < \infty,
	\end{equation}
	$\mathbb{P}$-almost surely. Let  $z_0$ be  an $\mathscr F_0$-measurable random variable with values in
	$$ \mathcal H_{\trm{loc}}^k : =H_{\trm{loc}}^{k+1}(\mathbb R;\mathbb R^n)\times H_{\trm{loc}}^k(\mathbb R;\mathbb R^n). $$
	Assume that  an $\mathcal{H}_{\trm{loc}}^k$-valued  process $z=z(t)$, $t\in [0,T]$, satisfies
	$$
	z(t)=S_tz_0+\int_0^tS_{t-s}\left(\begin{array}{c}0\\f(s)\end{array}\right)\,ds+\int_0^tS_{t-s}\left(\begin{array}{c}0\\g(s)\end{array}\right)\,dW(s),\qquad 0\le t\le T.
	$$
	Given $x\in\mathbb R$,  we define  the energy function $e:[0,T]\times \mathcal H_{\trm{loc}}^k \to
	\mathbb{R}^+$ by, for $z=(u,v)\in\mathcal H_{\trm{loc}}^k$,
	\begin{equation}\nonumber
	\mathbf e(t,T;x,z) = \frac{1}{2} \left\{\Vert u\Vert^2_{L^2(B(x,T-t))}+\sum_{l=0}^k\left[\Vert D^{l+1} u\Vert^2_{L^2(B(x,T-t))} + \Vert D^lv \Vert^2_{L^2(B(x,T-t))}\right] \right\}.
	\end{equation}
	Assume that $L:[0,\infty)\to\mathbb R$ is a non-decreasing $\mathcal{C}^2$-smooth function and define the second energy function $E: [0,T]\times\mathcal H_{\trm{loc}}^k \to \mathbb{R}$, by
	$$
	\mathbf E(t,z)=L(\mathbf e(t,T;x,z)), \; z=(u,v)\in\mathcal H_{\trm{loc}}^k.
	$$
	Let  $\{e_j\}$ be  an orthonormal basis of  $K$. We define a function $V: [0,T]\times\mathcal H_{\trm{loc}}^k \to \mathbb{R}$, by
	\begin{eqnarray*}
		V(t,z)&=&L^\prime(\mathbf e(t,T;x,z))\left[\langle
		u,v\rangle_{L^2(B(x,T-t))}+\sum_{l=0}^k\langle
		D^lv,D^lf(t)\rangle_{L^2(B(x,T-t))} \right] \\
		&\quad +&\frac{1}{2}L^\prime(\mathbf
		e(t,T;x,z))\sum_{j}\sum_{l=0}^k\vert D^l[g(t)e_j]\vert^2_{L^2(B(x,T-t))}+
		\\
		&\quad +&\frac{1}{2}L^{\prime\prime}(\mathbf
		e(t,T;x,z))\sum_{j}\left[\sum_{l=0}^k\langle
		D^lv,D^l[g(t)e_j]\rangle_{L^2(B(x,T-t))}\right]^2, \; (t,z)\in
		[0,T]\times \mathcal H_{\trm{loc}}^k,
	\end{eqnarray*}
	where we suppress the dependency of the left hand side on $T$ and $x$. Then $\mathbf E$ is continuous on $[0,T]\times\mathcal H_{\trm{loc}}^k$, and for every $0\le t\le T$,
	\begin{eqnarray}\nonumber
	\mathbf E(t,z(t))&\le &\mathbf
	E(0,z_0)+\int_0^tV(r,z(r)\,dr\\
	&\quad +&\sum_{l=0}^k\int_0^tL^\prime(\mathbf
	e(r,z(r)))\langle D^lv(r),D^l[g(r)\,dW(r)]\rangle_{L^2(B(x,T-r))}, \quad \mathbb{P}\trm{-a.s.}.
	\nonumber\end{eqnarray}
\end{Proposition}

\textbf{Acknowledgements:} The last author wishes to thank the York Graduate Research School, to award the Overseas scholarship (ORS), and the Department of Mathematics, University of York, to provide financial support and excellent research facilities during the period of this work. The results of this paper are part of his  Ph.D. thesis. He also presented a lecture on the topic of this paper at the Workshop on Stochastic Partial Differential Equations, held at the University of Sydney, Australia, in August 2019.

\end{document}